\documentclass[12pt]{amsart}
\usepackage{amssymb}
\usepackage{amsmath}
\numberwithin{equation}{section}
\numberwithin{section}{part}
\usepackage{color}
\usepackage{graphicx}
\usepackage{epstopdf}

\newtheorem{thm}{Theorem}[section]
\newtheorem{lemma}[thm]{Lemma}
\newtheorem{cor}[thm]{Corollary}
\newtheorem{prop}[thm]{Proposition}
\newtheorem{con}[thm]{Conjecture}

\theoremstyle{definition}

\newtheorem{defn}[thm]{Definition}

\theoremstyle{remark}
\newtheorem{remark}[thm]{Remark}

\newcommand{\ol}[1]{\overline{#1}}

\newcommand{\sgn}{{\rm sgn}}

\renewcommand{\S}{\mathfrak S}
\newcommand{\s}{\sigma}

\newcommand\C{{\mathbb{C}}}

\newcommand\Q{{\mathbb Q}}
\newcommand\Z{{\mathbb{Z}}}
\newcommand\F{{\mathbb{F}}}
\newcommand\PP{{\mathbb{P}}}
\newcommand\N{{\mathbb{N}}}

\newcommand\wh{\widehat}

\newcommand\bars{{\rm bar}}

\newcommand\bq{\begin{equation}}
\newcommand\eq{\end{equation}}
\newcommand\beq{\begin{eqnarray*}}
\newcommand\eeq{\end{eqnarray*}}
\newcommand\ben{\begin{enumerate}}
\newcommand\een{\end{enumerate}}
\newcommand\bit{\begin{itemize}}
\newcommand\eit{\end{itemize}}

\newcommand\des{{\rm des}}
\newcommand\exc{{\rm exc}}
\newcommand\inv{{\rm inv}}
\newcommand\maj{{\rm maj}}
\newcommand\comaj{{\rm comaj}}
\newcommand\ai{{\rm ai}}
\newcommand\aid{{\rm aid}}
\newcommand\sg{{\mathfrak S}}
\newcommand\Des{{\rm DES}}
\newcommand\Exd{{\rm DEX}}
\newcommand\Exc{{\rm EXC}}
\newcommand\fix{{\rm fix}}
\newcommand\ch{{\rm ch}}

\newcommand\x{{\mathbf x}}
\newcommand\bbar{{\rm bar}}
\newcommand\Exp{{\rm Exp}}
\newcommand\wt{{\rm wt}}
\newcommand\Cov{{\rm Cov}}
\newcommand\Com{{\rm Com}}
\newcommand\Par{{\rm Par}}
\newcommand\bnd{{\rm bnd}}

\def\wh{\widehat}
\def\hz{\hat 0}
\def\ho{\hat 1}

\def\ov{\overline}
\def\ttn{T_{t,n}}

\def\rh{\tilde{H}}

\def\zz{{\mathbb Z}}

\def\ff{{\mathbb F}}

\def\xx{{\mathbf x}}

\begin{document}

\title[Eulerian quasisymmetric functions]
{Eulerian quasisymmetric functions and poset topology}
\author[Shareshian]{John Shareshian$^1$}
\address{Department of Mathematics, Washington University, St. Louis, MO 63130}
\thanks{$^{1}$Supported in part by NSF Grants
 DMS 0300483 and DMS 0604233, and the Mittag-Leffler Institute}
\email{shareshi@math.wustl.edu}

\author[Wachs]{Michelle L. Wachs$^2$}
\address{Department of Mathematics, University of Miami, Coral Gables, FL 33124}
\email{wachs@math.miami.edu}
\thanks{$^{2}$Supported in part by NSF Grants
DMS 0302310 and DMS 0604562, and the Mittag-Leffler Institute}


\date{May 15, 2008}

\dedicatory{}

\begin{abstract} We introduce  a   family of  quasisymmetric functions called {\em Eulerian quasisymmetric functions},
which have the property of specializing to   enumerators for the joint distribution of the permutation statistics, major index and excedance number on permutations of fixed cycle type.
 This family is analogous to a family of quasisymmetric functions that Gessel and Reutenauer used to study the joint distribution of major index and descent number on permutations of fixed cycle type.
 Our central result is a formula for
the generating function for the Eulerian quasisymmetric  functions, which specializes to
a new  and surprising
$q$-analog of a classical formula for the exponential generating
function of the Eulerian polynomials.  This $q$-analog computes the joint distribution of  excedance number and major index, the only of the four important Euler-Mahonian distributions that had not yet been computed.   Our  study of the Eulerian quasisymmetric functions   also yields  results  that include the descent statistic and  refine   results of Gessel and Reutenauer.   We also obtain $q$-analogs,  $(q,p)$-analogs and quasisymmetric function analogs of classical results on the symmetry and unimodality of the Eulerian polynomials.  Our Eulerian quasisymmetric functions refine  symmetric functions that have occurred in various representation theoretic and enumerative contexts such as in MacMahon's study of multiset derangements,  in work of  Procesi and Stanley   on  toric varieties of  Coxeter complexes and in Stanley's work on symmetric chromatic polynomials.  Here we present yet another occurence in connection with the homology of a poset introduced by Bj\"orner and Welker.
 \end{abstract}

\maketitle

\vbox{
\tableofcontents
}
\section{Introduction}
Through our study of the homology of a certain partially ordered set introduced by
Bj\"orner and Welker, we  have discovered a
remarkable  $q$-analog of a classical formula for the Eulerian polynomials.    The
 Eulerian polynomials enumerate permutations according to their number of
 descents or their number of excedances.  Our q-Eulerian polynomials are
 the enumerators for the joint distribution of the excedance statistic and
 the major index.  There is a vast literature on q-Eulerian polynomials
 that involves other combinations of Eulerian and Mahonian permutation
 statistics, but ours is the first result to address the combination of
 excedance number and major index.  Although poset topology led us to
 conjecture our formula, it is symmetric function theory that provides  the proof  of our original formula, as well as  of more refined versions involving additional permutation statistics.
 
Part 1 of this paper deals only with our permutation statistic and symmetric function theoretic results. In Part 2, results of Part 1 are used to compute the homology of  a q-analog of the  Bj\"orner and Welker poset and  the character of the symmetric group acting on the homology of the Bj\"orner-Welker poset.

  \subsection{Permutation Statistics}
 The modern study of permutation statistics began with the work of
Major Percy McMahon \cite[Vol. I, pp.~135, 186; Vol. II, p.~viii] {mac1}, \cite{mac2}.
It deals with the enumeration of  permutations
according to   natural statistics.    A permutation statistic is
simply a function from the union of all
the symmetric groups $\S_n$ to the set of
nonnegative integers.     MacMahon studied four fundamental
permutation statistics,  the inversion index ($\inv$), the major index ($\maj$), the
descent number ($\des$), and the excedance number ($\exc$), which are defined in Section~\ref{statsec}.

MacMahon observed in \cite[Vol. I, p. 186]{mac1} the now well
known result that the statistics $\des$ and $\exc$ are
equidistributed, that is,
\begin{equation*} \label{euler}
A_n(t):=\sum_{\sigma \in \sg_n}t^{\des(\sigma)}=\sum_{\sigma \in \sg_n}t^{\exc(\sigma)}.
\end{equation*}
The coefficients of the polynomials $A_n(t)$ were studied by
Euler, and are called {\it Eulerian numbers}.  The polynomials are
known as {\it Eulerian polynomials}.  (Note that it is common in
the literature to define the Eulerian polynomials to be
$tA_n(t)$.) Any permutation statistic that is equidistributed with
$\des$ and $\exc$ is called an {\it Eulerian statistic}.  Eulerian
numbers and polynomials have been extensively studied (see for
example \cite{fs} and \cite{kn}).  Euler proved (see \cite[p.
39]{kn}) the generating function formula
\begin{equation}  \label{expgen}
1+\sum_{n \geq 1}A_n(t)\frac{z^n}{n!}=\frac{1-t}{e^{z(t-1)}-t}.
\end{equation}

For a positive integer $n$, the polynomials $[n]_q$ and $[n]_q!$
are defined as
\[
[n]_q:=1+q+\ldots+q^{n-1}
\]
and
\[
[n]_q!:=\prod_{j=1}^{n}[j]_q.
\]
MacMahon proved in \cite{mac2} the first equality in the equation
\begin{equation*} \label{majeq}
\sum_{\sigma \in \S_n}q^{\maj(\sigma)}=[n]_q!=\sum_{\sigma \in
\S_n}q^{\inv(\sigma)}
\end{equation*}
after the second equality had been obtained in \cite{rod} by
Rodrigues.  A permutation statistic that is equidistributed with
$\maj$ and $\inv$ is called a {\it Mahonian} statistic.

Much effort has been put into the examination of joint
distributions of pairs of permutation statistics, one Eulerian and
one Mahonian (for a sample see,
\cite{bs,br,c,csz,foa3,fs2,fz,gar,gg,grem,gr,hag,rrw,ra,sk,st,st1,wa2}).
One beautiful result on such joint distributions is found in the
paper \cite{st1} of Stanley. For permutation statistics ${\mathsf
{f_1}},\ldots,{\mathsf {f_k}}$ and a positive integer $n$, define
the polynomial
\[
A_n^{{\mathsf {f_1,\ldots,f_k}}}(t_1,\ldots,t_k):=\sum_{\sigma \in
\S_n}t_1^{{\mathsf {f_1}}  (\sigma) } t_2^{{\mathsf {f_2}}
(\sigma)}\cdots t_k^{{\mathsf {f_k}}  (\sigma)}.
\]
Also, set
\[
A_0^{{\mathsf {f_1,\ldots,f_k}}}(t_1,\ldots,t_k):=1.
\]
Stanley showed that if we define
\[
\Exp_q(z):=\sum_{n \geq 0}q^{{n} \choose {2}}\frac{z^n}{[n]_q!}
\]
then we have the $q$-analogue
\begin{equation*} \label{staneq}
\sum_{n \geq 0}A_n^{\inv,\des}(q,t)\frac{z^n}{[n]_q!}=\frac{1-t}{\Exp_q(z(t-1))-t}
\end{equation*}
of (\ref{expgen}).
\\
\\
Although there has been much work on the joint distributions of $(\maj,\des)$, $(\inv,\exc)$ and $(\inv,\des)$ there are to our knowledge no results about
$A^{\maj,\exc}_n(q,t)$ in the existing literature prior to \cite{sw}, where our work was first  announced.   Our main result on
these polynomials is as follows.  Set
\[
\exp_q(z):=\sum_{n \geq 0}\frac{z^n}{[n]_q!}.
\]

\begin{thm} \label{expgenth}
We have
\begin{equation} \label{expgeneq}
\sum_{n \geq 0}A^{\maj,\exc}_n(q,t)\frac{z^n}{[n]_q!}=\frac{(1-tq)\exp_q(z)}{\exp_q(ztq)-tq\exp_q(z)}.
\end{equation}
\end{thm}

It is well-known that $\exp_q(z)$ is a specialization (called the stable principal specialization) of the symmetric function
$$ H(z)=H(\xx,z):=\sum_{n \geq 0}h_n(\xx)z^n,
$$
where $h_n$ is the complete homogeneous symmetric function of degree n (details are given in Section~\ref{statsec}).
Hence the right
hand side of (\ref{expgeneq}) is the stable principal specialization of the symmetric function
$$\frac{(1-tq)H(z)}{H(ztq)-tqH(z)}.$$

We introduce a family of  quasisymmetric functions $Q_{n,j,k}(\mathbf x)$, called  Eulerian quasisymmetric functions, whose
 generating function $$\sum_{n,j ,k\ge 0} Q_{n,j,k}(\x) t^j r^k z^n$$
specializes to
$$\sum_{n\ge 0} \sum_{\s \in \S_n} q^{\maj(\s) - \exc(\s)} t^{\exc(\s)}r^{\fix(\s)} {z^n \over [n]_q!}, $$
   where $\fix(\sigma) $ is the number of fixed points of  $\sigma$.
The $Q_{n,j,k}$ are defined as sums of fundamental quasisymmetric functions that we associate (in a nonstandard way) with permutations  $\s \in \mathfrak S_n$ satisfying  $\exc(\s) = j$ and $\fix(\s) = k$.   Our central result is the following theorem.

\begin{thm}\label{introsymgenth}
We have
\begin{eqnarray}
\label{introsymgenth1}\sum_{n,j,k \geq 0}Q_{n,j,k}(\xx)t^jr^kz^n&=&\frac{(1-t)H(rz)}{H(zt)-tH(z)}\\ \label{introsymgenth2}&=& \frac{H(rz)}{1-\sum_{n\ge 2}t[n-1]_t h_nz^n}.
\end{eqnarray}
\end{thm}

The proof of Theorem \ref{introsymgenth} appears in Section
\ref{identsec}.  It depends on combinatorial bijections and
involves nontrivial extension of techniques introduced by  Gessel
and Reutenauer in \cite{gr}, D\'esarm\'enien and Wachs in
\cite{dw} and Stembridge in \cite{stem1}.   Gessel and Reutenauer
construct a bijection from  multisets of  primitive circular words
of fixed content to permutations in order   to enumerate
permutations with a given  descent set and  cycle type. This
bijection, which is related to Stanley's theory of P-partitions
\cite{st},  has also proved useful in papers of D\'esarm\'enien
and Wachs \cite{dw2,dw},  Diaconis, McGrath and Pitman \cite{dmp},
Hanlon \cite{h}, and Wachs \cite{wa3}.   Reiner \cite{re}
introduced a type B analog. Here we introduce a bicolored version
of the Gessel-Reutenauer bijection.

By specializing  Theorem~\ref{introsymgenth} we get the following strengthening of Theorem~\ref{expgenth}.
\begin{cor} \label{expgenthfix}
We have
\begin{equation} \label{expgeneqfix}
\sum_{n \geq 0}A^{\maj,\exc,\fix}_n(q,t,r)\frac{z^n}{[n]_q!}=\frac{(1-tq)\exp_q(rz)}{\exp_q(ztq)-tq\exp_q(z)}
\end{equation}
\end{cor}

By setting $t=1$ in (\ref{expgeneqfix}) one  obtains a  formula of Gessel and Reutenauer \cite{gr}.
 By setting $r=0$, we obtain a new result on the $(\maj,\exc)$-enumerator of derangements.

We will show  that a different specialization (the nonstable principal specialization) of Theorem~ \ref{introsymgenth} readily yields a further extension of Theorem~\ref{expgenth} that Foata and Han obtained after seeing a preprint containing our work \cite{sw}.

  \begin{cor}[Foata and Han \cite{FH}] \label{foha}
We have
\[
\sum_{n \geq
0}A_n^{\maj,\des,\exc,\fix}(q,p,t,r)\frac{z^n}{(p;q)_{n+1}}=\sum_{m
\geq
0}p^m\frac{(1-qt)(z;q)_m(ztq;q)_m}{((z;q)_m-tq(ztq;q)_m)(zr;q)_{m+1}},
\]
where\[
(a;q)_n:= \left\{\begin{array}{ll} 1 & \mbox{if } n=0 \\
(1-a)(1-aq)\ldots(1-aq^{n-1}) & \mbox{if } n \geq 1.
\end{array}
\right.
\]
\end{cor}

Essential to our proof of Theorem~\ref{introsymgenth} is a refinement of   $Q_{n,j,k}$ which is  quite interesting in its own right.  Given a partition $\lambda $ of $n$, we define  $Q_{\lambda,j}$ to be a sum of fundamental quasisymmetric functions that we associate with permutations of cycle type $\lambda$ that have $j$ excedances.  We prove
that  $Q_{\lambda,j}$ is, in fact, a symmetric function.
It is well-known that the Eulerian polynomials are symmetric and unimodal.  We prove
$q$- and $(q,p)$-analogs of these results, as well as cycle type and  quasisymmetric function analogs.  These results appear in Section~\ref{secQ}.

 It follows from Theorem~\ref{introsymgenth} that the Eulerian quasisymmetric functions $Q_{n,j,k}$ are h-positive. In Section~\ref{secQ} we show that this does not hold for the more refined $Q_{\lambda,j}$, but  conjecture that they are Schur positive.   Of particular interest are the
 $Q_{\lambda,j}$ when the partition $\lambda$ consists of a single part.  We observe in Section~\ref{repthsec} that the $Q_{\lambda,j}$ for general $\lambda$ can be expressed via plethysm in terms of  these.    We also present results and conjectures on the virtual representation of the symmetric group whose Frobenius characteristic is $Q_{\lambda,j}$ when $\lambda$ consists of a single part.
 
 The  symmetric functions $Q_{\lambda,j}$  resemble  the symmetric functions $L_\lambda$ studied by
Gessel and Reutenauer in \cite{gr} in their work on quasisymmetric
functions and permutation statistics.  However our $Q_{\lambda,j}$ are not
refinements of the  $L_\lambda$. Indeed, $L_\lambda$ is the
Frobenius characteristic of a representation induced from a linear
character of the centralizer of a permutation of cycle type
$\lambda$.  On the other hand, we show that if $\lambda= (n)$, where $n \ge 3$,  then $\sum_jQ_{\lambda,j}$ is
the Frobenius characteristic of a  virtual representation (conjecturally, an actual representation) whose character
takes nonzero values on elements that do not commute with any
element of cycle type $\lambda$ (see Corollary~\ref{cvalcor}).

The symmetric function on the right hand side of (\ref{introsymgenth1})  and  (\ref{introsymgenth2}) refines  symmetric functions that have been studied earlier in the literature.  These symmetric functions include  enumerators for multiset derangements studied by MacMahon \cite[Sec. III, Ch. III]{mac1} and Askey and Ismail \cite{ai};  enumerators  for words with no adjacent repeats studied by Carlitz, Scoville and  Vaughan \cite{csv},
 Dollhopf, Goulden and Greene \cite{dgg} and Stanley \cite{st2};  chromatic symmetric functions  of Stanley \cite{st2}; and the  Frobenius characteristic of the representation of the symmetric group on the degree $2j$ cohomology of the toric variety $X_n$ associated to the Coxeter complex of the symmetric group $\mathfrak S_n$ studied by Processi \cite{pr}, Stanley \cite{st2}, Stembridge \cite{stem1,stem2}, and Dolgachev and Lunts \cite{dl}.   It is a consequence of our work that these symmetric functions   have nice interpretations as sums of fundamental quasisymmetric functions.  These connections and others are discussed in Section~\ref{repthsec}.

\subsection{Poset topology} Representations with Frobenius characteristic $Q_{n,j}$ also occur in poset topology.
In fact, it was our study of the homology of a certain poset introduced by Bj\"orner and Welker \cite{bw} that led us to conjecture Theorems~\ref{expgenth} and~\ref{introsymgenth} in the first place. The poset we consider is the Rees product $(B_n\setminus\{\emptyset\}) * C_n$, where $ B_n$ is the Boolean algebra on $\{1,2,\dots,n\}$ and $C_n$ is an $n$-element chain.  Rees products of posets were introduced by
Bj\"orner and Welker in \cite{bw}, where they study connections between poset topology and commutative algebra.  (Rees products of affine semigroup posets arise from the ring-theoretic Rees construction.)
 The Rees product of  two ranked posets  is a subposet of the usual product poset (the precise definition  is given in Section~\ref{rep}).

 Bj\"orner and Welker \cite{bw} conjectured and Jonsson \cite{jo} proved  that the dimension of the top homology of $(B_n\setminus\{\emptyset\}) * C_n$ is equal to the number of derangements in $\mathfrak S_n$.  Here we prove a refinement of this result; namely that the dimension of the top homology of any open principal  lower order ideal of $(B_n\setminus\{\emptyset\}) * C_n$ is an Eulerian number.  Moreover, we use Theorems~\ref{expgenth} and \ref{introsymgenth} to obtain a $q$-analog and an equivariant version of our refinement and of the Bj\"orner-Welker-Jonsson result.

   The poset $(B_n\setminus\{\emptyset\}) * C_n$ has $n$ maximal elements  all of rank $n$.  The usual action of $\mathfrak S_n$ on $B_n$ induces an action of $\mathfrak S_n$ on each lower order ideal generated by a maximal element, which in turn induces a representation of $\mathfrak S_n$ on the homology of the open lower order ideal.  In Section~\ref{treesec} we  prove the following equivariant version of our refinement.
 \begin{thm} \label{introhomol} Let $x_1,\dots, x_n$ be the maximal elements of $(B_n\setminus\{\emptyset\}) * C_n$.  For each $j=1,\dots, n$, let $I_{j}(B_n)$ be the open lower order ideal generated by $x_j$.  Then
 \begin{equation*}
{\rm {ch}}(\rh_{n-2}(I_{j}(B_n)) )=\omega \sum_{k=0}^n Q_{n,j-1,k},
\end{equation*}
where $\ch$ denotes the Frobenius characteristic and $\omega$ is the standard involution on the ring of symmetric functions\end{thm}

As a consequence we have the following equivariant version of the Bj\"orner-Welker-Jonsson result:
$$\ch(\tilde H_{n-1}((B_n\setminus \{\emptyset\}) \ast C_n)) =\omega \sum_{j= 0}^{n-1}  Q_{n,j,0}.$$

We also prove a $q$-analog of our refinement by considering the Rees product   $(B_n(q)\setminus\{0\}) * C_n$, where $B_n(q)$ is the lattice of subspaces of the $n$-dimensional vector space $\F_q^n$ over the $q$ element field $\F_q$.  Like $(B_n\setminus\{\emptyset\}) * C_n$, the q-analog $(B_n(q)\setminus\{0\}) * C_n$ has $n$ maximal elements all of rank $n$.

\begin{thm} \label{intohomolq} Let $x_1,\dots, x_n$ be the maximal elements of $(B_n(q)\setminus\{0\}) * C_n$.  For each $j=1,\dots, n$, let $I_{n,j}(q)$ be the lower order ideal generated by $x_j$.
  Then  \begin{equation*} \label{intoq}
  \dim \rh_{n-2}(I_{n,j}(q))=\sum_{\scriptsize\begin{array}{c} \s \in\sg_n \\ \exc(\s) = j-1 \end{array}} q^{\comaj(\s)+j-1},\end{equation*}
  where $\comaj(\s) = {n \choose 2}- \maj(\s)$.
\end{thm}

As a consequence, we have the following $q$-analog of the
Bj\"orner-Welker-Jonsson result:\begin{equation*} \dim \tilde
H_{n-1}((B_n(q)\setminus\{0\}) \ast C_n)= \sum_{\sigma \in \mathcal D_n} q^{\comaj(\s) + \exc(\s)},\end{equation*} where $\mathcal
D_n$ is the set of derangements in $\S_n$.

It is also interesting to consider the Rees product of $B_n$ (or $B_n(q)$) with a tree.
We prove the following result in Section~\ref{treesec}. For any poset $P$ with a minimum element $\hat 0$, let $P^-:=P\setminus \{\hat 0\}$.

\begin{thm} \label{introtreeth} For all $n,t\ge 1$, let $T_{t,n}$ be a poset whose Hasse diagram is a complete $t$-ary tree of height $n$ with the root at the bottom.  Then
\begin{eqnarray*}\dim \tilde H_{n-2} ((B_n * \ttn)^-) &=& tA_n(t)\\
\dim \tilde H_{n-2} ((B_n(q) * \ttn)^-) &=&    t q^{n\choose 2}A^{\maj,\exc}_n(q^{-1},qt) \\
\ch \tilde H_{n-2} ((B_n * \ttn)^-) &=&  \sum_{\scriptsize \begin{array}{c}0\le k\le n\\0 \le j\le n-1\end{array}}  \omega Q_{n,j,k}t^{j+1}.\end{eqnarray*}
\end{thm}
We  derive a general result relating the homology of  lower order ideals of $P^-*C_n$  to the homology of  the Rees product $P^**T_{t,n}$, where $P$ is any bounded, ranked poset of length $n$ and  $P^*$ is the dual of $P$. This result enables us to show that
   Theorem~\ref{introtreeth} implies Theorems~\ref{introhomol} and~\ref{intohomolq}.  We exploit  the recursive nature of $B_n * \ttn$ and $B_n(q) * \ttn$ in our proof of Theorem~\ref{introtreeth}.

 Various authors have studied Mahonian (resp. Eulerian) partners to
Eulerian (resp. Mahonian) statistics whose joint distribution is
equal to a known Euler-Mahonian distribution.   We mention, for
example, Foata \cite{foa3},  Foata and Zeilberger \cite{fz}, Clarke, Steingr\'{\i}misson and Zeng
\cite{csz}, Haglund \cite{hag}, Babson and Steingr\'{\i}msson \cite{bs},
Skandera \cite{sk} and Br\"and\'en \cite{bra}.  In  Section~\ref{elsec} we define a new Mahonian
statistic to serve as a partner for $\des$ in the $(\maj, \exc)$
distribution.   We do not have a simple proof of the
equidistribution.  We have a highly nontrivial proof  which uses Theorem~\ref{intohomolq} and poset topology techniques, such as an EL-labeling of $I_{n,j}(q)$, which is derived from  an EL-labeling of Simion \cite{si} for $B_n(q)$.

In the last section (Section~\ref{bcsec}) we use results from Section~\ref{rep} to derive  type BC analogs (in the sense of Coxeter groups) of the Bj\"orner-Welker-Jonsson derangement  result and our $q$-analog.

\part{Permutation statistics}

\section{Permutation statistics and quasisymmetric functions} \label{statsec}

For $n \ge 1$, let $\S_n$ be the symmetric group on the set
$[n]:=\{1,\ldots,n\}$.  A permutation $\sigma \in \S_n$ will be
represented here in two ways, either as a function that maps $i
\in [n]$ to $\sigma(i)$, or in one line notation as
$\sigma=\sigma_1\ldots\sigma_n$, where $\sigma_i=\sigma(i)$. If
$n\le 0$ then we set $[n] =\emptyset$ and $\S_n = \{\theta\}$
where $\theta$ denotes the empty word.

 The
{\it descent set} of  $\s$ is
\[
\Des(\s):=\{i \in [n-1]:\sigma_i>\sigma_{i+1}\},
\]
and the {\it excedance set} of $\sigma$ is
\[
\Exc(\sigma):=\{i \in [n-1]:\sigma_i>i\}.
\]

We now define  the two basic Eulerian permutation statistics.
The {\it descent number} and {\it excedance number} of $\sigma$
are, respectively,
\[
\des(\sigma):=|\Des(\sigma)|
\]
and
\[
\exc(\sigma):=|\Exc(\sigma)|.
\]
 For
example, if $\sigma = 32541$, written in one line notation, then
$$\Des(\s) = \{1,3,4\} \qquad \mbox{ and }\qquad \Exc(\s) = \{1,3\};$$
 hence $\des(\s) = 3 $ and $\exc(\s) = 2$.  If $i \in \Des(\s)$ we say
 that $\s$ has a descent at $i$.  If $i \in \Exc(\s)$ we say that
 $\s(i)$ is an excedance of $\s$ and that $i$ is an excedance position.

Next we define the two basic Mahonian permutation statistics.  The
{\it inversion index} of $\sigma \in \S_n$ is
\[
\inv(\sigma):=|\{(i,j):1 \leq i<j \leq n \mbox { and }\sigma_i>\sigma_j\}|,
\]
and the {\it major index} of $\sigma$ is
\[
\maj(\sigma):=\sum_{i \in \Des(\sigma)}i,
\]
 For
example, if $\sigma = 32541$ then $\inv(\s) = 6$ and $\maj(\s) = 8$.

We review  some basic facts of Gessel's theory of quasisymmetric functions; a good reference is  \cite[Chapter 7]{st3}. A {\it quasisymmetric function} is a formal power series $f(\xx)=f(x_1,x_2,\ldots)$ of finite degree with rational coefficients in the infinitely many variables $x_1,x_2,\ldots$ such that  for any $a_1,\dots,a_k \in \PP$,  the coefficient of $x_{i_1}^{a_1}\dots x_{i_k}^{a_k}$ equals the  coefficient of $x_{j_1}^{a_1}\dots x_{j_k}^{a_k}$ whenever $i_1 < \dots < i_k$ and $j_1 < \dots < j_k$.  Thus each symmetric function is quasisymmetric.

For a positive integer $n$ and $S \subseteq [n-1]$, define
\[
F_{S,n}=F_{S,n}(\xx):=\sum_{\scriptsize\begin{array}{c}i_1 \geq \ldots \geq i_n \geq 1\\ j \in S \Rightarrow i_j>i_{j+1}\end{array}}x_{i_1}\dots x_{i_n}
\] and let
$F_{\emptyset,0}=1$.
Each  $F_{S,n}$ is a quasisymmetric function.  The  set $\{F_{S,n} : S \subseteq [n-1], n\in \N\}$ is a basis
for the ring $\mathcal Q$ of quasisymmetric functions  and $F_{S,n}$ is called {\em a fundamental} quasisymmetric function.
If $S = \emptyset $ then $F_{S,n}$ is the complete homogeneous symmetric function $h_n$ and if $S=[n-1]$ then $F_{S,n}$ is the elementary symmetric function $e_n$.

We review two important ways to specialize a quasisymmetric funtion. Let
$\Q[q]$ denote the ring of polynomials in variable $q$ with coefficients in  $\Q$ and let
$\Q[[q]]$ denote the ring of formal power series in variable $q$ with coefficients in $\Q$.  The  {\em stable principal  specialization} is  the ring homomorphism
$\Lambda:\mathcal Q \to \Q[[q]]$ defined by $$\Lambda(x_i) =  q^{i-1} ,$$ and the {\em  principal  specialization} of order $m$ is the ring
homomorphism $\Lambda_m:\mathcal Q \to \Q[q]$ defined by
$$\Lambda_m(x_i) = \begin{cases} q^{i-1} &\mbox{if } 1 \le i \le m \\ 0 &\mbox{if } i > m \end{cases} .$$  Sometimes we will need to apply  $\Lambda$ and $\Lambda_m$ to  quasisymmetric functions whose coefficient ring is $\Q$ with indeterminants  adjoined.

\begin{lemma}[{\cite[Lemma 5.2]{gr}}] \label{desspec}For all $n \ge 0$ and $S\in [n-1]$, we have
\begin{equation}\label{spec1}\Lambda (F_{S,n})= \frac {q^{\sum_{i\in S} i}} {(q;q)_n} \end{equation} and
 \begin{equation} \label{spec2}\sum_{m \ge 0} \Lambda_m( F_{S,n}) p^m = {p^{|S|+1} q^{\sum_{i\in S} i} \over (p;q)_{n+1}}.\end{equation} \end{lemma}

It follows from Lemma~\ref{desspec} that
\begin{equation} \label{stabh}  \Lambda (h_n) = \frac 1{(1-q)(1-q^2) \dots (1-q^n)}\end{equation} for all $n$ and therefore
\begin{equation}\label{spech}\Lambda (H(z(1-q)) )= \exp_q(z).\end{equation}

The standard way to connect quasisymmetric functions with  permutation statistics is by associating  the fundamental quasisymmetric funciton $F_{\Des(\sigma),n}$ with  $\sigma \in \sg_n$.   By Lemma~\ref{desspec}, we have for any subset $A \in \sg_n$
$$\Lambda\left(\sum_{\sigma \in A} F_{\Des(\sigma),n}\right) = {1 \over (q;q)_n} \sum_{\s \in A} q^{\maj(\s)},$$
and
$$\sum_{m\ge 0} \Lambda_m\left(\sum_{\sigma \in A} F_{\Des(\sigma),n}\right) p^m = {1 \over (p;q)_{n+1}}\sum_{\s \in A} p^{\des(\s)+1}q^{\maj(\s)}$$
Gessel and Reutenauer \cite{gr} used this to study the $(\maj,\des)$-enumerator for permutations of a fixed cycle type.

Here we introduce a new way to associate   fundamental quasisymmetric functions with permutations.  For $n \in \N$, we set \[
[\ov{n}]:=\{\ov{1},\ldots,\ov{n}\}
\]
and totally order the alphabet $[n] \cup [\ov{n}]$ by
\begin{equation} \label{order1}
\ov{1}<\ldots<\ov{n}<1<\ldots<n.
\end{equation}
For a permutation $\sigma=\sigma_1\ldots\sigma_n \in \S_n$, we
define $\ov{\sigma}$ to be the word over  alphabet $[n] \cup
[\ov{n}]$ obtained from $\sigma$ by replacing $\sigma_i$ with
$\ov{\sigma_i}$ whenever $i \in \Exc(\sigma)$.  For example, if
$\sigma={\rm {531462}}$ then $\ov{\sigma}={\rm
{\ov{5}\ov{3}14\ov{6}2}}$.  We define a descent in a word
$w=w_1\ldots w_n$ over any totally ordered alphabet to be any $i
\in [n-1]$ such that $w_i>w_{i+1}$ and let $\Des(w)$ be the set of
all descents of $w$.  Now, for $\sigma \in \S_n$, we define
\[
\Exd(\sigma):=\Des(\ov{\sigma}).
\]
For example, $\Exd({\rm {531462}})=\Des({\rm {\ov{5}\ov{3}14\ov{6}2}})=\{1,4\}$.

\begin{lemma} \label{exdlem}  For all $\s \in \S_n$,
\begin{equation}\label{exd}  \sum_{i \in \Exd(\s)} i = \maj(\s) - \exc(\s),\end{equation}
and
\begin{equation}\label{exdsize} |\Exd(\s)| =  \begin{cases} \des(\s)  &\mbox{if } \s(1) = 1 \\  \des(\s)-1 &\mbox {if } \s(1) \ne 1 \end{cases} \end{equation}
\end{lemma}

\begin{proof} Let $$J(\s) = \{i \in [n-1]:  i \notin \Exc(\s) \,\,\, \& \,\,\, i+1 \in \Exc(\s)\},$$
and let
$$ K(\s) = \{i \in [n-1]:  i\in \Exc(\s) \,\,\, \& \,\,\, i+1 \notin \Exc(\s)\}.$$
If $i \in J(\s)$ then $\s(i) \le i < i+1 < \s(i+1)$.  Hence $i \notin \Des(\s)$, but $i \in \Exd(\s)$.  If $i \in K(\s)$ then $\s(i)\ge  i+1 \ge \s(i+1)$.  Hence $i \in \Des(\s)$, but $i \notin \Exd(\s)$.  It follows that $K(\s) \subseteq \Des(\s)$ and
\bq \label{dexdes}  \Exd(\s) = \Des(\s) \uplus J(\s) - K(\s).\eq
Hence
\bq \label{dexdes2}  |\Exd(\s)| = \des(\s) +| J(\s) | - |K(\s)|.\eq

Let $J(\s) = \{j_1 < j_2 < \dots < j_t\}$ and $K(\s) = \{k_1 < k_2< \dots < k_s\}$.  It is easy to see that if
$\s(1) = 1$ then $t= s$ and
\bq \label{shuf}  j_1 < k_1 <j_2 < k_2 <\dots <j_t <k_t, \eq  since neither $1$ nor $n$ are excedance positions.  On the other hand if $\s(1) \neq 1$ then $s=t+1$ and
 \bq \label{shuf2}  k_1 < j_1 <k_2 <j_2<\dots< j_t < k_{t+1}.\eq  It now follows from (\ref{dexdes2}) that (\ref{exdsize}) holds.

To prove (\ref{exd}), we again handle the cases  $\s(1)=1$ and $\s(1) \ne 1$ separately.

Case 1: Suppose $\s(1) = 1$.  Then (\ref{shuf}) holds.  By (\ref{dexdes}),
$$\sum_{i \in \Exd(\s)} i = \sum_{i \in \des(\s)}  i - \sum_{i=1}^t (k_i -j_i) . $$
Clearly $$\Exc(\s) = \biguplus_{i = 1}^t \{j_i+1,j_i+2,\dots, k_i\}.$$  Hence
$$\exc(\s) = \sum_{i=1}^t (k_i -j_i) $$ and so (\ref{exd}) holds in this case.

Case 2: Suppose $\s(1) \ne 1$. Now (\ref{shuf2}) holds.
By (\ref{dexdes}),
$$\sum_{i \in \Exd(\s)} i = \sum_{i \in \des(\s)}  i - k_1 - \sum_{i=1}^t (k_{i+1} -j_{i}) . $$
Now $$\Exc(\s) = \{1,2,\dots, k_1\} \uplus  \biguplus_{i = 1}^t \{j_i+1,j_i+2,\dots, k_{i+1}\},$$
which implies that
$$\exc(\s) =  k_1 + \sum_{i=1}^t (k_{i+1} -j_i) .$$  Hence (\ref{exd}) holds in this case too.
  \end{proof}

We now define the  quasisymmetric functions that  play a central role in this paper. Let  $n\ge 1$, $j,k \ge 0$ and  $\lambda\vdash n$.   The {\em Eulerian quasisymmetric functions}  are defined by $$ Q_{n,j} = Q_{n,j}(\x):=
\sum_{\scriptsize \begin{array}{c} \s \in \sg_n \\ \exc(\s) = j
\end{array}} F_{\Exd(\s),n}(\x).$$
 The {\em fixed point Eulerian quasisymmetric functions} are  defined by $$ Q_{n,j,k} = Q_{n,j,k}(\x):=
\sum_{\scriptsize \begin{array}{c} \s \in \sg_n \\ \exc(\s) = j \\ \fix(\s) =k
\end{array}} F_{\Exd(\s),n}(\x).$$
The {\em cycle type Eulerian quasisymmetric functions} are defined by
$$Q_{\lambda,j} = Q_{\lambda,j}(\x):=
\sum_{\scriptsize \begin{array}{c} \s \in \sg_n \\ \exc(\s) = j \\ \lambda(\s) =\lambda
\end{array}} F_{\Exd(\s),n}(\x),$$
where $\lambda(\s)$ denotes the cycle type of $\s$.
For convenience we let $\Exd(\theta) = \Des(\theta) = \Exc(\theta) = \emptyset$ and $ \exc(\theta) = \fix(\theta) = \des(\theta) = \maj(\theta)= 0$.   So $Q_{0,0,0} = 1$.
 It is a consequence of Theorem~\ref{introsymgenth} that the Eulerian quasisymmetric functions $Q_{n,j}$ and $Q_{n,j,k}$ are symmetric functions.  We prove that the refinement $Q_{\lambda,j}$ is symmetric, as well, in Section~\ref{symcyclesec}.

The {\em  $(q,p)$-Eulerian numbers} are defined by $$ a_{n,j}(p,q) =
\sum_{\scriptsize \begin{array}{c} \s \in \sg_n \\ \exc(\s) = j
\end{array}} q^{\maj(\s)} p^{\des(\s)}.$$
The {\em fixed point $(q,p)$-Eulerian numbers} are defined by $$ a_{n,j,k}(p,q) =
\sum_{\scriptsize \begin{array}{c} \s \in \sg_n \\ \exc(\s) = j \\ \fix(\s) =k
\end{array}} q^{\maj(\s)} p^{\des(\s)}.$$
The {\em cycle type $(q,p)$-Eulerian numbers} are defined by
$$a_{\lambda,j}(p,q) =
\sum_{\scriptsize \begin{array}{c} \s \in \sg_n \\ \exc(\s) = j \\ \lambda(\s) =\lambda
\end{array}}  q^{\maj(\s)} p^{\des(\s)}.$$

It  follows from (\ref{spec1}) and(\ref{exd}) that for $\sigma \in \sg_n$ we have
$$
\Lambda(F_{\Exd(\sigma),n})=(q;q)_n^{-1}q^{\maj(\sigma)-\exc(\sigma)}.
$$
Hence \begin{eqnarray}
 \label{stablespec1} a_{n,j,k}(q,1) &=& q^j (q;q)_n\Lambda (Q_{n,j,k}),\\ \label{stablespec2} a_{\lambda,j}(q,1) &=& q^j (q;q)_n\Lambda (Q_{\lambda,j}).\end{eqnarray}
From the first of these equations and (\ref{spech}), we see that Corollary~\ref{expgenthfix} is obtained from Theorem~\ref{introsymgenth} by applying the stable principal specialization.

The (nonstable) principle specialization is a bit more
complicated. Given two partitions $\lambda \vdash n$ and $\mu
\vdash m$, let $(\lambda,\mu)$ denote the partition of $n+m$
obtained by concatenating $\lambda$ and $\mu$ and then
appropriately rearranging the parts.
\begin{lemma} \label{nonstable} For $n,j,k \ge 0$ and $\lambda
\vdash n-k$, where $\lambda$ has no parts of size $1$,
\begin{equation*}a_{(\lambda,1^k),j}(q,p) = (p;q)_{n+1} \sum_{m\ge 0} p^m \sum_{i=0}^k q^{im+j} \Lambda_m (Q_{(\lambda,1^{k-i}), j}).\end{equation*}  Consequently, for $n,j,k \ge 0$,
 \begin{equation*}  a_{n,j,k}(q,p) = (p;q)_{n+1} \sum_{m\ge 0} p^m \sum_{i=0}^k q^{im+j} \Lambda_m (Q_{n-i, j,k-i}).\end{equation*}
\end{lemma}

\begin{proof}

For all $\sigma \in \sg_n$ we have, by (\ref{spec2}) and Lemma~\ref{exdlem},
$$\sum_{m \ge 0} \Lambda_m(F_{\Exd(\sigma),n}) p^m = {1 \over (p;q)_{n+1} } p^{|\Exd(\s)|+1} q^{\maj(\s) - \exc(\s)}.$$
It follows that \begin{equation} \label{specx} \sum_{m \ge 0} \Lambda_m( Q_{(\lambda,1^k),j}) p^m = X^k_{\lambda,j}(q,p),\end{equation}
where
$$ X^k_{\lambda,j}(q,p) := {1 \over (p;q)_{n+1} }\sum_{\scriptsize \begin{array}{c} \s \in \sg_n \\ \exc(\s) = j \\ \lambda(\s) = (\lambda, 1^k)
\end{array}} q^{\maj(\s)-j } p^{|\Exd(\s)|+1}.$$
By (\ref{exdsize}) we have
\begin{eqnarray*} X^k_{\lambda,j}(q,p) := & &{1 \over (p;q)_{n+1} } \sum_{\scriptsize \begin{array}{c} \s \in \sg_n \\ \exc(\s) = j \\ \lambda(\s) = (\lambda, 1^k) \\ \s(1) = 1
\end{array}} q^{\maj(\s)-j } p^{\des(\s)+1} \\ &+&
 {1 \over (p;q)_{n+1} }\sum_{\scriptsize \begin{array}{c} \s \in \sg_n \\ \exc(\s) = j \\ \lambda(\s) = (\lambda, 1^k) \\ \s(1) \ne 1
\end{array}} q^{\maj(\s)-j } p^{\des(\s)}.
\end{eqnarray*}
Let $$ Y^k_{\lambda,j}(q,p) := {1 \over (p;q)_{n+1} }\sum_{\scriptsize \begin{array}{c} \s \in \sg_n \\ \exc(\s) = j \\ \lambda(\s) = (\lambda, 1^k) \\ \s(1) =1
\end{array}} q^{\maj(\s)-j } p^{\des(\s)}.$$
Write  $a^k_{\lambda,j}(q,p)$ for $a_{(\lambda,1^k),j}(q,p)$.
We have  \begin{eqnarray}\label{ayxeq} \nonumber {a^k_{\lambda,j}(q,p) \over q^j(p;q)_{n+1} } &=& Y^k_{\lambda,j}(q,p) + X^k_{\lambda,j}(q,p) - p Y^k_{\lambda,j}(q,p)\\ &=& (1-p) Y^k_{\lambda,j}(q,p)  + X^k_{\lambda,j}(q,p)\end{eqnarray}

Let
$\varphi: \{\s \in \sg_n : \s(1) = 1\} \to \sg_{n-1}$ be the bijection defined by letting $\varphi(\s)$ be the permutation  obtained by removing the $1$ from the beginning of $\s$ and subtracting $1$ from each letter of  the remaining word.
It is clear that  $\maj(\s) = \maj(\varphi(\s)) + \des(\s)$,  $\des(\s) = \des(\varphi(\s)) $, $\exc(\s)=\exc(\varphi(\s)) $ and $\fix(\s)=\fix(\varphi(\s)) +1$.  Hence
$$Y^k_{\lambda,j}(q,p)= {a^{k-1}_{\lambda,j}(q,qp) \over q^j (p;q)_{n+1} }$$
Plugging this into (\ref{ayxeq}) yields the recurrence
$$ {a^k_{\lambda,j}(q,p)  \over q^j (p;q)_{n+1} }= { a^{k-1}_{\lambda,j}(q,qp) \over q^j(qp;q)_{n} } + X^k_{\lambda,j}(q,p).$$
By iterating this recurrence we get
$$ {a^k_{\lambda,j}(q,p) \over q^j (p;q)_{n+1} }= \sum_{i=0}^k X^{k-i}_{\lambda,j}(q,q^ip).$$
The result now follows from this and (\ref{specx}).
\end{proof}

\begin {proof} [Proof of Corollary \ref{foha}]  By Lemma~\ref{nonstable} and Theorem~\ref{introsymgenth}, we have

 \begin{eqnarray*} \sum_{n\ge 0 }A_{n}^{\maj,\des,\exc,\fix}&&\hspace{-.4in}(q,p,t,r) {z^n \over (p;q)_{n+1}} 
 \\ &=&  \sum_{n,j,k \ge 0} a_{n,j,k}(p,q) t^j r^k {z^n \over (p;q)_{n+1}}\\
 &=&\sum_{n,j,k\ge 0}z^n  t^j r^k \sum_{m\ge 0} p^m \sum_{i=0}^k q^{im+j} \Lambda_m (Q_{n-i, j,k-i})
 \\  &=&  \sum_{m\ge0}p^m  \sum_{i \ge 0} (zrq^m)^i  \sum_{\scriptsize \begin{array} {c}
 n,k\ge i \\ j \ge 0\end{array}} \Lambda_m (Q_{n-i, j,k-i}) (qt)^j r^{k-i} z^{n-i}
  \\ 
 &=& \sum_{m\ge 0}{p^m\over 1-zrq^m}\Lambda_m {(1-tq)H(zr) \over H(ztq) -tqH(z)}
\\ &=& \sum_{m \ge 0} p^m {(1-tq) (z;q)_m (ztq;q)_m \over
((z;q)_m -tq(ztq;q)_m)(zr;q)_{m+1}}, \end{eqnarray*}
with the last step following from
$$ \Lambda_m( H(z)) =\Lambda_m\left(\prod_{i \ge 0} \frac 1 {1-x_i z}\right) = \frac 1 {(z;q)_{m}}.$$

\end{proof}

\section{Bicolored necklaces and words} \label{identsec}

This section is devoted to the proof of  Theorem~\ref{introsymgenth}.  There are three main steps.   In the first step (Section~\ref{secnec}) we modify a bijection that  Gessel and Reutenauer \cite{gr} constructed in order to enumerate permutations with a fixed descent set and
fixed cycle type.  This yields an alternative characterization of the Eulerian quasisymmetric functions involving  bicolored necklaces.   In the second step (Section~\ref{secban}) we construct  a  bijection from multisets of bicolored necklaces to
bicolored words, which involves Lyndon decompositions of words.  This yields  yet another characterization of the Eulerian quasisymmetric functions.    In the third step (Section~\ref{recsec}) we generalize a bijection that Stembridge constructed to study the representation of the symmetric group on the cohomology of the toric variety associated with the type A Coxeter complex.  This enables us to derive a recurrence relation, which yields
Theorem~\ref{introsymgenth}.

 \subsection{Step 1: bicolored version of the Gessel-Reutenauer bijection} \label{secnec}
The Gessel-Reutenauer bijection is a   bijection between   pairs $(\sigma,s)$, where $\s$ is a  permutation and $s$ is a ``compatible'' weakly decreasing sequence, and multisets of primitive circular words over the alphabet of positive integers.  This bijection enabled Gessel and Reutenauer  to use the fundamental quasisymmetric functions $F_{\Des(\s),n}$ to study properies of permutations with a fixed descent set and cycle type.   Here we introduce a bicolored version of the Gessel-Reutenauer bijection.  

We  consider circular words over the alphabet of bicolored positive integers $$\mathcal A:=\{1,\bar 1,2,\bar 2, 3,\bar 3, \dots\}.$$   (We can think of ``barred'' and ``unbarred'' as  colors assigned to each positive integer.)  For each such circular
word and each starting position, one gets a linear word by
reading the circular word in a clockwise direction. If one gets a
distinct linear word for each starting position then the
circular word is said to be {\em primitive}.  For example the circular word $(\bar 1,1,1) $  is primitive while the circular word $(\bar 1, 2,\bar 1, 2)$  is not.  (If $w$ is a linear word then $(w)$ denotes the circular word obtained by placing the letters of $w$ around a circle in a clockwise direction.)  The
{\em absolute value} or just {\em value} of a letter $a$ is the letter obtained by ignoring
the bar if there is one. We denote this by $|a|$.
\begin{defn}  \label{orndef} A  {\em bicolored  necklace} is a primitive circular word $w$ over alphabet $\mathcal A$   such that if the length of $ w$ is greater than $1$ then \begin{enumerate}
\item every barred letter  is followed
(clockwise) by a letter less  than or equal to it  in absolute value
\item every unbarred  letter is followed by a letter greater than or equal to it  in absolute value.
\end{enumerate}      A circular
word of length $1$ is a bicolored necklace  if its sole letter is unbarred.
 A
{\em bicolored ornament}  is a
multiset
of
bicolored necklaces.
\end{defn}

For example the following
circular words are bicolored necklaces:
$$ (\bar 3, 1 ,3, \bar 3, 2, 2), (\bar 3,1, \bar 3, \bar 3, 2, 2) , (\bar 3, 1, 3, \bar 3, \bar 2, 2),
(\bar 3, 1, \bar 3, \bar 3, \bar 2, 2), (2),$$
 while $(\bar 3, \bar 1, 3, 3, 2, \bar 2) $
 and $(\bar 3)$  are not.

From now on we will drop the word ``bicolored'';   so ``necklace'' will stand for  ``bicolored necklace'' and ``ornament"  will stand for  ``bicolored ornament''.  The type $\lambda(R)$ of an ornament $R$ is the
partition whose parts are the sizes of the necklaces  in $R$. The weight of a letter $a$ is  the
indeterminate $x_{|a|}$.  The
weight $\wt(R)$ of an ornament $R$  is the product of the weights of
the letters of  $R$.  For example
$$\lambda((\bar 3,2,2),(\bar 3,\bar 2, 1,1,2)) = (5,3)$$
 and
$$\wt((\bar 3,2,2),(\bar 3,\bar 2,1,1,2)) = x_3^2x_2^4x_1^2.$$
For each partition $\lambda$ and nonnegative integer $j$, let
$\mathfrak R_{\lambda,j}$ be the set of ornaments of type
$\lambda$ with $j$ bars.

Given a permutation  $\sigma \in \S_n$, we say that a weakly decreasing sequence $s_1\ge s_2\ge \dots \ge s_n$ of
positive integers is  $\sigma$-{\em compatible} if
$s_i > s_{i+1}$ whenever $i \in \Exd(\sigma)$.  For example,
$7,7,7,5,5,4,2,2$ is $\s$-compatible, where $\s=45162387$.
Note that for all $\s \in \mathfrak S_n$,
$$F_{\Exd(\s),n} = \sum_{s_1,\dots, s_n} x_{s_1}\dots x_{s_n}$$
where $s_1,\dots,s_n$ ranges over all $\s$-compatible sequences.

 For $\lambda \vdash n$ and $j =0,\dots,n-1$, let $\Com_{\lambda,j}$ be the set of pairs $(\s,s)$, where $\s $ is a permutation of cycle type $\lambda$ with $j$ excedances and $s$ is  a $\s$-compatible sequence.
 Let  $\phi: \Com_{\lambda,j} \to {\mathfrak R}_{\lambda,j} $ be the map defined by letting
$\phi(\sigma, s)$ be the ornament obtained by first writing $\sigma$ in cycle form with bars above the letters that are followed (cyclicly) by larger letters (i.e., the excedances) and then replacing each $i$ with $s_i$, keeping the bars in place.  For example,
let $\s = 45162387$ and $s=7,7,7,5,5,4,2,2$.  First we write $\s$  in cycle form,  $$\s = (1,4,6,3)(2,5) (7,8).$$
Next we put  bars above the
letters that are followed  by larger letters, $$(\bar 1,\bar 4,6,3)(\bar 2,5) (\bar 7,8).$$
After replacing each $i$ by $s_i$, we have the ornament $$(\bar 7,\bar 5,4,7)(\bar 7,5) (\bar 2,2).$$

\begin{thm} \label{ornbij} The map $\phi: \Com_{\lambda,j} \to {\mathfrak R}_{\lambda,j} $ is a  well-defined bijection.
\end{thm}

\begin{proof}

Let $(\s,s) \in \Com_{\lambda,j}$.  It is clear that the circular words in the multiset $\phi(\sigma,s)$ satisfy conditions (1) and (2) of Definition~\ref{orndef} and that  the number of bars of $\phi(\sigma,s)$ is $j$.  It is also clear that $\lambda(\phi(\s,s))= \lambda$.

We must now show that the circular words in $\phi(\sigma,s)$ are primitive.
Suppose there is a circular word $(a_{i_1}, \dots, a_{i_k})$ that is not primitive.  Let this word come from the cycle $(i_1,i_2,\dots,i_k)$ of $\sigma$, where $i_1$ is the smallest element of the cycle.   Suppose $a_{i_1}, \dots, a_{i_k} = (a_{i_1},\dots a_{i_d})^{k/d}$, where $d < k$.   Since $a_{i_1} = a_{i_{d+1}}$,  $i_1 < i_{d+1}$,  and the $|a_i|$'s form a weakly decreasing sequence, it follows that $|a_{i}| = |a_{i_1}|$ for all $i$ such that $i_1 \le  i  \le i_{d+1}$.  Hence since $(|a_1|, |a_2|, \dots, |a_n| )$ is $\s$-compatible, there can be no element of $\Exd(\s)$ in the set $\{i_1,i_1+1, \dots,i_{d+1}-1\}$.  Since $a_{i_1} $ and $a_{i_{d+1}}$ are both barred or both unbarred, it follows that both $i_1$ and $i_{d+1}$ are  excedance positions or  both are  not. Hence $$\s(i_1) < \s(i_1+1) < \dots < \s(i_{d+1}).$$

Since $\s({i_1}) = i_2$ and $\s(i_{d+1})= i_{d+2}$ we have $i_2 < i_{d+2}$.  Repeated use of the above argument yields  $i_3 < i_{d+3}$ and eventually  $i_{k+1-d} < i_{1}$, which is impossible since $i_1$ is the smallest element of its cycle.  Hence all the circular words of $\phi(\sigma,s)$ are primitive and $\phi(\sigma,s)$ is an ornament of type $\lambda$ with $j$ bars, which means that $\phi$ is  well-defined.

To show that $\phi$ is a bijection, we construct its inverse $\eta:{\mathfrak R}_{\lambda,j} \to \Com_{\lambda,j}$, which takes the ornament $R$ to the pair $(\s(R),s(R))$.
Here $s(R)$ is simply  the weakly decreasing rearrangment of the letters  of $R$ with bars removed.  To construct the  permutation $\s(R)$, we need to first order the alphabet  $\mathcal A$  by
\begin{eqnarray} \label{alphorder} 1<\bar 1 < 2 <  \bar 2 < \dots .\end{eqnarray}   Note that this ordering is different from the ordering  (\ref{order1}) of the finite
alphabet that was used to define $\Exd$.

For each position $x$ of each necklace  of $R$, consider the infinite word
$w_x$ obtained by reading the necklace in a clockwise direction starting at position $x$.      Let $$w_x \le_L w_y$$ mean that $w_x$ is lexicographically less than or equal to $w_y$.  We use the lexicographic order on words  to order the positions:  if $w_x <_L w_y$ then we say  $x< y$.  We break ties as follows:  if   $w_x= w_y$ and $x \ne y$ then $x$ and $y$ must be positions in distinct (but equal) necklaces since necklaces are primitive.  If  $w_x= w_y$   and  $x$ is a position in an earlier necklace than that of $y$ under some fixed linear ordering of the necklaces   then let $x<y$.   If $x$ is the ith {\em largest } position  then replace the letter in position $x$ by $i$.  We now have a multiset of circular words in which each letter  $1,2,\dots,n$ appears exactly once.  This is the cycle form of the permutation $\s(R)$.

For example, if
$$R= ((\bar 7,  \bar 3,  3, 5), (\bar 7, 3, \bar 5, 3), (\bar 7, 3, \bar 5, 3), (5)),$$
then $$\s(R) = (1,8,13,6), (2,11,4,9) ,(3,12,5,10),(7) $$and $$ s(R) =  7,7,7,5,5,5,5,3,3,3,3,3,3.$$

It is not hard to see  that the letter in position $x$ of $R$ is barred if and only if  the letter that replaces it is  an excedance position of $\s(R)$.  This implies  that the number of bars of $R$ equals the number of excedances of $\s(R)$.

We now  show that $s(R)$ is $\s(R)$-compatible.
Suppose $$s(R)_i = s(R)_{i+1}.$$ We must show $i \notin \Exd(\s(R))$.  Let $x$ be the $i$th largest position of $R$ and let $y$ be the $(i+1)$st largest.  Then $i$ is placed in position $x$ and $i+1$ is placed in position $y$.   Let $f(w)$ denote the first letter of a word $w$.  Then one of the following must hold
\begin{enumerate}
\item
$f(w_x) =  s(R)_i = f(w_y)$
 \item $f(w_x) =  \overline{s(R)_i}= f(w_y)$
\item $f(w_x) =  \overline{s(R)_i}$
and  $f(w_y) = s(R)_i$ .
\end{enumerate}

Cases (1) and (2):    Either $w_x  >_L w_y$ or  $w_x = w_y$ and $x$ is in an earlier necklace than $y$. Let $u$ be the position that follows $x$ clockwise and let $v$ be the position that follows $y$.
  Since  $w_u$ is the word obtained from $w_x$ by removing its first letter,  $w_v$ is obtained from $w_y$ by removing its first letter, and the first letters are equal, we conclude that   $u >v$.  Hence the letter that gets placed in position $u$ is smaller than the letter that gets placed in position $v$.  Since the letter  placed in position $u$ is  $\s(R)(i) $ and the letter placed in position $v$ is $\s(R)(i+1) $, we have $\s(R)(i) <\s(R)(i+1) $. Since  $i,i+1 \in \Exc(\s(R))$ or $i,i+1 \notin \Exc(\s(R))$, we conclude that $i \notin \Exd(\s(R))$.

Case (3):  Since the letter in position $x$ is barred we have $i \in \Exc(\s(R))$.  Since the letter in position $y$ is not barred, we have $i+1 \notin \Exc(\s(R))$.  Hence $i \notin \Exd(\s(R))$.

In all three cases we have that $s(R)_i = s(R)_{i+1}$ implies $i \notin \Exd(\s(R))$.  Hence $s(R)$ is $\s(R)$-compatible.

Now we show that the map $\eta$  is the inverse of $\phi$.  It is easy to see  that $\phi(\eta(R)) =R$.
Establishing $\eta(\phi(\s,s)) = (\s,s)$ means establishing $\s(\phi(\s,s)) = \s$ and  $s(\phi(\s,s)) = s$.  The latter equation is obvious.   To establish the former, let $R= \phi(\sigma,s)$.  Recall that $R$ is obtained by writing $\s$ in cycle form, barring the excedances,  and then replacing each  $i$ by $s_i$, keeping the bars intact.    Let $p_i$ be the position that $i$ occupied before the replacement.  By ordering the cycles of $\s$ so that the minimum elements of the  cycles increase, we get  an ordering of the necklaces in $R$, which we use to break ties between the $w_{p}$. To show that $\s(R) = \s$, we need
to show that  \begin{enumerate}
\item if $i<j$ then $w_{p_i}   \ge_L w_{p_j}$
\item  if $i<j$ and $w_{p_i}  = w_{p_j} $ then $i$ is in a cycle whose minimum element is less than that of $j$.
\end{enumerate}

To prove (1) and (2),  we will use the following implication:
\begin{equation} \label{claimimp} i<j \mbox{ and }w_{p_i} \le_L w_{p_j}  \implies f(w_{p_i}) = f(w_{p_j}) \mbox{ and } \s(i) < \s(j) ,\end{equation} (Recall $f(w) $ is the first letter of a word $w$.)

Proof of implication:  Since $i <j$ and $s$ is weakly decreasing, we have  $s_i \ge s_j$.    Since $w_{p_i} \le_L w_{p_j}$, we have $f(w_{p_i}) \le f(w_{p_{j}})$, which implies that  $s_i \le s_j$.   Hence $s_i=s_j$, which implies $$s_i = s_{i+1} = \dots = s_j.$$  It follows from this and  the fact that $s$ is $\s$-compatible that $k \notin \Exd(\s)$ for all $k= i, i+1, \dots, j-1$.   This implies that either $$i,i+1,\dots, j \in \Exc(\s) \mbox{ and  } \s(i) < \s(i+1) <\dots < \s(j)$$ or  $$i,i+1,\dots, j \notin \Exc(\s) \mbox{ and } \s(i) < \s(i+1) <\dots < \s(j)$$  or $$i \in \Exc(\s) \mbox { and }  j \notin \Exc(\s).$$ In the first case, $f(w_{p_i}) =\bar s_i =\bar s_j = f(w_{p_j})$.  In the second case,  $f(w_{p_i}) = s_i =s_j = f(w_{p_j})$.  In the third case $f(w_{p_i}) = \bar s_i >s_i = s_j = f(w_{p_j})$,  which is impossible.   Hence the conclusion of the implication (\ref{claimimp}) holds.

Now we use (\ref{claimimp}) to prove (1).  Suppose $i<j$ and $w_{p_i} <_L w_{p_j}$.  Then by (\ref{claimimp}), we have $\s(i) < \s(j)$ and  $f(w_{p_i}) = f(w_{p_j})$, which implies $w_{p_{\s(i)}} <_L w_{p_{\s(j)}}$.  Hence we can apply (\ref{claimimp}) again with $\s(i)$ and $\s(j)$ playing the  roles of $i$ and $j$.  This yields  $f(w_{p_{\s^2(i)}}) = f(w_{p_{\s^2(j)}})$,  $\s^2(i) <\s^2(j)$, and $w_{p_{\s^2(i)}} <_L w_{p_{\s^2(j)}}$.    Repeated application of (\ref{claimimp}) yields
$ f(w_{p_{\s^m(i)}}) = f(w_{p_{\s^m(j)}})$ for all $m$, which implies $w_{p_i} =w_{p_j}$, a contradiction.

Next we prove (2).  Repeated application of (\ref{claimimp}) yields $\s^m(i) < \s^m(j)$ for all $m$.  Hence the cycle containing $i$ has a smaller minimum than the cycle containing $j$.
\end{proof}

\begin{cor} \label{ornth}
For all $\lambda\vdash n$ and $j= 0,1,\dots,n-1$,
$$Q_{\lambda,j} = \sum_{R \in \mathfrak R_{\lambda,j}} w(R).$$
\end{cor}

Corollary~\ref{ornth}  has several interesting consequences.  For one thing,
it can be  used  to prove that the Eulerian quasisymmetric functions
$Q_{\lambda,j}$ are actually symmetric (see Section~\ref{secQ}).  It also
has the following useful consequence.

\begin{cor}
\label{dercor} For all $n,j,k$,  $$ Q_{n,j,k} =  h_k
Q_{n-k,j,0}.$$
\end{cor}

It follows from Corollary~\ref{dercor} that Theorem~\ref{introsymgenth}
is equivalent to
\begin{equation*} \label{symgen2}
\sum_{n,j \ge 0} Q_{n,j,0} t^j z^n = {1-t  \over H(zt) -tH(z)},
\end{equation*}
which in turn, is equivalent to the recurrence relation
\begin{equation} \label{rr}
Q_{n,j,0} = \sum_{\scriptsize \begin{array}{c}0 \le m \le n-2
\\ j+m-n < i < j \end{array}} Q_{m,i,0}   h_{n-m}.
\end{equation}

\subsection{Step 2: banners} \label{secban}

In order to establish the recurrence relation (\ref{rr}), we  introduce another type
of configuration, closely related to ornaments.

\begin{defn}  \label{bandef} A  {\em banner} is a  word $B$ over alphabet $\mathcal A$   such that  for all $i =1,\dots ,\ell(B)$, \begin{enumerate}
\item if $B(i)$ is barred then $|B(i)| \ge |B(i+1)|$
\item if $B(i)$ is unbarred then $|B(i)| \le |B(i+1)|$ or $i= \ell (B)$,
\end{enumerate}
where $B(i)$ denotes the $i$th letter of $B$ and $\ell(B)$ denotes the length of $B$.     \end{defn}

A {\em Lyndon word} over an ordered alphabet is a  word that is
strictly lexicographically larger than all its circular rearrangements.  A
{\em Lyndon factorization} of a word over an ordered alphabet is a
factorization into a weakly lexicographically  increasing sequence
of  Lyndon words.  It is a result of Lyndon (see  \cite[Theorem~5.1.5]{p}) that every
word has a unique Lyndon factorization.  The {\em Lyndon type} $\lambda(w)$ of a word $w$
is the partition whose parts are the lengths of the words in its
Lyndon factorization.

To apply the theory of Lyndon words to banners, we use the ordering of $\mathcal A$ given in (\ref{alphorder}).  Using this order,  the  banner $B:= \bar 2 2\bar 7 5\bar 7\bar 547 $  has Lyndon factorization   $$\bar 2 2\cdot \bar 7 5 \cdot\bar 7\bar 547 .$$ So the Lyndon type  of $B $ is the partition $(4,2,2)$.

The weight $\wt(B)$ of a banner $B$ is the product of the weights of
its letters, where as before the weight of a letter $a$ is $x_{|a|}$.  For each partition $\lambda$ and nonnegative
integer $j$, let $\mathfrak B_{\lambda,j}$ be the set of banners
with $j$ bars whose Lyndon type is $\lambda$.

\begin{thm} \label{banprop} For each partition $\lambda$ and nonnegative
integer $j$, there is a weight-preserving bijection
$$\psi: \mathfrak B_{\lambda,j} \to \mathfrak R_{\lambda,j}.$$
Consequently, $$Q_{\lambda,j} = \sum_{B \in  \mathfrak B_{\lambda,j}} \wt(B).$$ \end{thm}

\begin{proof} First note that there is a natural weight-preserving bijection from the set of Lyndon banners  of length $n$ to the set of necklaces of size $n$.  To go from a Lyndon banner $B$ to a necklace $(B)$ simply  attach  the ends of  $B$ so that the left end follows the right end when read in a clockwise direction.    To go from a necklace back to a banner simply find the lexicographically largest linear word obtained by reading the circular word in a clockwise direction.
The number of bars of $B$  clearly is the  same as that of $(B)$.

Let $B \in \mathfrak B_{\lambda,j} $ and let
$$B = B_1 \cdot B_2 \cdots B_k$$ be the unique Lyndon factorization of $B$.  Note that each $B_i$ is a  Lyndon  banner.  To see this we need only check that the last letter of each $B_i$ is unbarred.  The last letter of $B_k$ is the last letter of $B$; so it is clearly unbarred.  For $i < k$, the  last letter of $B_i$ is strictly less than the first letter of $B_i$, which is less than or equal to the first letter of $B_{i+1}$.  Since the last letter of $B_i$ immediately precedes the first letter of $B_{i+1}$ in the banner $B$,  it must be unbarred.

 Now define
$\psi(B)$ to be the ornament whose necklaces are $$(B_1), (B_2), \dots, (B_k).$$ This map is clearly weight preserving, type preserving, and bar preserving.   To go from an ornament back to a banner simply arrange the Lyndon banners obtained from the necklaces in the ornament in weakly increasing order and then concatenate.
\end {proof}

\subsection{Step 3: the recurrence relation} \label{recsec}
Define a {\em marked sequence} $(\omega,i)$  to be a weakly
increasing finite sequence $\omega$ of positive integers together
with an integer  $i$ such that $1 \le i \le \mbox
{length}(\omega)-1$.  (One can visualize this as a weakly increasing sequence with a mark above any of its elements except the last.) For $n \ge 2$, let $\mathfrak M_n$ be the set of marked
sequences of length $n$.  For $n \ge 0$,  let ${\mathfrak B}^0_n$ be the set
of banners of length $n$ whose Lyndon type has no parts of size
$1$.    It will be convenient to consider the empty word to be a banner of length $0$, weight $1$, with no bars, and whose Lyndon type is the partition of $0$ with no parts.  So ${\mathfrak B}^0_0$  consists of a single element,  namely the empty word.  Note that ${\mathfrak B}^0_1$ is the empty set.

Theorem~\ref{banprop} and
Theorem~\ref{bij} below are all that is needed to establish the
recurrence relation (\ref{rr}), which we have shown is equivalent to
Theorem~\ref{introsymgenth}.

\begin{thm} \label{bij} For all $n \ge 2 $, there is a bijection
$$\gamma:  {\mathfrak B}^0_{n} \to \biguplus_{0 \le m\le n-2}
{\mathfrak B}^0_{m} \times \mathfrak M_{n-m},$$ such that  if
$\gamma(B) = (B^\prime,(\omega,b)) $ then \bq \label{wteq} \wt(B) =
\wt(B^\prime)\wt(\omega)\eq and \bq \label{bbbeq} \bbar(B) = \bbar(B^\prime) + b,\eq
where $\bbar(B)$ denotes the number of bars of $B$. \end{thm}

We will make use of another type of factorization of a
word over an ordered alphabet
used by D\'esarm\'enien and Wachs \cite{dw}.   For any alphabet $A$, let
$A^+$ denote the set of words over $A$ of positive length.

\begin{defn}[\cite{dw}]  An {\em increasing} factorization of a word $w$ of positive length over a
totally ordered alphabet $A$ is a factorization
$w=w_1 \cdot w_2\cdots  w_k$ such that
\begin{enumerate}
\item each $w_i$ is of the form $a_i^{j_i} u_i$, where $a_i \in A$, $j_i > 0$ and $$u_i \in \{x \in A : x < a_i\}^+$$
\item $a_1 \le a_2 \le \dots \le a_k$.
\end{enumerate}
\end{defn}

For example, $87\cdot 8866\cdot 995587 \cdot 95$ is an increasing factorization of
the word $w:= 87886699558795$ over the totally ordered alphabet of positive integers.
 Note that this factorization is different from the Lyndon factorization of $w$, which is   $87\cdot 8866\cdot 99558795$

 \begin{prop}[{\cite[Lemmas 3.1 and  4.3 ]{dw}}]  \label{deswa} A word over an ordered alphabet admits an increasing factorization if and only if its Lyndon type has no parts of size 1.   Moreover,  the increasing factorization is unique.
 \end{prop}

\begin{proof}[Proof of Theorem~\ref{bij}]     Given a banner $B$ in $ {\mathfrak B}^0_{n} $, take its unique increasing factorization $$B=B_1 \cdot B_2 \cdots B_k,$$ whose existence is guaranteed  by Proposition~\ref{deswa}.    We will extract an increasing word from $B_k$.
We have
$$B_k= a^p i_1\cdots i_l,$$
where $p,l \ge 1$, and $a > i_1,i_2,\dots,i_l.$  It follows from the definition of banner that $a$ is a barred letter.  Determine the unique index $r$ such that $i_1\ge \dots \ge i_{r-1}$ are barred, while $i_{r} $ is unbarred.     Then let $s $ be  the unique index that satisfies  $r \le s \le l$ and either
 \begin{enumerate}  \item   $i_r \le i_{r+1} \le \dots \le i_{s}$ are all unbarred and less than $i_{r-1}$ (note that $i_s$ can be equal to $i_{r-1}$ in absolute value), while $i_{s+1} > i_{r-1}$ if $s <l$, or
 \item $i_r \le i_{r+1} \le \dots \le i_{s-1}$  are all unbarred and less  than  $i_{r-1}$, and $i_s$ is barred and  less than or equal to $i_{r-1}$.
 \end{enumerate}

  {\bf Case 1: } $s= l$.  In this case (1) must hold since the last letter of a banner is unbarred.  Let $\omega$ be the weakly increasing rearrangement of $B_k$ with bars removed and let
\bq \label{case1} B^\prime = B_1 \cdot B_2 \cdots B_{k-1}.\eq
  To see that $B^\prime$ is a banner, one need only note that the last letter of $B_{k-1}$ is unbarred (as is the last letter of each $B_i$).  Since (\ref{case1}) is an increasing factorization of $B^\prime$, it follows from Proposition~\ref{deswa} that the Lyndon type of $B^\prime$ has no parts of size 1.  Let $$\gamma(B) = (B^\prime, (\omega,b)),$$ where $b$ is the number of bars of $B_k$.   Clearly $p \le b < p+l $; so $(\omega,b) $
  is a marked sequence for which (\ref{wteq}) and (\ref{bbbeq}) hold.
  For example, if
  $$B=\bar 2\bar 2\bar 2 1 \cdot \bar 5 22 \bar 4 2\cdot  \bar 8\bar 8 \bar 7 \bar 5 2235.$$
  then $a=\bar 8$, $p=2$, $ r= 3$ and $s= 6 = l$.  It follows that $$(\omega, b) = (22355788, 4) \mbox{ and }
  B^\prime = \bar 2\bar 2\bar 2 1 \cdot \bar 5 22 \bar 4 2.$$

{\bf  Case 2: }  $s < l$. In this case either (1) or (2) can hold.  Let $b$ be the number  of  bars of $i_1,i_2, \dots, i_s$.  Clearly $b <s$.   If (1) holds then $r>1$ since $i_{s+1} > i_{r-1}$; so $b>0$.   If (2) holds  clearly $b > 0$.   Let $\omega$ be the increasing rearrangement of $i_1,i_2, \dots, i_s$ with bars removed,  let $B_k^\prime =
 a^p i_{s+1} \cdots i_{l}$, let
 \bq \label{case2} B^\prime =  B_1 \cdot B_2 \cdots B_{k-1} \cdot B^\prime_k ,\eq
  and let $$\gamma(B) = (B^\prime, (\omega,b)).$$  Clearly  $B^\prime $ is a banner with increasing factorization given by (\ref{case2})  and $(\omega,b) $
 is a marked sequence for which (\ref{wteq}) and (\ref{bbbeq}) hold.
  For example, if
  $$B=\bar 2\bar 2\bar 2 1 \cdot \bar 5 22 \bar 4 2\cdot  \bar 8\bar 8 \bar 7 \bar 5 2235\bar 6 24$$ then $a=\bar 8$, $p=2$, $ r= 3$, and $s = 6 < l$.  Hence
  $$(\omega, b) = (223557, 2) \mbox{ and }
  B^\prime = \bar 2\bar 2\bar 2 1 \cdot \bar 5 22 \bar 4 2 \cdot \bar 8 \bar 8 \bar 6 24.$$
   If
  $$B=\bar 2\bar 2\bar 2 1 \cdot \bar 5 22 \bar 4 2\cdot  \bar 8\bar 8 \bar 7 \bar 5 223\bar 5 4\bar 6 24$$ then $a=\bar 8$, $p=2$, $ r= 3$, and $s = 6 < l$.  Hence
  $$(\omega, b) = (223557, 3) \mbox{ and }
  B^\prime = \bar 2\bar 2\bar 2 1 \cdot \bar 5 22 \bar 4 2 \cdot \bar 8 \bar 8 4\bar 6 24.$$

 In order to prove that the map $\gamma$  is a bijection, we describe its inverse.
 Let  $ (B^\prime, (\omega,b) ) \in {\mathfrak B}^0_{m} \times \mathfrak M_{n-m}$, where
$0 \le m\le n-2$.
Let
  $$B^\prime = B_1 \cdot B_2 \cdots B_{k-1}$$
  be the unique increasing factorization of $B^\prime$,  whose existence is guaranteed by Proposition~\ref{deswa}, and
  let $a$ be the largest letter of $B_{k-1}$.  We also let $\omega_i$ denote the $i$th letter of $\omega$.

 {\bf Case 1: } $|a| \le \omega_{n-m}$.  Let
 $$B_{k} = \bar{\omega}_{n-m} \cdots \bar{\omega}_{n-m-b+1} \omega_1 \cdots \omega_{n-m-b}.$$
 Clearly $B_{k}$ is a  banner with exactly $b$ bars that are placed on a rearrangement of $\omega$.  Now let
 $$B = B_1 \cdot B_2 \cdots B_{k-1} \cdot B_{k}.$$  It is easy to see that this is an increasing decomposition of a banner and that equations (\ref{wteq}) and (\ref{bbbeq}) hold.

 {\bf Case 2: } $|a| > \omega_{n-m}$.  In this case we expand the banner $B_{k-1}$ by inserting the letters of $\omega$ in the following way.  Suppose
 $$B_{k-1} = a^p j_1\cdots j_l,$$
  where $p,l \ge 1$, and $a > j_i$ for all $i$.
 If $j_1> \bar{\omega}_{n-m-b+1}$  let $$\tilde B_{k-1} = a^p  \bar{\omega}_{n-m} \cdots \bar{\omega}_{n-m-b+1} \omega_1 \cdots \omega_{n-m-b}  j_1,\dots,j_l.$$  Otherwise if $j_1\le \bar \omega_{n-m-b+1}$ let
 \beq\tilde B_{k-1} = a^p  \bar{\omega}_{n-m} \cdots \bar{\omega}_{n-m-b+2} \omega_1 \cdots\omega_{n-m-b}\ \bar{\omega}_{n-m-b+1}  j_1,\dots,j_l.\eeq
 In both cases  $\tilde B_{k-1}$ is a banner.  Now let
 $$B = B_1 \cdot B_2 \cdots  B_{k-2} \cdot \tilde B_{k-1}.$$  It is easy to see that this is an increasing decomposition of a banner and that equations (\ref{wteq}) and (\ref{bbbeq}) hold.
It is also easy to check that the map $( B^\prime, (\omega,b)) \mapsto B$ is the inverse of $\gamma$.
\end{proof}

\begin{remark} The bijection $\gamma$ when restricted to banners with distinct letters (permutations) reduces to   a
bijection that Stembridge
\cite{stem1} constructed to study the representation of the
symmetric group on the cohomology of the toric variety assoiciated
with the type A Coxeter complex (see Section \ref{repthsec}).    For words with distinct letters the notion of  decreasing decomposition coincides with the notion of Lyndon decomposition.    In Stembridge's work the Lyndon decomposition corresponds to the cycle decomposition of a permutation.  Although the term ``marked sequence'' is borrowed  from Stembridge's paper, he  defines the term differently from the way in which we do.  However there is a close connection between his marked sequences and ours.
\end{remark}

\begin{remark} In Section~\ref{othersec} we discuss a connection, pointed out to us by Richard Stanley,  between banners and   words with no adjacent repeats.  This connection can be used to provide an alternative to Step 3 in our proof of Theorem~\ref{introsymgenth}.\end{remark}

\section{Alternative formulations}
In this section we present some equivalent formulations of 
Theorem~\ref{introsymgenth} and some immediate consequences.

\begin{cor}[of Theorem~\ref{introsymgenth}]  Let $Q_n(t,r) = \sum_{j,k \ge 0} Q_{n,j,k}\, t^j \,r^k$. Then $Q_n(t,r)$ satifies the following recurrence relation:
\begin{equation} \label{altrecrel} Q_n(t,r) = r^nh_n + \sum_{k=0}^{n-2} Q_k(t,r) h_{n-k} t [n-k-1]_t. \end{equation}
\end{cor}

\begin{proof}
The recurrence relation is equivalent to
$$\sum_{k=0}^{n} Q_k(t,r) h_{n-k} t [n-k-1]_t = -r^n h_n.$$ Taking the generating function
we have
$$\left(\sum_{n\ge 0} Q_n(t,r) z^n \right ) \left( \sum_{n\ge 0}h_n t[n-1]_t z^n \right ) =- H(rz).$$  The result follows from this.
\end{proof}

The right hand side of (\ref{introsymgenth1}) is the Frobenius characteristic of a graded permutation representation that Stembridge \cite{stem1} described in terms of $\sg_n$ acting on ``marked words".
From his work we were led to the following formula, which can easily be proved by showing that the right hand side satisfies the recurrence relation (\ref{altrecrel}).

 \begin{cor} \label{formQcor} For all $n \ge 0$, \begin{equation} \label{formQ} Q_n(t,r)=\sum_{m = 0}^{\lfloor {n \over 2} \rfloor}\,\, \!\!\!\!\sum_{\scriptsize
\begin{array}{c} k_0\ge 0  \\ k_1,\dots, k_m \ge 2 \\ \sum k_i = n
\end{array}}
\!\!\!\!\!\!\! r^{k_0}h_{k_0}
\prod_{i=1}^m h_{k_i} t [k_i-1]_{t}.\end{equation}
 \end{cor}

Let $$\left[\begin{array}{c} n \\k\end{array}\right]_q = {[n]_q! \over [k]_q! [n-k]_q!}\, \mbox{  and  } \,\left[\begin{array}{c} n \\k_0,\dots,k_m\end{array}\right]_q = {[n]_q!
\over [k_0]_q![k_1]_q!\cdots [k_m]_q!}.$$ By taking the principal stable specialization of both sides of  the recurrence relation (\ref{altrecrel}) and the formula (\ref{formQ}), we have the following result.

\begin{cor} For all $n \ge 0$,
$$ A_n^{\maj,\exc, \fix}(q,t,r) = r^n + \sum_{k=0}^{n-2} \left[\begin{array}{c} n \\k\end{array}\right]_q\,\,
A_k^{\maj,\exc, \fix}(q,t,r) \,tq[n-k-1]_{tq},$$
and
$$A_n^{\maj,\exc, \fix}(q,t,r)  =   \sum_{m = 0}^{\lfloor {n \over 2} \rfloor}  \!\!\!\!\sum_{\scriptsize
\begin{array}{c} k_0\ge 0  \\ k_1,\dots, k_m \ge 2 \\ \sum k_i = n
\end{array}} \left[\begin{array}{c} n \\k_0,\dots,k_m\end{array}\right]_q\,\,
r^{k_0}
\prod_{i=1}^m tq[k_i-1]_{tq}.$$
\end{cor}

Gessel and Reutenauer  \cite{gr} and  Wachs \cite{wa1} derive a major index q-nalaog of the classical formula for the number of derangements in $\S_n$ (or more generally the number of permutations with a given number of fixed points).  As an immediate consequence of \cite[Corollary 3]{wa1}, one can obtain a $(\maj,\exc)$ generalization.  Since this generalization also follows from Corollary~\ref{expgenthfix} and is not explicitly stated in \cite{wa1},  we state and prove it here.   For all $n \in \N$, let $\mathcal D_n$ be the set of derangements in $\S_n$.

\begin{cor}[of Corollary~\ref{expgenthfix}] \label{excderang} For all $n \ge 0$, we have
\begin{equation}\label{fixform} \sum_{\scriptsize \begin{array}{c} \s \in \mathfrak S_n \\ \fix(\s) = k \end{array}} q^{\maj(\s)} t^{ \exc(\s)} =  \left[\begin{array}{c} n \\k\end{array}\right]_q \,\,\sum_{\s \in \mathcal D_{n-k}} q^{\maj(\s)} t^{ \exc(\s)} .\end{equation}
Consequently, \begin{equation}\label{derform} \sum_{\s \in \mathcal D_n} q^{\maj(\s)} t^{ \exc(\s)} = \sum_{k = 0}^n (-1)^{k} q^{k\choose 2} \left[\begin{array}{c} n \\k\end{array}\right]_q A_{n-k}^{\maj,\exc}(q,t).\end{equation}
\end{cor}

\begin{proof} 
Since the left hand side of (\ref{fixform}) equals the coefficient of $r^k\frac{z^n}{[n]_q!}$ on the left hand side of (\ref{expgeneqfix}), we have that the left hand side of (\ref{fixform}) is $\left[\begin{array}{c} n \\k\end{array}\right]_q$ times the coefficient of 
$\frac{z^{n-k}}{[n-k]_q!}$ in 
$(1-tq)/(\exp_q(ztq)-tq\exp_q(z))$.   This coefficient is precisely   $ \sum_{\s \in \mathcal D_{n-k}} q^{\maj(\s)} t^{ \exc(\s)} $.

By summing (\ref{fixform}) over all $k$ and applying Gaussian inversion we obtain (\ref{derform}).
\end{proof}

We point out that our results pertaining to major index have comajor index versions.
Recall that the {\em comajor index}  of  $\sigma \in \S_n$  is defined to be
$$\comaj(\s) := \sum_{i \in [n-1] \setminus \Des(\s)} i = {n \choose 2 } - \maj(\s).$$
 For example, we have the following comajor index version of Corollary~\ref{expgenthfix}.

\begin{cor}[of Corollary~\ref{expgenthfix}] \label{comajcor}We have
\begin{equation} \label{expgeneqfixco}
\sum_{n \geq 0}A^{\comaj,\exc,\fix}_n(q,t,r)\frac{z^n}{[n]_q!}=\frac{(1-tq^{-1})\Exp_q(rz)}{\Exp_q(ztq^{-1})-(tq^{-1})\Exp_q(z)}
\end{equation}
\end{cor}

\begin{proof} Use the facts that $[n]_{q^{-1}}! = q^{-{n\choose 2}} [n]_q!$ and $\exp_{q^{-1}}(z) = \Exp_q(z)$ to show that equations (\ref{expgeneqfix}) and (\ref{expgeneqfixco}) are equivalent.
\end{proof}

We also have the following comajor index version of Corollary~\ref{excderang}.
\begin{cor} \label{coexcderang} For all $n \ge 0$,
 we have
\begin{equation*} \sum_{\scriptsize \begin{array}{c} \s \in \mathfrak S_n \\ \fix(\s) = k \end{array}} q^{\comaj(\s)} t^{ \exc(\s)} = q^{k\choose 2}  \left[\begin{array}{c} n \\k\end{array}\right]_q \,\,\sum_{\s \in \mathcal D_{n-k}} q^{\comaj(\s)} t^{ \exc(\s)} .\end{equation*}
Consequently, $$ \sum_{\s \in \mathcal D_n} q^{\comaj(\s)} t^{ \exc(\s)}= \sum_{k = 0}^n (-1)^{k}  \left[\begin{array}{c} n \\k\end{array}\right]_q A_{n-k}^{\comaj,\exc}(q,t).$$
\end{cor}

\section{Symmetry and unimodality} \label{secQ}

\subsection{Eulerian quasisymmetric functions} \label{secQ1}
It is well known that the Eulerian numbers $a_{n,j}$ form a symmetric and unimodal sequence for each fixed $n$; i.e., $a_{n,j} = a_{n,n-1-j}$ for all $j=0,1,\dots, n-1$ and $$ a_{n,0} \le a_{n,1} \le  \dots \le a_{n, \lfloor {n-1\over 2} \rfloor} = a_{n, \lfloor {n\over 2} \rfloor} \ge \dots \ge a_{n,n-2}\ge a_{n,n-1} .$$  In this subsection we discuss symmetry and unimodality of  the coefficients of $t^j$ in  the Eulerian quasisymmetric functions and the $q$- and $(q,p)$-Eulerian polynomials.

Let  $f(t):= f_r t^r + f_{r+1} t^{r+1} + \dots + f_{s}t^{s}$ be a nonzero polynomial in $t$ whose coefficients come from a partially ordered ring $R$.  We say that  $f(t)$ is {\em  t-symmetric} with center of symmetry ${ s+r \over 2}$ if   $f_{r+i} = f_{s-i}$ for all $i =0, \dots, s-r$.  We say that  $f(t)$ is also {\em t-unimodal} if
\begin{equation}\label{unimodeq}   f_r \le_R f_{r+1} \le_R \dots \le_R f_{\lfloor {s+r\over 2} \rfloor}= f_{\lfloor {s+r+1\over 2} \rfloor}\ge_R \dots \ge _R f_{s-1} \ge_R f_s.\end{equation}

Let $\Par$ be the set of all partitions of all nonnegative
integers. By choosing a $\Q$-basis $b=\{b_\lambda:\lambda \in
\Par\}$ for the space of symmetric functions, we obtain the partial
order relation given by $f \le_b g $ if $g-f$ is $b$-positive,
where a symmetric function is said to be $b$-positive if it is a
nonnegative  linear combination of elements of the basis
$\{b_\lambda\}$. Here we are concerned with the $h$-basis,
$\{h_\lambda : \lambda \in \Par\}$, where $h_\lambda =
h_{\lambda_1} \cdots h_{\lambda_k}$ for $\lambda = (\lambda_1 \ge
\dots \ge \lambda_k)$, and the Schur basis $\{s_{\lambda}:\lambda
\in \Par\}$.
   Since $h$-positivity implies Schur-positivity, the following result also holds for the Schur basis.

\begin{thm} \label{symunimodth} Using the h-basis  to partially order the ring of symmetric functions, we have for all $n,j,k$, \begin{enumerate}
\item the Eulerian quasisymmetric functions $Q_{n,j,k}$ and $Q_{n,j}$ are h-positive symmetric functions,
 \item  the polynomial $\sum_{j=0}^{n-1}Q_{n,j,k} t^j$ is t-symmetric and t-unimodal with center of symmetry ${n-k \over 2}$,
\item   the polynomial $\sum_{j=0}^{n-1}Q_{n,j} t^j$ is t-symmetric and t-unimodal with center of symmetry ${n-1 \over 2}$.
\end{enumerate}
\end{thm}
\begin{proof}

Part (1) is a corollary of Theorem~\ref{introsymgenth} (see also Corollary~\ref{formQcor}).

Parts (2) and (3) follow from Theorem~\ref{introsymgenth} and results of Stembridge~\cite{stem1} on the symmetric function given on the right hand side of (\ref{introsymgenth1}).
For the sake of completeness, we include a proof of Parts (2) and (3)  based on Stembridge's work.

It is  well-known that the product of two symmetric unimodal polynomials in $\N[t]$ with respective centers of symmetry $c_1$ and $c_2$, is symmetric and unimodal with center of symmetry $c_1+c_2$. This result and the proof given in  \cite[Proposition 1.2]{st2} hold more generally for polynomials over an arbitrary partially ordered ring.    By Corollary~\ref{formQcor} we have $$\sum_{j=0}^{n-1}Q_{n,j,k} t^j= \sum_{m = 0}^{\lfloor {n-k \over 2} \rfloor}\,\, \!\!\!\!\sum_{\scriptsize
\begin{array}{c}   k_1,\dots, k_m \ge 2 \\ \sum k_i = n-k
\end{array}}
\!\!\!\!\!\!\! h_{k}
\prod_{i=1}^m h_{k_i} t [k_i-1]_{t}$$
 Each term  $h_{k}
\prod_{i=1}^m h_{k_i} t [k_i-1]_{t}$ is t-symmetric and t-unimodal with center of symmetry $\sum_{i \ge 0} {k_i \over 2} ={n-k \over 2}$.  Hence the sum of these terms has the same property.

With a bit more effort we show that Part (3) also follows from Corollary~\ref{formQcor}.
We have $$\sum_{j=0}^{n-1}Q_{n,j} t^j = \sum_{j=0}^{n-1} Q_{n,j,1} t^j + \sum_{j=0}^{n-1} Q_{n,j,0} t^j + \sum_{k \ge 2} \sum_{j=0}^{n-1} Q_{n,j,k} t^j .$$
By Part (2) we need only show that $$X(t) :=  \sum_{j=0}^{n-1} Q_{n,j,0} t^j + \sum_{k \ge 2} \sum_{j=0}^{n-1} Q_{n,j,k} t^j$$ is t-symmetric and t-unimodal with center of symmetry ${n-1 \over 2}$.  For any sequence of positive integers $(k_1,\dots,k_m)$,
let $$G_{k_1,\dots,k_m} := \prod_{i=1}^m h_{k_i} t [k_i-1]_{t}.$$  We have by Corollary~\ref{formQcor},
$$ \sum_{j=0}^{n-1} Q_{n,j,0} t^j  =  \sum_{\scriptsize
\begin{array}{c}  m \ge 0 \\ k_1,\dots, k_m \ge 2 \\ \sum k_i = n
\end{array}} G_{k_1,\dots,k_m} $$ and
$$\sum_{k \ge 2} \sum_{j=0}^{n-1} Q_{n,j,k} t^j  =  \sum_{\scriptsize
\begin{array}{c}  m \ge 0 \\ k_1,\dots, k_m \ge 2 \\ \sum k_i = n
\end{array}}h_{k_1} G_{k_2,\dots,k_m}. $$
Hence $$X(t) = \sum_{\scriptsize
\begin{array}{c}  m \ge 0 \\ k_1,\dots, k_m \ge 2 \\ \sum k_i = n
\end{array}} G_{k_1,\dots,k_m}+ h_{k_1} G_{k_2,\dots,k_m} . $$
We claim that $G_{k_1,\dots,k_m}+ h_{k_1} G_{k_2,\dots,k_m} $ is t-symmetric and t-unimodal with center of symmetry ${n-1 \over 2}$.  Indeed, we have
$$G_{k_1,\dots,k_m}+ h_{k_1} G_{k_2,\dots,k_m} = h_{k_1}(t[k_1 - 1]_t + 1) G_{k_2,\dots,k_m}.$$  Clearly $t[k_1 - 1]_t + 1 = 1+t+\dots t^{k_1-1}$ is t-symmetric and t-unimodal with center of symmetry ${k_1-1 \over 2}$, and $G_{k_2,\dots,k_m}$ is   t-symmetric and t-unimodal with center of symmetry ${n-k_1 \over 2}$.  Therefore our claim holds and $X(t) $ is  t-symmetric and t-unimodal with center of symmetry ${n-1 \over 2}$. \end{proof}

We now obtain  analogous results for $ A_n^{\maj,\des,\exc,\fix}$ and $A_n^{\maj,\des,\exc}$ by applying  the  principal specializations.
For the ring of polynomials $\Q[\bf q]$, where ${\bf q}$ is a list of indeterminants, we use the partial order relation: for  $f({\bf q}),g ({\bf q}) \in\Q[({\bf q})]$, define $f({\bf q}) \le_{{\bf q}} g({\bf q})$ if $g({\bf q})-f({\bf q})$ has nonnegative coefficients.

\begin{lemma}  \label{hposqpos}  If  $f$ is a Schur positive homogeneous symmetric function of degree $n$ then
 $(q;q)_{n} \Lambda( f )$ is a polynomial in $q$ with  nonnegative coefficients and $(p;q)_{n+1}\sum_{m\ge 0}  \Lambda_m ( f )p^m$ is a polynomial in $q$ and $p$ with  nonnegative coefficients.  Consequently if $f$ and $g$ are homogeneous symmetric functions of degree $n$ and  $f \le_{{\rm Schur}} g$ then $$(q;q)_{n} \Lambda (f) \le_{(q)} (q;q)_{n} \Lambda (g) $$ and $$(p;q)_{n+1} \sum_{m\ge 0} \Lambda_m( f  )p^m\le_{(q,p)} (p;q)_{n+1} \sum_{m\ge 0} \Lambda_m (g) p^m .$$
 \end{lemma}

 \begin{proof}  This follows from Lemma~\ref{desspec} and the fact that Schur functions are nonnegative  linear combinations of  fundamental quasisymmetric functions (cf. \cite[pp. 360-361]{st3}).
 \end{proof}

\begin{thm} \label{unieulerth}  For all $n,k$,  let $$A^{\maj,\des,\exc}_{n,k}(q,p,t) = \sum_{\scriptsize \begin{array}{c} \s \in \sg_n \\ \fix (\s) = k\end{array}} q^{\maj(\s)} p^{\des(\s)}t^{\exc(\s)}.$$  Then  \begin{enumerate}
\item   $A^{\maj,\des,\exc}_{n,k}(q,p,q^{-1}t)$ is $t$-symmetric with center of symmetry ${n-k \over 2}$
 \item  $A^{\maj,\des,\exc}_{n,0}(q,p,q^{-1}t)$ is t-symmetric and t-unimodal, with center of symmetry ${n \over 2}$
  \item $A^{\maj,\des,\exc}_{n,k}(q,1,q^{-1}t)$ is t-symmetric and t-unimodal, with center of symmetry ${n-k \over 2}$
\item  $A^{\maj,\des,\exc}_{n}(q,1,q^{-1}t)$ is t-symmetric and t-unimodal, with center of symmetry ${n-1 \over 2}$.
\end{enumerate}
\end{thm}

\begin{proof}
Since h-positivity implies Schur positivity, we can use Lemma~\ref{hposqpos} to specialize Theorem~ \ref{symunimodth}.  By  Lemma~\ref{nonstable}, Parts (1)  and (2) are obtained by specializing Part (2) of Theorem~ \ref{symunimodth}.
By (\ref{stablespec1}), Parts (3) and (4) are obtained by specializing Parts (2) and (3) of Theorem~ \ref{symunimodth},  respectively.
\end{proof}

   In
Section~\ref{symcyclesec} we conjecture that the $t$-symmetric polynomial $A^{\maj,\des,\exc}_{n,k}(q,p,q^{-1}t)$  is t-unimodal even when $k \ne 0$ and $p \ne 1$.     However, $A^{\maj,\des,\exc}_{n}(q,p,q^{-1}t)$ is not $t$-symmetric as can be seen in the  computation,

\beq A^{\maj,\des,\exc}_4(q,p,t) = 1 &+&  (3 p + 2pq+pq^2+2p^2q^2+ 2p^2q^3+p^2q^4) t \\ &+& (3p +pq +p^2q +3 p^2q^2 +2 p^2 q^3 +p^3q^4)  t^2\\  &+  & p t^3.
\eeq

\subsection{Cycle type Eulerian quasisymmetric functions}  \label{symcyclesec} By means of ornaments we extend the symmetry  results of Section~\ref{secQ1} to the refinements  $Q_{\lambda,j}$, and the
{\em cycle type $(q,p)$-Eulerian polynomials} defined for each partition $\lambda\vdash n$  by
$$A^{\maj,\des,\exc}_{\lambda}(q,p,t) := \sum_{\scriptsize\begin{array}{c} \s \in \sg_n \\ \lambda(\s) = \lambda  \end{array}}      q^{\maj(\s)}  p^{\des(\s)} t^{\exc(\s)} .$$

\begin{thm} \label{symquasith} For all $\lambda \vdash n$ and $j = 0,1,\dots,n-1$, the Eulerian quasisymmetric function $Q_{\lambda,j}$ is a symmetric function.
\end{thm}
\begin{proof}For each $k \in \PP$, we will construct an involution
$\psi$ on necklaces that  exchanges the number of occurrences of
the letter $k$ (barred or unbarred) and $k+1$ (barred or unbarred)
in a necklace, but preserves the number of occurrences of all
other letters and  the total number of bars. The result will then
follow from Corollary~\ref{ornth}. We start with necklaces that
contain only the letters $\{k,\bar k, k+1,\overline{ k+1}\}$. Let
$R$ be such a necklace.  We may assume without loss of generality
that $k=1$. First replace all $1$'s with $2$'s and all $2$'s with
$1$'s, leaving the bars in their original positions. The problem
is that now each $2$ that is followed by a $1$ lacks a bar (call
such a 2 {\em bad}) and each $1$ that is followed by a $2$ has a
bar (call such a 1 {\em bad}).  Since the number of bad 1's equals
the number of bad 2's, we can move all the bars from the bad 1's
to the bad 2's, thereby obtaining a necklace $R^\prime$ with the
same number of bars as $R$ but with the number of $1$'s and $2$'s
exchanged.  Let $\psi(R) = R^\prime$.  Clearly $\psi^2(R) = R$.
For example if $R= (2\bar 21\bar 1122\bar 2\bar 211111)$ then we
get $(1\bar 12\bar 2211\bar 1\bar 122222)$ before the bars are
adjusted.  After the bars are moved we have $\psi(R)= (1 12\bar
2\bar 211\bar 112222\bar 2)$.

Now we handle the  case in which $R$ has  (barred and unbarred) $k$'s and $k+1$'s, and other letters which we will call intruders.   The intruders enable us to form  linear segments of $R$ consisting only of (barred and unbarred) $k$'s and $k+1$'s.  To obtain such a linear segment start with a letter of value $k$ or $k+1$ that follows an intruder and read the letters of $R$ in a clockwise direction until another intruder is encountered.   For example if
\bq \label{symnec} R= (\bar 5 334\bar 4\bar 33 \bar 33 6\bar 6 \bar 33 3\bar 4 2 4 4)\eq and $k=3$ then the segments are $334\bar 4\bar 33 \bar 33$,  $\bar 333 \bar 4$, and  $ 4  4$.

 There are two types of segments, even segments and odd segments.  An even (odd) segment contains an even (odd) number of switches, where a switch is a letter of value $k$ followed by one of value $k+1$ or a letter of value $k+1$ followed by one of value $k$. We handle the even and odd segments differently.   In an even segment, we replace all  $k$'s with $k+1$'s and all $k+1$'s with $k$'s, leaving the bars in their original positions.  Again, at the switches, we have  bad $k$'s and bad $k+1$'s,  which are equinumerous.  So we move all the bars from the bad $k$'s to the bad $k+1$'s to obtain a good segment.   This preserves the number of bars and exchanges the number of $k$'s and $k+1$'s.  For example, the even segment $334\bar 4\bar 33 \bar 33$ gets replaced by $4 43\bar 3\bar 44 \bar 44$ before the bars are adjusted.  After the bars are moved we have
 $4 \bar 433 \bar 44 \bar 44$.

An odd segment either starts with a $k$ and ends with a $k+1$ or
vice-verse. Both cases are handled similarly.  So suppose we have
an odd segment of the form $$k^{m_1}(k+1)^{m_2}k^{m_3}(k+1)^{m_4}
\dots k^{m_{2r-1}} (k+1)^{m_{2r}},$$ where each $m_i >0$ and the
bars have been suppressed.   The number of switches is $2r-1$.  We
replace it with the odd segment
$$k^{m_2}(k+1)^{m_1}k^{m_4}(k+1)^{m_3} \dots k^{m_{2r}}(k+1)^{m_{2r-1}},$$
and put  bars  in their original positions.  Again  we may have
created bad $k$'s (but not bad $k+1$'s);  so we need to move some
bars around.  The  positions of the bad $k$'s are in the set
$\{N_1+m_2, N_2+m_4, \dots,N_{r}+m_{2r}\}$, where
$N_i=\sum_{t=1}^{2i-2} m_t$. If there is a bar on the $k$ in
position $N_i+m_{2i}$ we move it to position $N_i+m_{2i-1}$, which
had been barless. This preserves the number of bars and exchanges
the number of $k$'s and $k+1$'s. For example, the odd segment
$\bar 333 \bar 4$ gets replaced by $\bar 344 \bar 4$ before the
bars are adjusted.  After the bars are moved we have $ 34\bar 4
\bar 4$.

Let $\psi(R)$ be the necklace obtained by replacing all the segments in the way described above.
For example if $R$ is the necklace given in (\ref{symnec}) then
$$\psi(R) = (\bar 5 4 \bar 433 \bar 44 \bar 44 6\bar 6 34\bar 4 \bar 42 33).$$

It is easy to check that $\psi^2(R) = R$ for all necklaces $R$.
 Now extend the involution $\psi$ to  ornaments by applying $\psi$ to each necklace of the ornament.

\end{proof}

\begin{thm} \label{symth2}For all $\lambda \vdash n$ with exactly $k$ parts equal to 1, and $j = 0,\dots, n-k$,
  $$Q_{\lambda,j} = Q_{\lambda, n-k-j}.$$
  \end{thm}

 \begin{proof}
 We construct a type-preserving involution $\gamma$ on ornaments.  Let $R$ be an ornament of type $\lambda$.  To obtain $\gamma(R)$, first we bar each unbarred letter of each nonsingleton necklace of $R$ and unbar each barred letter.  Next for each i, we replace each occurrence of  the $i$th smallest value  in $R$ with the $i$th largest value leaving the bars intact.  Clearly $\gamma(R)$ is an ornament with $n-k-j$ bars whenever $R$ is an ornament with $j$ bars.  Also $\gamma^2(R) = R$.    The result now follows from the fact that  $Q_{\lambda,j} $ is symmetric (Theorem~\ref{symquasith}) and from Corollary~\ref{ornth}.
   \end{proof}

   \begin{remark} Although the equation
   \begin{equation} \label{symth3} Q_{n,j} = Q_{n, n-1-j}\end{equation} follows easily from Theorem~\ref{introsymgenth}, we can give a bijective proof by means of banners, using Theorem~\ref{banprop}.  We construct
 an involution $\tau$ on $\bigcup_{\lambda \vdash n}  \mathfrak B_{\lambda,j}$ that is similar to the involution $\gamma$ used in the proof of Theorem~\ref{symth2}.  Let $B$ be a banner of length $n$ with $j$ bars.  To obtain $\tau(B)$,  first we bar each unbarred letter of $B$, except for the last letter,  and unbar each barred letter.  Next for each i, we replace each occurrence of  the $i$th smallest value  in $B$ with the $i$th largest value, leaving the bars intact.  Clearly $\tau(B)$ is a banner of length $n$, with $n-1-j$ bars.
    \end{remark}

Since  $Q_{n,j,k}$ and $Q_{n,j}$ are h-positive, one might expect that the same is true for the refinement $Q_{\lambda,j}$.  It turns out that this is  not the case as the following  computation using the Maple package ACE \cite{ace} shows,
\begin{equation} \label{counterex} Q_{(6),3} = 2 h_{(4,2)} - h_{(4,1,1)} + h_{(3,2,1)} + h_{(5,1)}.\end{equation}
Therefore Theorem~\ref{symunimodth}   cannot hold  for $Q_{\lambda,n}$, as stated.
However, by expanding in the Schur basis we obtain
$$Q_{(6),3} = 3 s_{(6)} + 3 s_{(5,1)} +  3 s_{(4,2)} + s_{(3,3)} + s_{(3,2,1)},$$  which establishes Schur positivity of $Q_{(6),3}$.  We have  used    ACE  \cite{ace}  to confirm the following conjecture for $\lambda = (n)$ up to $n=9$.

\begin{con} \label{sposconj} Let $\lambda \vdash n$.  For all  $j=0,1,\dots,n-1$, the Eulerian quasisymmetric function
$Q_{\lambda,j}$ is  Schur positive.  Moreover, $Q_{\lambda,j} - Q_{\lambda,j-1}$ is Schur positive if $1\le j \le \frac{n-k}2$, where $k$ is the number of parts of $\lambda$ that are equal to $1$.  Equivalently, the t-symmetric polynomial $\sum_{j=0}^{n-1}Q_{\lambda,j}t^j$ is   t-unimodal under the  partial order relation on  the ring of symmetric functions induced by the Schur basis.  \end{con}

From Stembridge's work \cite{stem2} one obtains a nice combinatorial description of the
coefficients in the Schur function expansion of $Q_{n,j,k}$.   It would be interesting to do the same for the refinement
$Q_{\lambda,j}$,  at least when $\lambda= (n)$.  In Section~\ref{secsing} we discuss the  expansion in the power sum symmetric function basis.

By Lemmas~\ref{nonstable} and~\ref{hposqpos}, we have the following consequence of Theorem~\ref{symth2} (and refinement of Part (1) of Theorem~\ref{unieulerth}).

\begin{thm} \label{desmajsym}    Let $\lambda$ be a partition of $n$ with exactly $k$ parts of size $1$.   Then $A^{\maj,\des,\exc}_{\lambda}(q,p,q^{-1}t)$ is $t$-symmetric with center of symmetry ${n-k\over 2}$.
    \end{thm}

\begin{con} Let $\lambda$ be a partition of $n$ with exactly $k$ parts of size 1.  Then the
t-symmetric polynomial $A_{\lambda}^{\maj,\des,\exc}(q,p,q^{-1}t)$ is t-unimodal (with center of symmetry ${n-k \over 2}$).  Consequently for all $n,k$,  $A_{n,k}^{\maj,\des,\exc}(q,p,q^{-1}t)$ is $t$-symmetric and $t$-unimodal with center of symmetry ${n-k \over 2}$.
\end{con}

The next result is easy to prove  and does not rely on the $Q_{\lambda,j}$.
\begin{prop} \label{easyprop} Let $\lambda$ be a partition of $n$ with exactly $k$ parts of size $1$.   Then for all $j=0,1,\dots,n-1$, $$a_{\lambda,j}(q,p) = a_{\lambda,n-k-j}(1/q,q^np).$$
\end{prop}

\begin{proof}  This follows from the type preserving involution $\phi: \sg_n \to \sg_n$ defined by letting
$\phi(\s) $ be obtained from $\s$ by writing $\s$ in cycle form and replacing each
$i$ by $n-i+1$.   Clearly $\phi$  sends permutations with $j$ excedences to permutations with $n-k-j$ excedences.   Also $i \in \Des(\s)$ if and only if $n-i \in \Des(\phi(\s))$.
It follows that $\phi$ preserves $\des$ and  $$\maj(\phi(\s)) = \des(\s) n- \maj (\s).$$
\end{proof}

By combining Proposition~\ref{easyprop} and Theorem~\ref{desmajsym}, we obtain the following result.
\begin{cor} Let $\lambda$ be a partition of $n$ with exactly $k$ parts of size $1$.  Then for all $j=0,1,\dots,n-1$,
$$a_{\lambda,j}(q,p) = q^{2j+k-n} a_{\lambda,j}(1/q,q^np).$$ Equivalently,  the coefficient of $p^dt^j$ in $A_{\lambda}^{\maj,\des,\exc}(q,p,t)$ is $q$-symmetric with center of symmetry $j + {k+nd -n \over 2}$.\end{cor}

\begin{con} The coefficient of $p^dt^j$ in $A_{\lambda}^{\maj,\des,\exc}(q,p,t)$ is $q$-unimodal.
\end{con}

\section{Further properties of the Eulerian quasisymmetric functions} \label{repthsec}

 \subsection{Partitions with a single part} \label{secsing} The ornament characterization of $Q_{\lambda,j}$ yields  a plethystic formula expressing $Q_{\lambda,j}$ in terms of $Q_{(i),k}$.   (Note $(i)$ stands for the partition with a single part $i$.) Let $f$ be a symmetric function in variables $x_1,x_2,\dots$ and let $g$ be a formal power series with positive integer coefficients.  Choose any ordering of the monomials of $g$, where a monomial appears in the ordering $m$ times if its coefficient is $m$.  The {\em plethysm} of $f$ and $g$, denoted $f[g]$ is defined to be the the formal power series obtained from $f$ by replacing $x_i$ by the $i$th monomial of $g$, for each $i$.  The following result is an immediate consequence of  Corollary~\ref{ornth}.

\begin{cor} \label{plethform} Let $\lambda$ be a partition with $m_i$ parts of size $i$ for all  $i$.
Then \begin{equation}\label{conj}\sum_{j=0}^{|\lambda|-1} Q_{\lambda,j } t^j = \prod_{i\ge 1} h_{m_i}[ \sum_{j=0}^{i-1} Q_{(i), j} t^j].\end{equation}
Consequently,
$$\sum_{n,j \ge 0} Q_{n,j } t^j z^n  = \sum_{n\ge 0} h_n[\sum_{i,j\ge 0}Q_{(i), j} t^j z^i ] .$$
\end{cor}

By specializing (\ref{conj}) we obtain the following result. Recall that $(\lambda,\mu)$ denotes the partition of $m+n$ obtained by  concatenating  $\lambda$ and $\mu$ and then appropriately rearranging the parts.

\begin{cor} Let $\lambda$ and $\mu$ be  partitions of $m$ and $n$, respectively.  If  $\lambda$ and $\mu$ have no  common parts then
$$A^{\maj,\exc}_{(\lambda,\mu)}(q,t) = \left[ \begin{array}{c} m+n \\ m \end{array}\right]_q \,\, A^{\maj,\exc}_{\lambda}(q,t) A^{\maj,\exc}_{\mu}(q,t).$$
Consequently, if $\lambda$ has no parts equal to $1$ then for all $j$,
$$a_{(\lambda, 1^m),j}(q,1) = \left[ \begin{array}{c} m+n \\ m \end{array}\right]_q \,\, a_{\lambda,j}(q,1).$$
\end{cor}

\begin{cor} \label{pleth} For all $j$ and $k$ we have $$Q_{(2^j,1^k),j} = h_j[h_2] h_k.$$
 Moreover,
 \begin{equation}\label{invol} \sum_{\lambda,j} Q_{\lambda,j} t^j z^{|\lambda|} = \prod_{i\ge 1} (1-x_iz)^{-1} \prod_{1 \le i \le j} (1-x_ix_j t z^2)^{-1},\end{equation}
where $\lambda$ ranges over all partitions with no parts greater than $2$.
\end{cor}

By specializing (\ref{invol}) one can obtain formulas for the  generating functions of the $(\maj,\des,\exc)$ and the $(\maj, \exc)$ enumerators of involutions.  Since $\exc$ and $\fix$ determine each other for involutions, these formulas can be immediately obtained from the formulas that Gessel and Reutenauer derived for $\maj,\des,\fix$ in \cite[Theorem 7.1]{gr}.

It follows from Corollary~\ref{plethform} that in order to prove the conjecture on Schur positivity of    $Q_{\lambda,j}$ (Conjecture~\ref{sposconj}), it suffices to prove it  for all $\lambda$ with a single part because the plethysm of Schur positive symmetric functions is a Schur positive symmetric function.   Note that  if $\lambda$ has no parts greater than $2$, then by Corollary~\ref{pleth}, we conclude that  $Q_{\lambda,j}$ is Schur positive.

The Frobenius characteristic  is a fundamental isomorphism
from the ring of  virtual representations of  the symmetric groups to the ring
of homogeneous symmetric functions over the integers.  We recall the definition here.
For  a virtual representation $V$ of $\sg_n$,  let $\chi^V_\lambda$ denote the value of the  character of $V$ on the conjugacy  class of type $\lambda$.  If  $\lambda$ has $m_i$ parts of size $i$ for each $i$ define
$$ z_\lambda := 1^{m_1} m_1! 2^{m_2} m_2! \cdots n^{m_n} m_n!.$$
Let $$p_\lambda := p_{\lambda_1} \cdots p_{\lambda_k},$$ for $\lambda = (\lambda_1,\dots, \lambda_k), $ where $p_n$ is the power symmetric function $\sum_{ i \ge 0} x_i^n$.
The Frobenius characteristic of a virtual representation $V$ of $\sg_n$ is defined as follows:
$$\ch V:= {1 \over n! }\sum_{\lambda \vdash n}  z_\lambda\,\, \chi^V_\lambda\,\, p_\lambda .$$

Recall that the Frobenius characteristic of a virtual
representation of $\sg_n$ is $h$-positive if and only if it is a
permutation representation induced from  Young subgroups.  Hence
by part (1) of Theorem~\ref{symunimodth},  the Eulerian
quasisymmetric function $Q_{n,j,k}$ is the Frobenius
characteristic of a permutation representation induced from  Young
subgroups.
However (\ref{counterex})
shows that this is not the case in general for the refined
Eulerian quasisymmetric functions $Q_{\lambda,j}$. (It can be
shown that $Q_{(6),3}$ is not the Frobenius characteristic of any
permutation representation at all.) Recall also that the Frobenius
characteristic of a virtual representation is Schur positive if
and only if it is an actual representation. Hence if
$V_{\lambda,j}$ is the virtual representation  whose Frobenius
characteristic is $Q_{\lambda,j}$ then Conjecture~\ref{sposconj}
says that $V_{\lambda,j}$ is  an actual representation.

\begin{prop} \label{dimvprop} Let $\lambda \vdash n$.  Then the dimension of $V_{\lambda,j}$ equals the number of permutations of cycle type $\lambda$ with $j$ excedances.  Moreover, the dimension of $V_{(n),j}$ is the Eulerian number $a_{n-1,j-1}$.
\end{prop}

\begin{proof} The dimension of $V_{\lambda,j}$ is the coefficient of $x_1x_2 \cdots x_n$ in $Q_{\lambda,j}$.  Using the definition of $Q_{\lambda,j}$, this is the number of permutations of
cycle type $\lambda$ with $j$ excedances.

The set of  $n$-cycles with $j$ excedances maps bijectively to the set $\{\s \in \sg_{n-1} : \des(\s) = j-1\}$.  Indeed, the bijection is obtained by writing the cycle in the form $(c_1,\dots,c_{n-1},n)$ and then extracting the word $c_{n-1}\cdots c_1$.  Hence the number of $n$-cycles with $j$ excedances is the Eulerian number $a_{n-1,j-1}$.
\end{proof}

 We have a conjecture for the character of $V_{(n),j}$, which generalizes Proposition~\ref{dimvprop}.  We have confirmed our conjecture  up to $n=8$ using the Maple package ACE \cite{ace}.  Equivalently, our conjecture gives the coefficients in the expansion of  $Q_{(n),j} $ in the basis of power symmetric functions and implies that  $Q_{(n),j} $ is $p$-positive.  For a polynomial $F(t)=a_0 + a_1 t + ... + a_k t^k$ and a positive integer
$m$, define $F(t)_m$ to be the polynomial obtained from $F(t)$ by erasing all terms
$a_i t^i$ such that $\gcd(m,i) \neq 1$.  For example, if $F(t)=1+t+2t^2+3t^3$
then $F(t)_2=t+3t^3$.
For a partition $\lambda= (\lambda_1,\lambda_2, \dots, \lambda_k)$, define $$g(\lambda): =\gcd(\lambda_1,...,\lambda_k)$$.

\begin{con} For $\lambda = (\lambda_1,\dots,\lambda_k)\vdash n$,
let  $$G_\lambda(t) := \left ( t A_{k-1}(t) \prod_{i=1}^{k}
[\lambda_i]_t \right)_{g(\lambda)}.$$  Then $$\sum_{j=0}^{n-1}
Q_{(n),j} t^j = {1 \over n! }\sum_{\lambda \vdash n}  z_\lambda
G_\lambda(t) p_\lambda. $$ Equivalently, the character of
$V_{(n),j}$ evaluated on conjugacy class $\lambda$ is the
coefficient of $t^j$ in $G_{\lambda}(t)$. \label{cvalcon}
\end{con}

Conjecture \ref{cvalcon} holds whenever $\lambda_k=1$, see
Corollary \ref{cvalcor}. It follows from this  that the
character of $V_{(n),j}$ evaluated at $\lambda = 1^n$ is the
Eulerian number $a_{n-1,j-1}$. This special case of the conjecture
is also a consequence of  Proposition~\ref{dimvprop} because any character
evaluated at $1^n$ is the dimension of the representation. In the
tables below we give the character values for $V_{(n),j}$ at
conjugacy class $\lambda$, which were computed using  ACE
\cite{ace} and confirm the conjecture up to $n=8$.  The columns
are indexed by $n,j$ and the rows by the partitions $\lambda$.

{\small 
$$
\begin{array}{|c||c|c|}
\hline & 4,1 & 4,2   \\ \hline\hline 4 & 1 & 0
\\ \hline 31 & 1& 1 \\ \hline 2^2 & 1 & 0
\\ \hline 21^2 & 1 & 2
\\ \hline 1^4 & 1 & 4 \\ \hline
\end{array}\hspace{.3in} \begin{array}{|c||c|c|}
\hline & 5,1 & 5,2   \\ \hline\hline 5 & 1 & 1
\\ \hline 41 & 1& 1 \\ \hline 32 & 1 & 2
\\ \hline 31^2 & 1 & 2
\\ \hline 2^2 1 & 1 & 3
\\ \hline 2 1^3 & 1 & 5
\\ \hline 1^5 & 1 & 11 \\ \hline
\end{array}\hspace{.3in}
\begin{array}{|c||c|c|c|}
\hline & 6,1 & 6,2  & 6,3 \\ \hline\hline 6 & 1 & 0 & 0
\\ \hline 51 & 1& 1 & 1
\\ \hline 42 & 1 & 0 & 2
\\ \hline 3^2 & 1 & 2 & 0
\\ \hline 41^2 & 1 & 2 & 2
\\ \hline 321 & 1 & 3 & 4
\\ \hline 2^3 & 1 & 0 & 6
\\ \hline 3 1^3& 1 & 5 & 6
\\ \hline 2^21^2 & 1 & 6 & 10
\\ \hline 2  1^4  &1 & 12 &  22
\\ \hline 1^6& 1 & 26 & 66
\\ \hline
\end{array}$$

\vspace{.1in}$$\begin{array}{|c||c|c|c|}
\hline & 7,1 & 7,2  & 7,3
\\ \hline\hline 7 & 1 & 1 & 1
\\ \hline 61 & 1& 1 & 1
\\ \hline 52 & 1 & 2 & 2
\\ \hline 43 & 1 & 2 & 3
\\ \hline 51^2 & 1 & 2 & 2
\\ \hline 421 & 1 & 3 & 4
\\ \hline 3^21 & 1 & 3 & 5
\\ \hline 32^2 & 1 & 4 & 7
\\ \hline 4 1^3& 1 & 5 & 6
\\ \hline 321^2 & 1 & 6 & 11
\\ \hline 2^31 & 1 & 7 & 16
\\ \hline 3  1^4  &1 & 12 &  23
\\ \hline 2^21^3 & 1 & 13 & 34
\\ \hline 21^5 & 1 & 27 & 92
\\ \hline 1^7& 1 & 57 & 302
\\ \hline
\end{array} \hspace{.3in}
\begin{array}{|c||c|c|c|c|}
\hline & 8,1 & 8,2  & 8,3 & 8,4
\\ \hline\hline 8& 1 & 0 & 1 & 0
\\ \hline 71 & 1& 1 & 1 &1
\\ \hline 62 & 1 & 0 & 2 & 0
\\ \hline 53 & 1 & 2 & 3 & 3
\\ \hline 4^2 & 1 & 0 & 3 & 0
\\ \hline 61^2 & 1 & 2 & 2 & 2
\\ \hline 521 & 1 & 3 & 4 & 4
\\ \hline 431 & 1 & 3 & 5 & 6
\\ \hline 42^2 & 1 & 0 & 7 & 0
\\ \hline 3^22 & 1 & 4 & 8 & 10
\\ \hline 5 1^3& 1 & 5 & 6 & 6
\\ \hline 421^2 & 1 & 6 & 11 & 12
\\ \hline 3^2 1^2 & 1 & 6 & 12 & 16
\\ \hline 32^2 1 & 1 & 7 & 17 & 22
\\ \hline 2^4  &1 & 0 &  23 & 0
\\ \hline 41^4  &1 & 12 &  23 & 24
\\ \hline 321^3 & 1 & 13 & 35 & 46
\\ \hline 2^31^2 & 1 & 14 & 47 & 68
\\ \hline 31^5 & 1 & 27 & 93 & 118
\\ \hline 2^21^4 & 1 & 28 & 119 & 184
\\ \hline 21^6 & 1 & 58 & 359 & 604
\\ \hline 1^8& 1 & 120 &1191& 2416
\\ \hline
\end{array}
$$}

Conjecture \ref{cvalcon} resembles the following immediate consequence of Theorem~\ref{introsymgenth} and Stembridge's   computation  \cite[Proposition 3.3]{stem1} of the coefficients  in the expansion of the $r=1$ case of the right hand side of (\ref{introsymgenth1}), in the basis of power symmetric functions.

\begin{prop}\label{stem}  We have 
$$\sum_{j=0}^{n-1} Q_{n,j} t^j = \frac {1}{n!}\sum_{\lambda \vdash n} z_{\lambda} 
\left(  A_{\ell(\lambda)}(t) \prod_{i=1}^{\ell(\lambda)} [\lambda_i]_t\right) p_\lambda,$$
where $\ell(\lambda)$ denotes the length of the partition $\lambda$ and $\lambda_i$ denotes its $i$th part. 
\end{prop}

By using banners we are able to prove the following curious fact
about the representation $V_{(n),j}$.

\begin{thm} \label{resthm}  For all
$j=0,\dots, n-1$,  the restriction of  $V_{(n),j}$ to $\sg_{n-1}$
is  the permutation representation whose Frobenius characteristic
is $Q_{n-1,j-1}$.
\end{thm}

\begin{proof}
Given a homogeneous symmetric function $f$ of degree $n$, let $\tilde f$ be the polynomial obtained from $f$ by setting  $x_i =0$  for all $i > n$.   Since $f$ is symmetric, $\tilde f$ determines $f$.

We use the fact that
for any virtual representation $V$ of $\sg_n$, \begin{equation} \label{vcheq} \widetilde {\ch \,V \!\!\downarrow}_{\sg_{n-1} }= {\partial \over \partial x_n} \widetilde{ \ch V }\,|_{x_n=0}.\end{equation}

By Theorem~\ref{banprop},  $\widetilde{Q_{(n),j}}$ is the sum of weights of Lyndon banners of length $n$, with $j$ bars, whose letters have value  at most $n$.  The partial derivative with respect to $x_n$ of  the weight of a such a Lyndon banner $B$ is $0$  unless  $n$ or $\bar n$ appears in $B$.  Since $B$  is Lyndon, $\bar n$ must be its first letter.  If the partial derivative is not $0$ after setting $x_n=0$ then all the other letters of $B$  must be less in absolute value than $n$.  In this case,  the partial derivative is the weight of the banner $B^\prime$ obtained from $B$ by removing its first letter $\bar n$.
We thus have
\bq \label{partialeq}  {\partial \over \partial x_n} \widetilde{ \ch V_{(n),j}} |_{x_n=0} = \sum_{B^\prime} \wt(B^\prime),\eq
 where $B^\prime$ ranges over the set of all banners obtained by removing the first letter from a Lyndon banner of length $n$ with $j$ bars, whose first letter is $\bar n$ and whose other letters have value strictly less than $n$.     Clearly this is the set of all banners of length $n-1$, with $j-1$ bars, whose letters have value at most $n-1$.
  Thus the sum  on the right hand side of (\ref{partialeq}) is precisely $\widetilde {Q_{n-1,j-1}}$.  It therefore follows from (\ref{vcheq}) that
$$\widetilde {\ch \,V_{(n),j} \!\!\downarrow}_{\sg_{n-1} } = \widetilde {Q_{n-1,j-1}},$$ which implies
$$ {\ch \,V_{(n),j} \!\!\downarrow}_{\sg_{n-1} } =  {Q_{n-1,j-1}},$$
 \end{proof}
 
Theorem \ref{resthm}, along with Proposition~\ref{stem}, allows us  to
prove that Conjecture \ref{cvalcon} holds when $\lambda$ has a part
of size one.

\begin{cor} \label{cvalcor}
Let $\lambda=(\lambda_1,\ldots,\lambda_k=1)$ be a partition of $n$
and let $\sigma \in \S_n$ have cycle type $\lambda$. Then the
character of $V_{(n),j}$ evaluated at $\sigma$ is the coefficient
of $t^j$ in $tA_{k-1}(t)\prod_{i=1}^{k}[\lambda_i]_t$.
\end{cor}

\begin{proof}
We may assume that $\sigma$ fixes $n$, which alllows us to think
of $\sigma$ as an element of $\S_{n-1}$.  Let $V_{n-1,j-1}$ be the
representation of $\S_{n-1}$ whose Frobenius characteristic is
$Q_{n-1,j-1}$.  By Theorem \ref{resthm}, the character value in
question is equal to the character value of $\sigma$ on
$V_{n-1,j-1}$.  Corollary
\ref{cvalcor} now follows from Proposition~\ref{stem}.
\end{proof}

\subsection{Other occurences} \label{othersec} The Eulerian quasisymmetric functions refine symmetric functions that have appeared earlier in the literature.        A multiset derangement of order $n$ is a $2 \times n$  array  of  positive integers whose top row is weakly increasing, whose bottom row is a rearrangement of the top row, and whose columns contain distinct entries.  An excedance of a multiset derangement $D = (d_{i,j})$  is a column $j$ such that $d_{1,j} < d_{2,j}$.   Given a multiset derangement $D=(d_{i,j})$, let $x^D:=\prod_{i=1}^n x_{d_{1,i}}$.   For all $j < n$, let   $\mathcal D_{n,j}$ be the set of all derangements in $\mathfrak S_n$ with $j$ excedances and let $\mathcal {MD}_{n,j}$ be the set of all multiset derangements of order $n$ with $j$ excedances. Set
$$d_{n,j}({\bf x}) :=  \sum_{D\in \mathcal {MD}_{n,j}} {\rm x}^D.$$ Askey and Ismail \cite{ai} (see also \cite{kz}) proved the following $t$-analog of MacMahon's \cite[Sec. III, Ch. III]{mac1}
result on multiset derangements  \begin{equation}\label{macderang}  \sum_{j,n \ge 0} d_{n,j}({\bf x})t^j z^n= { 1\over 1 - \sum_{i \ge 2} t[i-1]_t e_i z^i}.\end{equation}
\begin{cor}[to Theorem~\ref{introsymgenth}]  For all $n,j \ge 0 $ we have
\begin{equation} \label{derange1} d_{n,j}({\bf x}) = \omega Q_{n,j,0} = \sum_{\sigma \in \mathcal D_{n,j}} F_{[n-1]\setminus \Exd(\s) ,n},\end{equation} where $\omega$ is the standard involution on the ring of symmetric functions, which takes $h_n$ to $e_n$. Consequenty,
\begin{equation}\label{derangeq1}d_{n,j}(1,q,q^2,\dots) = (q;q)_{n}^{-1} \sum_{\sigma \in \mathcal D_{n,j}} q^{\comaj(\s)+j },
\end{equation} and
\begin{equation}\label{derangeq2} \sum_{m\ge 0} p^m \Lambda_m d_{n,j}({\bf x}) = (p;q)_{n+1}^{-1}\sum_{\sigma \in \mathcal D_{n,j}} q^{\comaj(\s)+j} p^{n-\des(\s)+1}.
\end{equation}
\end{cor}

\begin{proof} The right hand side of (\ref{macderang}) can be obtained by applying $\omega$ to the right hand side of  (\ref{introsymgenth2}) and setting  $r=0$.  Hence the first equation of (\ref{derange1}) holds.

The involution  on the ring $\mathcal Q$ of quasisymmetric funcitons defined  by
$ F_{S,n} \mapsto F_{[n-1]\setminus S,n}$ restricts to  $\omega$ on the ring of symmetric functions (cf. \cite[Exercise 7.94.a]{st3}).  Hence  the second equation of (\ref{derange1})  follows from the definition of $Q_{n,j,0}$.  Equations (\ref{derangeq1}) and (\ref{derangeq2}) now  follow from Lemmas~\ref{desspec} and~\ref{exdlem}.
\end{proof}

Let $W_n$ be the set of all words of length $n$ over alphabet $\PP$ with no adjacent repeats, i.e.,
$$W_n := \{w \in \PP^n : w(i )\ne w(i+1)\,\, \forall i = 1,2,\dots,n-1\}. $$  Define the enumerator
$$Y_n(x_1,x_2,\dots):= \sum_{w\in  W_n} {\bf x}^w,$$ where
${\bf x}^w:=x_{w(1)} \cdots x_{w(n)}$.
In \cite{csv} Carlitz, Scoville and Vaughan prove  the identity
\begin{equation}\label{carl} \sum_{n \ge 0} Y_n({\bf x}) z^n = {\sum_{i \ge 0} e_i z^i \over 1 - \sum_{i \ge 2} (i-1) e_i z^i}.\end{equation}  (See Dollhopf, Goulden and Greene  \cite{dgg} and Stanley \cite{st4} for alternative proofs.)   It was observed by Stanley  that there is a straightforward generalization of (\ref{carl}).
\begin{thm}[Stanley (personal communication)]  \label{stanth} For all $j <n$, define $$Y_{n,j}(x_1,x_2,\dots):= \sum_{\scriptsize \begin{array}{c} w\in  W_n \\ \des(w) = j\end{array}} {\bf x}^w.$$ Then
\begin{equation}\label{carlg} \sum_{n \ge 0} Y_{n,j}({\bf x})t^j z^n = {(1-t) E(z) \over E(zt) - t E(z)} ,\end{equation} where  $$E(z) = \sum_{n \ge 0} e_n z^n.$$
\end{thm}

By combining  Theorems~\ref{introsymgenth} and \ref{stanth} we conclude that
$$Q_{n,j} = \omega Y_{n,j}.$$  Stanley (personal communication)  observed that there is a combinatorial interpretation of this identity in terms of P-partition reciprocity \cite[Section 4.5]{st5}.
 Indeed, words in $W_n$ with  fixed descent set $S\subseteq [n-1]$ can be identified with strict $P$-partitions where $P$ is  the poset on $\{p_1,\dots,p_n\}$ generated by cover relations $p_i < p_{i+1}$ if $i \in S$ and $p_{i+1} < p_{i}$ if  $i \notin S$.  Banners of length $n$ in which the set of positions of barred letters equals $S$ can be identified with $P$-partitions for the same poset $P$.  It therefore follows from
P-partition reciprocity that
\begin{equation} \label{recipeq} Y_{n,j}({\bf x}) = \omega \sum_{B} \wt(B), \end{equation}  summed over all banners of length $n$ with $j$ bars.

In \cite{st4} Stanley views words in $W_n$  as proper colorings of a path $P_n$ with $n$ vertices and $Y_n$  as the chromatic symmetric function of $P_n$.  The chromatic symmetric function of a graph $G=(V,E)$ is a symmetric function analog of the chromatic polynomial $\chi_G$ of $G$.   Stanley  \cite[Theorem 4.2]{st4} also defines a symmetric function analog of   $(-1^{|V|})\chi_G(-m)$ which enumerates all pairs $(\eta,c)$ where $\eta$ is an acyclic orientation of $G$ and $c:V\to [m]$ is a coloring satisfying $c(u) \le c(v) $ if $(u,v)$ is an edge of $\eta$.   For $G=P_n$, one can see that these pairs can be identified with banners of length $n$.  Hence Stanley's reciprocity theorem for chromatic symmetric functions \cite[Theorem 4.2]{st4}   reduces to an identity that is refined by (\ref{recipeq}) when $G= P_n$.

Another interesting combinatorial interpretation of the Eulerian quasisymmetric functions comes from Gessel (personal communication).  He considers the set $U_n$ of words  of length $n$ over the alphabet $\PP$ with no double (i.e., adjacent)  descents  and no descent in the last position $n-1$;  and proves
\begin{equation} \label{ges} \sum_{n \ge 0} z^n \sum_{w \in U_n} {\bf x}^w t^{\des(w)} (1+t) ^{n-1-2\des(w)} = {(1-t) H(z) \over H(zt) - t H(z)} .\end{equation}

The symmetric function  on the right hand side of (\ref{ges}) has
also occurred  in the work of   Processi \cite{pr}, Stanley
\cite{st2}, Stembridge \cite{stem1,stem2}, and Dolgachev and Lunts
\cite{dl}.  They studied a representation of the symmetric group
on the cohomology of the toric variety $X_n$ associated with the
Coxeter complex of $\sg_n$.  (See, for example \cite{br}, for a
discussion of Coxeter complexes and \cite{fu} for an explanation
of how toric varieties are associated to polytopes.)  The action
of $\sg_n$ on the Coxeter complex determines an action on $X_n$
and thus a linear representation on the cohomology groups of
$X_n$. As $X_n$ is a complex manifold (of dimension $n-1$),
$H^{d}(X_n)=0$ whenever $d$ is odd. The action of $\sg_n$ on $X_n$
induces a representation of $\sg_n$ on the cohomology
$H^{2j}(X_n)$ for each $j = 0,\dots,n-1$. Stanley \cite{st2},
using a formula of Procesi \cite{pr}, proves that
$$\sum_{n\ge 0} \sum_{j=0}^{n-1} \ch H^{2j}(X_n)\,t^{j} z^n
= {(1-t) H(z) \over H(zt) -tH(z)}. $$  Combining this with
Theorem~\ref{introsymgenth} yields the following conclusion.

\begin{thm}\label{toricth}  For all $j = 0,1, \dots, n-1$,
$$\ch H^{2j}(X_n)=Q_{n,j}.$$
\end{thm}

In the next section we study  another occurrence of the representation whose Frobenius characteristic  is $Q_{n,j}$.

\part{Poset topology}

 \section{A conjecture of Bj\"orner and Welker on Rees products} \label{rep}

 In this section we consider   a  partially ordered set introduced
 by Bj\"orner and Welker  \cite{bw}, which originally motivated
the work in this paper.  Bj\"orner and Welker conjectured and
Jonsson \cite{jo} proved that the
absolute value of the M\"obius invariant of this poset is a
derangement number. We refine this result by showing that the
absolute value of the poset's M\"obius function evaluated at any
pair that includes the poset's bottom element,  is an Eulerian
number.  We prove an equivariant version and a $q$-analog of both
the Bj\"orner-Welker-Jonsson result and  our refinement that
involves the Eulerian quasisymmetric functions $Q_{n,j}$  and the
$q$-Eulerian numbers $a_{n,j}(q,1)$.

All posets in this paper are finite.
Recall that for a poset $P$, the {\it order complex} $\Delta P$ is the abstract simplicial complex whose vertices are the elements of $P$ and whose $k$-simplices are totally ordered subsets of size $k+1$ from $P$.  The (reduced) homology of $P$ is given by
$\rh_k(P):= \rh_k(\Delta P;\C)$.

Given ranked posets $P,Q$ with respective rank functions $r_P,r_Q$, the {\it Rees product} $P \ast Q$ is the poset whose underlying set is
\[
\{(p,q) \in P \times Q:r_P(p) \geq r_Q(q) \},
\]
with order relation given by $(p_1,q_1) \leq (p_2,q_2)$ if and only if all of the conditions
\begin{itemize}
\item $p_1 \leq_P p_2$,
\item $q_1 \leq_Q q_2$, and
\item $r_P(p_1)-r_P(p_2) \geq r_Q(q_1)-r_Q(q_2)$
\end{itemize}
hold.  In other words, $(p_2,q_2)$ covers $(p_1,q_1)$ in $P \ast Q$ if and only if $p_2$ covers $p_1$  in $P$ and either $q_2=q_1$ or $q_2$ covers $q_1$ in $Q$.

Rees products were introduced by Bj\"orner and Welker in \cite{bw}, where they study connections between poset topology and commutative algebra.  (Rees products of affine semigroup posets arise from the ring-theoretic Rees construction on semigroup algebras.)
We will need the following result from \cite{bw}.  Recall that a poset is said to be Cohen-Macaulay if it, all its open intervals, and all its open principal (upper and lower) order ideals have vanishing  homology below the top dimension.  It is a well-known fact that Cohen-Macaulay posets are ranked.
\begin{thm}[Bj\"orner and Welker \cite{bw}] \label{bwrees} The Rees product of two Cohen-Macaulay posets is a Cohen-Macaulay poset.
\end{thm}

Before proceeding we establish some poset notation and terminology.  We say that a poset $P$ is {\em bounded} if it has a minimum element $\hat 0_P$ and a maximum element $\hat 1_P$.  For any poset $P$, let $\wh{P}$ be the
bounded poset obtained from $P$ by adding a minimum element  and a
maximum element  and let $P^+$ be the poset obtained from $P$ by adding only a maximum element. For a  poset $P$ with minimum element $\hz_P$,
let $P^-=P \setminus \{\hz_P\}$.  For $x \le y$ in $P$,  let $(x,y)$ denote the open interval $\{z \in P : x <z <y\}$  and $[x,y]$ denote the closed interval $\{z \in P : x \le z \le y\}$.

The M\"obius invariant of any bounded poset $P$  is given by
$$\mu(P) := \mu_P(\hat 0_P,\hat 1_P),$$ where $\mu_P$ is the M\"obius function on $P$.   It follows
from a  well known result of P. Hall (see \cite[Proposition 3.8.5]{st5}) and the Euler-Poincar\'e formula that if poset $P$ has length $n$ then
\begin{equation} \label{eupon} \mu(\hat P) =\sum_{i=0}^n (-1)^i \dim \tilde H_i(P).\end{equation}  Hence if
 $P$ is Cohen-Macaulay then for all $x \le y$ in $P$
\begin{equation} \label{eupon2} \mu_P(x,y) = (-1)^{r} \dim \tilde H_r((x,y)),\end{equation}
where $r= r_P(y)-r_P(x)-2$, and if $y=x$ or $y$ covers $x$ we set $\tilde H_r((x,y)) = \C$.

Let $B_n$ be the Boolean algebra on the set $[n]$ and let $C_n$ be the chain $\{1<2<\ldots<n\}$.
  Jonsson \cite{jo} uses discrete Morse theory to prove the conjecture of     Bj\"orner and Welker \cite{bw} that
\begin{equation} \label{jj}
\mu(\wh{B_n^- \ast C_n})=(-1)^nd_n,
\end{equation}
where $d_n$ is the number of derangements in $\mathfrak S_n$.

In Theorem~\ref{bncn} below we give a refinement  of (\ref{jj}).   Indeed, (\ref{jj})
follows immediately from Theorem \ref{bncn} below, the well-known formula
\begin{equation}
d_n=\sum_{m=0}^{n} (-1)^m{{n} \choose {m}}(n-m)!
\end{equation}
and the recursive definition of the M\"obius function.

\begin{thm} \label{muintth}
Let $S \subseteq [n]$ have size $m>0$.  Then for $1 \leq j \le m$ we
have
\[
\mu_{\wh{B_n^- \ast C_n}}(\hz,(S,j))=(-1)^{m+1}a_{m,j-1}
\]
\label{bncn}
\end{thm}

We will present two  different  proofs of Theorem \ref{bncn} both as consequences of   general results on  Rees products that we derive.
The first proof, which is given in  Section~\ref{treesec},
is based on the recursive definition of M\"obius function applied to the  Rees product of $B_n$ with a poset whose Hasse diagram is a  tree and the second proof, which is given in Section~\ref{elsec}, involves the theory of lexicographic shellability.  The first proof yields an $\mathfrak S_n$-equivariant
version (Theorem~\ref{bncnsn}) and a $q$-analog (Theorem~\ref{bncnq})  of Theorem \ref{bncn}.   By combining this $q$-analog with a $q$-analog of  the second proof, we obtain a new Mahonian permutation statistic whose joint distribution with $\des$ is equal to the joint distribution of $\maj$ and $\exc$.

   Let $P$ be any ranked and bounded poset of length $n$.  Note that the minimal elements of $P^- \ast C_n$ are of the form $(a,1)$ where $a$ is an atom of $P$, and the maximal elements of $P^- \ast C_n$ are of the form $(\hat 1_P, j)$ where $j \in [n]$.
For
$j \in [n]$, set $$
I_{j}(P):=\{x \in P^- \ast C_n: x< (1_P,j)\}.$$

    Suppose a  group $G$ acts on a  poset $P$ by order preserving bijections (we say that $P$ is a $G$-poset).  The group $G$ acts simplicially on $\Delta P$ and thus arises a linear representation of $G$ on each  homology group of $P$.    Now suppose $P$ is ranked of length $n$.  The given
action also determines an action of $G$  on $P \ast X$  for
any length $n$ ranked poset $X$  defined by $g(a,x)=(ga,x)$ for all
$a \in P$, $x \in X$ and $g \in G$.  For a bounded ranked $G$-poset $P$ of length $n$, the  action of $G$  on $P$ restricts to an action on $P^-$, which gives an action of $G$ on $P^- \ast C_n$.  This action restricts to an action of $G$ on each subposet $I_j(P)$.

Since $B_n^-$ and $C_n$ are Cohen-Macaulay, it follows from Theorem~\ref{bwrees} that
$B_n^- \ast C_n$ is a Cohen-Macaulay poset.  Hence by (\ref{eupon2}), Theorem~\ref{muintth} is equivalent to
\begin{equation} \label{hnin}
\dim \tilde H_{n-2}(I_{j}(B_n))=a_{n,j-1},
\end{equation}
 for each $j \in [n]$.
The symmetric group $\mathfrak S_n$ acts on $B_n$ in an obvious way and therefore on each $I_j(B_n)$.    We prove the following result in Section~\ref{treesec}.
\begin{thm} \label{bncnsn} For all $j=1,2,\dots,n$,  \begin{equation} \label{iq}
{\rm {ch}}(\rh_{n-2}(I_{j}(B_n) )\otimes {\rm {sgn}})=Q_{n,j-1}.
\end{equation}
\end{thm}

Combining this with
Theorem~\ref{toricth} yields,

\begin{cor} Let $X_n$ be the toric variety associated with the type A Coxeter complex.  For all $1\le j \le n$,
$$H^{2j}(X_n) \cong_{\mathfrak S_n} \tilde H_{n-2}(I_{j+1}(B_n)) \otimes \sgn,$$
where ${\rm {sgn}}$ is the one dimensional vector space on which $\mathfrak S_n$ acts according to the sign character.
\end{cor}

 It would be interesting to find a topological explanation for this isomorphism, in particular one that extends the isomorphism to other Coxeter groups.

Our equivariant version of the Bj\"orner-Welker-Jonsson result  involves the multiset derangement enumerator $d_{n,j}({\bf x})$ discussed in Section~\ref{othersec}.
\begin{cor}\label{equibjwe}  For all $n\ge 1$,
$$\ch(\tilde H_{n-1}(B_n^- \ast C_n)) =\sum_{j=0}^{n-1} d_{n,j}({\bf x})=\omega \sum_{j= 0}^{n-1}  Q_{n,j,0}.$$
\end{cor}

\begin{proof}  We use the following
result of Sundaram \cite{sun} (see \cite[Theorem 4.4.1]{w1}):
 If $G$ acts on a bounded poset $P$ of length $n$  then we have the virtual $G$-module isomorphism,
\begin{equation}\label{sund}  \bigoplus_{r=0}^n (-1)^{r} \bigoplus_{x \in P/G} \tilde H_{r-2}((\hat 0,x)) \uparrow_{G_x}^G \cong_G 0,\end{equation}
where  $P/G$ denotes a complete set of orbit representatives, $G_x$ denotes the stabilizer of $x$, and $\uparrow_{G_x}^G$
denotes the induction of the $G_x$ module  from $G_x$ to $G$.
   Here $H_{r-2}((\hat 0,x))$ is the trivial representation of $G_x$ if  the rank of $x$ is  $r = 0,1$.   Applying this to the Cohen-Macaulay $\sg_n$-poset  $\wh{B_n^- * C_n}$, we have
$$\tilde H_{n-1}(B_n^- * C_n) \cong_{\sg_n}\bigoplus_{m=0}^n (-1)^{n-m} \bigoplus_{j=1}^{m}  \left(\tilde H_{m-2}(I_j(B_m)) \otimes 1_{\sg_{n-m}}\right)\uparrow_{\sg_m\times\sg_{n-m}}^{\sg_n},$$
where $1_G$ denotes the trivial representation of a group $G$.  From this we obtain
\begin{equation} \label{sundeq2} \ch \tilde H_{n-1}(B_n^- * C_n)= \sum_{m=0}^n (-1)^{n-m} \sum_{j=1}^{m} \ch \tilde H_{m-2}(I_j(B_m)) h_{n-m}
\end{equation}
Hence
\begin{eqnarray*} \sum_{n \ge 0} \ch \tilde H_{n-1}(B_n^- \ast C_n)\, z^n &=& H(-z)\sum_{n\ge 0}z^n\sum_{j =1}^n \ch \tilde H_{n-2}(I_j(B_n))\\ &=& H(-z) \sum_{n \ge 0  }  z^n \sum_{j = 0}^{n-1}\omega(Q_{n,j} ).\end{eqnarray*}
Recall that $E(z) = \sum_{n\ge 0} e_n z^n$.  By the  fact that $E(z) H(-z) =1$, and (\ref{introsymgenth2}), we thus have
$$ \sum_{n \ge 0} \ch \tilde H_{n-1}(B_n^- \ast C_n)\, z^n= {1\over 1-\sum_{n\ge 2}  (n-1) e_n z^n}.$$ The result now follows from (\ref{introsymgenth2}) and  (\ref{macderang}).\end{proof}

Let $B_n(q)$ be the lattice of subspaces of the $n$-dimensional
vector space $\ff_q^n$ over the finite field $\ff_q$.   Then
$B_n(q)$ is bounded and ranked of length $n$, with the rank of a
subspace equaling its dimension.   Like $B_n$,  the $q$-analog
$B_n(q)$ is Cohen-Macaulay and therefore $B_n(q)^- \ast C_n$ is
Cohen-Macaulay.  Hence  $I_{j}(B_n(q))$ has vanishing homology
below its top dimension $n-2$. In  Section~\ref{treesec} we  prove
the following result.
\begin{thm} \label{bncnq}
For all $j=1,2,\dots,n$,
\begin{equation} \label{bncnqgen}
\dim \rh_{n-2}(I_{j}(B_n(q)))=
 \sum_{\scriptsize\begin{array}{c} \s \in\sg_n \\ \exc(\s) = j-1 \end{array}} q^{\comaj(\s)+j-1}. \end{equation}
\end{thm}

\begin{cor} \label{bncnqcor} For all $n\ge 0$, let $\mathcal D_n$ be the set of derangements in $\sg_n$.  Then
\begin{equation*} \dim \tilde H_{n-1}(B_n(q)^- \ast C_n)= \sum_{\sigma \in \mathcal D_n} q^{\comaj(\s) + \exc(\s)}.\end{equation*}
\end{cor}

\begin{proof}   Since $B_n(q)^- \ast C_n$ is
Cohen-Macaulay and the number of $m$-dimensional subspaces of $\F_q^n$ is  $ \left[\begin{array}{c} n \\m\end{array}\right]_q $, the M\"obius function recurrence for $(B_n(q)^- \ast C_n) \cup \{\hat 0, \hat 1\}$ is equivalent to 
$$ \dim \tilde H_{n-1}(B_n(q)^- \ast C_n) = \sum_{m=0}^n  \left[\begin{array}{c} n \\m\end{array}\right]_q(-1)^{n-m} \sum_{j=1}^m \dim \tilde H_{m-2}(I_j(B_m(q))).$$
It therefore follows from Theorem~\ref{bncnq} that
$$\dim \tilde H_{n-1}(B_n(q)^- \ast C_n) = \sum_{m=0}^n  \left[\begin{array}{c} n \\m\end{array}\right]_q(-1)^{n-m} \sum_{\sigma \in \mathfrak S_m} q^{\comaj(\s) + \exc(\s)}.$$
The result thus follows from Corollary~\ref{coexcderang} \end{proof}

It follows from   Theorems~\ref{bncnsn} and~\ref{bncnq} that the symmetry properties of $Q_{n,j}$ and $a_{n,j}(q,1)$ given in Theorems~\ref{symth3} and \ref{desmajsym}, respectively, are both consequences of the following general result. We suspect that there is a general Rees product result which implies the unimodality properties as well.

 \begin{prop} Let $G$ be a group and let $P$ be a ranked and bounded $G$-poset of length $n$.  Then for all $j \in [n]$ we have the following isomorphism of $G$-posets,
 $$I_j(P) \cong_{G} I_{n-j+1}(P).$$
 \end{prop}

 \begin{proof} Let $f: I_j(P)  \to  I_{n-j+1}(P)$ be the map defined by
$ f(x,i) = (x,r_P(x)+1-i)$.  It is straightforward  to check that this is  a well-defined poset isomorphism, which
commutes with the action of $G$.
 \end{proof}

\section{Rees products with trees} \label{treesec}

 For $n,t \in \PP$, let $\ttn$ be the poset whose Hasse diagram is
a complete $t$-ary tree of height $n$, with the root at the bottom.  By {\em complete}  we mean that every nonleaf node has exactly $t$ children and that all the leaves are distance $n$ from the root. The following result, which  is  interesting in its own right, will be  used to prove the results of Section~\ref{rep}.

\begin{thm}\label{treeeq123} For all $n,t\ge 1$ we have
\begin{eqnarray}\label{treeeq1}\dim \tilde H_{n-2} ((B_n * \ttn)^-) &=& tA_n(t)\\
\label{treeeq2}\dim \tilde H_{n-2} ((B_n(q) * \ttn)^-) &=&    t A^{\comaj,\exc}_n(q,qt) \\
\label{treeeq3} \ch \tilde H_{n-2} ((B_n * \ttn)^-) &=& t \sum_{j=0}^{n-1} \omega Q_{n,j}t^j.\end{eqnarray}
\end{thm}

\begin{cor}  For all $n\ge 1$ we have
\begin{eqnarray*} \dim \tilde H_{n-2} ((B_n * C_{n+1})^-) &=&n!\\
\dim  \tilde H_{n-2} ((B_n(q) * C_{n+1})^-) &=&   \sum_{\sigma \in \mathfrak S_n } q^{\comaj(\sigma) +\exc(\sigma)} \\
 \ch \tilde H_{n-2} ((B_n *C_{n+1})^-) &=&\omega \sum_{j=0}^{n-1}  Q_{n,j}.\end{eqnarray*}
\end{cor}

\subsection{Uniform posets}
 Our goal in  this subsection is to prove Theorem~\ref{treeeq123}.

A bounded ranked poset $P$ is said to be {\em uniform} if $[x,\hat 1_P] \cong [y,\hat 1_P]$ whenever $r_P(x) = r_P(y)$ (see \cite[Exercise 3.50]{st5}).  We will say that a sequence of posets $(P_0, P_1,  \dots, P_n)$ is uniform if for each $k = 0,1, \dots,n$, the poset
$P_k$ is uniform of  length $k$ and
$$  P_k \cong [x,\hat 1_{P_n}]$$ for  each $x \in P_n$ of rank $n-k$.    The sequences $(B_0,\dots,B_n)$ and
$(B_0(q),\dots,B_n(q))$ are examples of uniform sequences as are the sequences of set partition lattices $(\Pi_0,\dots, \Pi_n)$ and the sequence of face lattices  of  cross polytopes $(\wh {PCP_0},\dots,\wh {PCP_n})$.

The following result is easy to verify.
\begin{prop}  \label{uniformtree} Suppose $P$ is a uniform poset of length $n$. Then for all $t \in {\mathbb P}$,   the poset $R:=(P \ast T_{t,n})^+$ is uniform of length $n+1$.
Moreover, if $x \in P$ and $y \in R$ with $r_P(x) = r_R(y) =k$ then
 $$[y,\hat 1_R ] \cong ([x,\hat 1_P] \ast T_{t,n-k})^+.$$
\end{prop}

\begin{prop}\label{propuniform} Let $(P_0, P_1,  \dots, P_n)$ be a uniform sequence of  posets.
Then for all $t \in {\mathbb P}$,
\begin{equation} \label{uniformrec}  1+\sum_{k=0}^n W_k(P_n) [k+1]_t \mu((P_{n-k}*T_{t,n-k})^+) =0,\end{equation} where    $W_k(P)$ is the number of elements of rank $k$ in $P$.
\end{prop}

\begin{proof}
Let $R:= (P_n \ast T_{t,n})^+$ and  let $y$ have rank $k$ in $R$. By Proposition~\ref{uniformtree}, $$\mu_R(y, \hat 1_R) =\mu((P_{n-k}*T_{t,n-k})^+) . $$
Clearly $$W_k(R) = W_k(P_n) [k+1]_t$$  for all $0\le k \le n$. Hence (\ref{uniformrec}) is just the  recursive
definition of the M\"obius function applied to the dual of $R$.
\end{proof}

 To prove (\ref{treeeq1}) either take dimension in (\ref{treeeq3})  or  set $q=1$ in the proof of (\ref{treeeq2}) below.

 \begin{proof}[Proof of (\ref{treeeq2})]     We  apply Proposition~\ref{propuniform} to the uniform sequence $(B_0(q)$, $B_1(q)$, \dots, $B_n(q))$. The number of $k$-dimensional subspaces of $\F_q^n$ is given by  $$W_k(B_n(q))= \left[\begin{array}{c} n
 \\ k\end{array}\right]_q.$$  Write $\mu_n(q,t)$ for $\mu((B_n(q) \ast \ttn)^+)$. Hence by Proposition~\ref{propuniform},
  \begin{equation} \label{recbn2q}
\sum_{k=0}^n  \left[\begin{array}{c} n \\k\end{array}\right]_q [k+1]_t\mu_{n-k}(q,t)=-1.
\end{equation}

Setting
\[
F_{q,t}(z):=\sum_{j \geq 0}\mu_j(q,t)\frac{z^j}{[j]_q!}
\]
and
\[
G_{q,t}(z):=\sum_{k \geq 0}[k+1]_t\frac{z^k}{[k]_q!},
\]
we derive from (\ref{recbn2q}) that
\begin{equation} \label{gbnq}
F_{q,t}(z)=-\exp_q(z)G_{q,t}(z)^{-1}.
\end{equation}
If we assume $t >1$ we have
\begin{eqnarray*}
G_{q,t}(z) & = & \frac{1}{1-t}\sum_{k \geq 0}(1-t^{k+1})
\frac{z^k}{[k]_q!}
\\ & = & \frac{\exp_q(z)-t\exp_q(tz)}{1-t}.
\end{eqnarray*}
We calculate that
$$F_{q,t}(-z) = -(1-t) - t  \frac{(1-t)\exp_q(-tz)}{\exp_q(-z)-t\exp_q(-tz)}.$$
Using the fact that $\exp_q(-z) \Exp_q(z) = 1$, we have
$$F_{q,t}(-z) = -(1-t) - t  \frac{(1-t)\Exp_q(z)}{\Exp_q(tz)-t\Exp_q(z)}.$$
It now follows from 
Corallary~\ref{comajcor} that for all $n \ge 1$ and $t >1$,
\begin{equation}\label{mueq}\mu_n(q,t) = (-1)^{n-1} t \sum_{\sigma \in \mathfrak S_n} q^{\comaj(\sigma) + \exc(\sigma)} t^{\exc(\sigma)}.\end{equation}

One can see from (\ref{recbn2q}) and induction that  $\mu_n(q,t)$ is a polynomial in $t$.
Hence since (\ref{mueq}) holds for infinitely many integers $t$, it
holds as an identity of polynomials, which implies that it holds for $t=1$.

By Theorem~\ref{bwrees},  the poset $(B_n(q)*\ttn)^-$ is Cohen-Macaulay.  Hence (\ref{treeeq2}) holds.
  \end{proof}

  We say that a bounded ranked $G$-poset $P$ is $G$-uniform  if the following holds,
\begin{itemize}
\item  $P$ is uniform
\item $G_x \cong G_y$ for all $x,y \in P$ such that $r_P(x) = r_P(y)$
\item there is an isomorphism between $[x, \hat 1_P]$ and $[y, \hat 1_P]$ that intertwines the  actions of $G_x$ and $G_y$ for all $x,y \in P$ such that $r_P(x) = r_P(y)$.  We will write
$$ [x, \hat 1_P] \cong_{G_x,G_y}[y, \hat 1_P].$$
\end{itemize}
Given a sequence of groups $G=(G_0, G_1, \dots, G_n)$.  We say that a sequence of posets
$(P_0, P_1, \dots, P_n)$ is $G$-uniform if
\begin{itemize}
\item  $P_k$ is $G_k$-uniform of length $k$ for each $k$
\item  $G_k \cong (G_n)_x$ and $P_k \cong_{G_k,(G_n)_x} [x,\hat 1_{P_n}]$  whenever $r_{P_n}(x) = n-k$.
\end{itemize}
For example, the sequence $(B_0,\dots,B_n)$ is  $(\mathfrak S_0,\dots,\mathfrak S_n)$-uniform.

The following proposition is easy to verify.
\begin{prop}[Equivariant version of  Propsition \ref{uniformtree}] \label {guniformtree} Suppose $P$ is a $G$-uniform poset of length $n$. Then for all $t \in {\mathbb P}$,  $R:=(P \ast T_{t,n})^+$ is $G$-uniform of length $n+1$.
Moreover, if $x \in P$ and $y \in R$ with $r_P(x) = r_R(y) =k$ then
 $$[y,\hat 1_R ] \cong_{G_y,G_x} ([x,\hat 1_P] \ast T_{t,n-k})^+.$$
\end{prop}

If $(P_0, P_1,  \dots, P_n)$ is a $(G_0,G_1,\dots,G_n)$-uniform sequence of  posets, we can view
$G_k $ as a subgroup of $G_{n}$ for each $k=0,\dots,n$. For $G$-uniform poset $P$, let $ W_k(P;G)$ be the number of $G$-orbits of  the  rank $k$ elements of $P$.
 The Lefschetz character of  a $G$-poset
$P$ of length $n\ge 0$  is defined to be the virtual representation $$L(P;G) := \bigoplus_{j=0}^{n} (-1)^j \tilde H_j(P).$$
Note that by (\ref{eupon}) the dimension of the Lefschetz character $L(P;G)$  is precisely $\mu(\hat P)$.

\begin{prop}[Equivariant verision of Proposition \ref{propuniform}] \label{gpropuniform} Let $(P_0, P_1,  \dots, P_n)$ be a $(G_0,G_1,\dots,G_n)$-uniform sequence of  posets.
Then for all $t \in {\mathbb P}$,
\begin{equation} \label{guniformrec}  1_{G_n}\oplus \bigoplus_{k=0}^n W_k(P_n;G_n) [k+1]_t L((P_{n-k}*T_{t,n-k})^-;G_{n-k}) \uparrow^{G_n}_{G_{n-k}}=0.\end{equation} \end{prop}

\begin{proof}
Sundaram's equation (\ref{sund}) applied to the dual of a $G$-poset $P$ is equivalent to the following equivariant version of the recursive definition of the M\"obius function:
\begin{equation} \label{equimob} \bigoplus_{y \in P/G} L((y,\hat 1_P);G_y) \uparrow_{G_y}^G = 0, \end{equation}
where $L((y,\hat 1_P);G_y)$ is the trivial representation if $y=\hat 1_P$ and is the negative of the trivial representation if $ y$ is covered by $\hat 1_P$.  We apply (\ref{equimob}) to the $G_n$-uniform poset $R:= (P_n \ast T_{t,n})^+$.   Let $y$ have rank $k$ in $R$.
It follows from Proposition~\ref{guniformtree} that
$$L((y, \hat 1_R) ;(G_{n})_{y})\uparrow_{(G_{n})_{y}}^{G_n}\cong L((P_{n-k}*T_{t,n-k})^- ;G_{n-k})\uparrow_{G_{n-k}}^{G_n}.$$
Clearly, $$W_k(R;G_n) =  W_k(P_n;G_n) [k+1]_t $$ for all $ k $.  Thus  (\ref{guniformrec}) follows from (\ref{equimob}).
\end{proof}

   \begin{proof}[Proof of (\ref{treeeq3})] Now we apply Proposition~\ref{gpropuniform} to the  $(\mathfrak S_0, \mathfrak S_1,\dots, \mathfrak S_n)$-uniform sequence $(B_0, B_1, \dots, B_n)$. Let  $$L_n(t):= \ch \,L((B_{n}*T_{t,n})^-;\mathfrak S_{n}).$$ Clearly $W_k(B_n;\mathfrak S_n) = 1$.  Therefore by Proposition~\ref{gpropuniform},  \begin{equation} \label{grecbn2}
\sum_{k=0}^n  [k+1]_t h_k L_{n-k}(t)=-h_n.
\end{equation}

 Setting
\[
F_t(z):=\sum_{j \geq 0}L_j(t){z^j}
\]
and
\[
G_t(z):=\sum_{k \geq 0}[k+1]_t h_k {z^k}
\]
we derive from (\ref{grecbn2}) that
\begin{equation} \label{gbn}
F_t(z)G_t(z) =- H(z).
\end{equation}
Now if $t >1$,
\begin{eqnarray*}
G_t(z) & = & \frac{1}{1-t}\sum_{k \geq 0}(1-t^{k+1}) h_k
{z^k}
\\ & = & \frac{H(z)-tH(tz)}{1-t},
\end{eqnarray*}
and we thus have
\begin{equation} \label{gqbn}
F_t(z)=-\frac{(1-t)H(z)}{H(z)-tH(tz)}.
\end{equation}
We calculate that
\begin{equation} \label{gqp}
F_t(-z)=-(1-t) - t\frac{(1-t)H(-tz)}{H(-z) - t H(-tz)}.
\end{equation}
 Using the fact that $H(-z)E(z) = 1$ we have
$$F_t(-z)=-(1-t) - t\frac{(1-t)E(z)}{E(tz)-tE(z)}.$$
By applying the standard symmetric function involution $\omega$, we obtain
$$ \omega F_t(-z) = -(1-t) - t\frac{(1-t)H(z)}{H(tz)-tH(z)}.$$
It follows from this and Theorem~\ref{introsymgenth} that for all $n \ge 1$ and $t >1$,
\begin{equation} \label{gnewpr}
\omega L_n(t)=(-1)^{n-1}t \sum_{j=0}^{n-1} Q_{n,j} t^j
\end{equation}

By (\ref{grecbn2}) and induction, $L_n(t)$ is a polynomial in $t$.  Hence (\ref{gnewpr}) holds for $t=1$ as well.
Since $(B_n*\ttn)^-$ is Cohen-Macaulay we are done.
   \end{proof}

\subsection{The tree lemma}
 By the following result, since $B_n$ is self-dual and Cohen-Macaulay, Theorem~\ref{bncn} is equivalent to (\ref{treeeq1}), and   since $B_n(q)$ is self-dual and Cohen-Macaulay, Theorem~\ref{bncnq}
is equivalent to (\ref{treeeq2}).

\begin{thm}[Tree Lemma] \label{tree}
Let $P$ be a bounded, ranked poset of length $n$.  Then for all $t
\in {\mathbb P}$,
\begin{equation} \label{treeeq}
\sum_{j=1}^{n}\mu(\wh{I_j(P)}) t^j = - \mu((P^**\ttn)^+),
\end{equation}
where $P^\ast$ is the dual of $P$.
\end{thm}

Before we can prove Theorem~\ref{tree}, we need a few lemmas.
  Set
\[
R(P):=P \ast\{x_0<x_1<\ldots <x_n\}
\]
and for $i \in [n]$ let $R_i(P)$ be the closed lower order ideal in $R(P)$ generated
by $(\hat 1_P,x_i)$.
Set
\[
R_i^+(P):=\{(a,x_j) \in R_i(P):j>0\}
\]
and
\[
R_i^-(P):=R_i(P) \setminus R_i^+(P).
\]

\begin{lemma}
The posets $R_i^+(P)$ and $I_i(P)^+ $ are
isomorphic. \label{qi}
\end{lemma}

\begin{proof}
The map that sends $(a,x_j)$ to $(a,j)$ is an isomorphism.
\end{proof}

  An {\it antiisomorphism} from poset $X$ to a poset $Y$ is
an isomorphism $\psi$ from $X$ to $Y^*$.  In other words, $\psi$
is an order reversing bijection from $X$ to $Y$ with order
reversing inverse.

\begin{lemma}
For $0 \leq i \leq n$, the map $\psi_i:R_i(P) \rightarrow
R_i(P^\ast)$ given by $\psi_i((a,x_j))=(a,x_{i-j})$ is an
antiisomorphism. \label{antii}
\end{lemma}

\begin{proof}
We show first that $\psi_i$ is well-defined, that is, if $(a,x_j)
\in R_i(P)$ then $(a,x_{i-j}) \in R_i(P^\ast)$.  For $a \in P$ and
$j \in \{0,\ldots,n\}$ we have $(a,x_j) \in R_i(P)$ if and only if
the three conditions
\begin{itemize}
\item[(1)] $0 \leq j \leq i$ \item[(2)] $r_P(a) \geq j$
\item[(3)] $n-r_P(a) \geq i-j$
\end{itemize}
hold.  If (1), (2), (3) hold then so do all of
\begin{itemize}
\item[($1^\prime$)] $0 \leq i-j \leq i$ \item[($2^\prime$)]
$r_{P^\ast}(a)=n-r_P(a) \geq i-j$  \item[($3^\prime$)]
$n-r_{P^\ast}(a)=r_P(a) \geq j=i-(i-j)$,
\end{itemize}
and ($1^\prime$), ($2^\prime$), ($3^\prime$) together imply that $(a,x_{i-j}) \in
R_i(P^\ast)$.  The map $\psi^\ast_i:R_i(P^\ast)\rightarrow R_i(P)$
given by $\psi^\ast_i((a,x_j))=(a,x_{i-j})$ is also well-defined by
the argument just given, and $\psi^\ast_i=\psi_i^{-1}$, so
$\psi_i$ is a bijection.

Now for $(a,x_j)$ and $(b,x_k)$ in
$R_i(P)$, we have $(a,x_j)<(b,x_k)$ if and only if the three
conditions
\begin{itemize}
\item [(4)] $a \leq_P b$ \item[(5)] $j \leq k$  \item[(6)]
$r_P(b)-r_P(a) \geq k-j$
\end{itemize}
hold.  If (4), (5), (6) hold then so do all of
\begin{itemize}
\item[($4^\prime$)] $b \leq_{P^\ast} a$ \item[($5^\prime$)] $i-k \leq i-j$ \item[($6^\prime$)] $r_{P^\ast}(a)-r_{P^\ast}(b)=r_P(b)-r_P(a) \geq
k-j=(i-j)-(i-k)$,
\end{itemize}
and ($4^\prime$), ($5^\prime$), ($6^\prime$) together imply that in $R_i(P^\ast)$ we have
$(b,x_{i-k}) \leq (a,x_{i-j})$.  Therefore, $\psi_i$ is order
reversing, and the same argument shows that $\psi_i^\ast$ is order
reversing.
\end{proof}

\begin{cor}\label{sumcor}
 For $1 \leq i \leq n$ we have
\begin{equation} \label{sumeq}
\mu(\wh{I_i(P)})=\sum_{(a,x_i) \in
R_i(P^\ast)}\mu_{R_i(P^\ast)}((\hat 1_P,x_0),(a,x_i)).
\end{equation}
\end{cor}

In case the notation has confused the reader, we remark before
proving Corollary \ref{sumcor} that the sum on the right side of
equality (\ref{sumeq}) is taken over all pairs $(a,x_i)$ such that
$a \in P$ with $r_P(a) \leq n-i$ (so $r_{P^\ast}(a) \geq i$),
and that $\hat 1_P$, being the maximum element of $P$, is the
minimum element of $P^\ast$ (so $(\hat 1_P,x_0)$ is the minimum
element of $R_i(P^\ast)$).

\begin{proof}
We have
\begin{eqnarray*}
\mu(\wh{I_i(P)})& = & -\sum_{\alpha \in I_i(P)^+}\mu_{\wh{I_i(P)}}(\alpha,(\hat 1_P,i)) \\ & = &
-\sum_{\beta \in R_i^+(P)}\mu_{R_i^+(P)}(\beta,(\hat 1_P,x_i)) \\ &
= & \sum_{\gamma=(a,x_0) \in
R_i^-(P)}\mu_{R_i(P)}(\gamma,(\hat 1_P,x_i)) \\ & = &
\sum_{\gamma=(a,x_0) \in R_i^-(P)}
\mu_{R_i(P^\ast)}(\psi_i((\hat 1_P,x_i)),\psi_i(\gamma)) \\ & = &
\sum_{\gamma=(a,x_0) \in
R_i^-(P)}\mu_{R_i(P^\ast)}((\hat 1_P,x_0),(a,x_i)) \\ & = &
\sum_{(a,x_i) \in
R_i(P^\ast)}\mu_{R_i(P^\ast)}((\hat 1_P,x_0),(a,x_i)).
\end{eqnarray*}
Indeed, the first equality follows from the definition of the
M\"obius function; the second follows from Lemma \ref{qi}; the
third follows from the definition of the M\"obius function and the
fact that $\mu_{R_i^+(P)}$ is the restriction of $\mu_{R_i(P)}$ to
$R_i^+(P) \times R_i^+(P)$ (as $R_i^+(P)$ is an upper order ideal in
$R_i(P)$); the fourth follows from Lemma \ref{antii} and the last
two follow from the definition of $\psi_i$.
\end{proof}

\begin{proof}[Proof of Tree Lemma (Theorem~\ref{tree})] The poset $\ttn$ has  exactly $t^j$ elements of rank $j$ for each $j=0,\dots, n$.  Let $r_T$ be the rank function of $\ttn$ and let $\hat 0_T$ be the minimum element of $\ttn$.

We have \begin{eqnarray*}  \mu((P^**\ttn)^+)
 & = &
-\sum_{\alpha \in P^\ast \ast \ttn}\mu_{P^\ast \ast
\ttn}((\hat 1_P,\hat 0_T),\alpha) \\ & = & -\sum_{j=0}^{n} \sum_{\alpha
\in P^\ast_{n,t,j}}\mu_{P^\ast \ast \ttn}((\hat 1_P,\hat 0_T),\alpha),
\end{eqnarray*}
where
\[
P^\ast_{n,t,j}:=\{(a,w) \in P^\ast \ast \ttn: r_T(w)=j\}.
\]

We have
\begin{eqnarray*}
\sum_{\alpha \in P^\ast_{n,t,0}}\mu_{P^\ast \ast
\ttn}((\hat 1_P,\hat 0_T),\alpha) & = & \sum_{a \in P^\ast}\mu_{P^\ast
\ast \ttn}((\hat 1_P,\hat  0_T),(a,\hat 0_T)) \\ & = & \sum_{a \in
P^\ast}\mu_{P^\ast}(\hat 1_P,a) \\ & = & 0.
\end{eqnarray*}

Now fix $j \in [n]$.  For any $w \in \ttn$ with $r_T(w)=j$, the
interval $[\hat 0_T,w]$ in $\ttn$ is a chain of length $j$.  Therefore,
for any $(a,w) \in P^\ast_{n,t,j}$, the interval
$[(\hat 1_P,\hat 0_T),(a,w)]$  in $P^\ast \ast \ttn$ is isomorphic with
the interval $[(\hat 1_P,x_0),(a,x_j)]$ in $R_j(P^\ast)$.  For any $a \in
P^\ast$, the four conditions
\begin{itemize} \item $r_{P^\ast}(a) \geq j$, \item $(a,w) \in
P^\ast_{n,t,j}$ for some $w \in \ttn$, \item $(a,v) \in
P^\ast_{n,t,j}$ for every $v \in \ttn$ satisfying $r_T(v)=j$,
\item $(a,x_j) \in R_j(P^\ast)$
\end{itemize}
are all equivalent.  There are exactly $t^j$ elements $v \in \ttn$
 of rank $j$. It follows that
\[
\sum_{\alpha \in P^\ast_{n,t,j}}\mu_{P^\ast \ast
\ttn}((\hat 1_P,\hat 0 _T),\alpha)=t^j\sum_{(a,x_j) \in
R_j(P^\ast)}\mu_{R_j(P^\ast)}((\hat 1_P,x_0),(a,x_j)),
\]
and the Tree Lemma now follows from Corollary \ref{sumcor}.
\end{proof}

Since $B_n$ is Cohen-Macaulay and self-dual, the following result shows that Theorem \ref{bncnsn} is equivalent to (\ref{treeeq3}).

\begin{thm}[Equivariant Tree Lemma] \label{gtreelem} Let $P$ be a bounded, ranked $G$-poset of length $n$.
Then for all $t \in {\mathbb P}$,
\begin{equation} \label{eqtreeeq}
\bigoplus_{j=1}^{n}t^j L(I_j(P);G) \cong_G - L((P^**\ttn)^-;G).
\end{equation}
Consequently, if $P$ is Cohen-Macaulay then for all $t \in \PP$,
$$\bigoplus_{j=1}^n  t^j \tilde H_{n-2}(I_j(P)) \cong_G  \tilde H_{n-1}((P^**\ttn)^-).$$
\end{thm}

\begin{proof}  The proof is an equivariant version of the proof of the Tree Lemma.
In particular,  the  isomorphism of Lemma~\ref{qi} is $G$-equivariant,
as is the antiisomorphism of Lemma~\ref{antii}.

  The equivariant version of (\ref{sumeq}) is
\begin{equation} \label{equisumeq}  L(I_i(P);G) = \bigoplus_{(a,x_i) \in R_i(P^\ast)/G} L(((\hat 1_P,x_0),(a,x_i));G_a)\uparrow_{G_a}^G.\end{equation}
 To prove (\ref{equisumeq}) we let (\ref{equimob}) play the role of the recursive definition of M\"obius function in the proof of (\ref{sumeq}).

 To prove
(\ref{eqtreeeq}) we follow the proof of the Tree Lemma again letting (\ref{equimob}) play the role of the recursive definition of M\"obius function, and in the last step applying  (\ref{equisumeq}) instead of (\ref{sumeq}).
\end{proof}

\section{Lexicographical shellability and a new mahonian statistic} \label{elsec}
In this section we  show that the posets
$\wh{I_j(B_n)}$ and $\wh{I_j(B_n(q))}$ are EL-shellable and use the theory of lexicographical shellability to compute their
M\"obius invariants.  This will yield a second proof of Theorem~\ref{muintth} and
a new Mahonian permutation statistic to serve as a partner for the Eulerian statistic $\des$ in the
joint $(\maj,\exc)$-distribution.

We  recall some basic facts from the theory of lexicographic
shellability  (cf. \cite{bj, bjwa1, bjwa2, bjwa3, w1}).  Let $P$ be a  bounded poset and let $\Cov(P)$
be the set of pairs $(x,y) \in P \times P$ such that $y$ covers
$x$ in $P$.  Let $L$ be another poset and let $W$ be the set of
all finite sequences of elements of $L$.  The given ordering of
$L$ induces a lexicographic ordering $\preceq$ on $W$.
 An {\it edge
labeling} of $P$ by $L$ is a function $\lambda:\Cov(P) \rightarrow
L$.  Given such a function $\lambda$ and a saturated chain
$C=\{x_1<\ldots<x_m\}$ from $P$, we write $\lambda(C)$ for
$(\lambda(x_1,x_2),\ldots,\lambda(x_{m-1},x_m)) \in W$.  An {\it
ascent} in $C$ is any $i \in [m-1]$ satisfying
$\lambda(x_i,x_{i+1})<\lambda(x_{i+1},x_{i+2})$. We say
$\lambda$ is {\it increasing} on $C$ if each  $i \in [m-1]$ is an
ascent in $C$.  The edge labeling $\lambda$ is an {\it
EL-labeling} of $P$ if whenever $x<y$ in $P$ there is a unique
maximal chain $C$ in the interval $[x,y]$ on which $\lambda$ is
increasing and for all other maximal chains $D$ in $[x,y]$ we have
$\lambda(C) \prec \lambda(D)$.  A poset that admits an EL-labeling is said to be
EL-shellable.

The notion of EL-shellability for {\em ranked}  posets was introduced by Bj\"orner in \cite{bj}.  A more general concept called CL-shellability,  introduced  by
Bj\"orner and Wachs  in \cite{bjwa1}, also associates label sequences with maximal chains of a poset. We will not define CL-labelings here.  (Both notions were subsequently extended to all  bounded posets in \cite{bjwa3}.)     Given an EL-labeling or a CL-labeling $\lambda$  on
$P$, we call a maximal chain $C$ from $P$ {\it ascent free} if its label sequence
contains no ascent.

One of the main results in the theory of
lexicographic shellability for ranked posets is the following result.

\begin{thm}[Bj\"orner \cite{bj}, Bj\"orner and Wachs \cite{bjwa1}]
If $\lambda$ is an EL-labeling (or more generally a CL-labeling) of a bounded ranked  poset $P$ of length $n$
then $\Delta(P\setminus\{\hat 0,\hat 1\})$ is homotopy equivalent to a wedge of $c$ spheres of dimension $(n-2)$,
where $c$ is the number of ascent free maximal chains.
 Consequently
\[
\mu_P(\hz,\ho)=(-1)^nc.
\]
\label{elth}
\end{thm}

To use Theorem \ref{elth} for our purposes, we need the following
result.

\begin{lemma}
Let $P$ be a bounded ranked poset.
Let $\lambda:\Cov(P) \rightarrow L$ be an EL-labeling  of $P$.  Let $(\hz_P,1)$ denote the minimum element of
$\wh{I_j(P)}$ and let $(\hat 1_P,j)$ denote the maximum element.  Define the edge labeling $$\lambda^+:\Cov(\wh{I_j(P)})
\rightarrow L \times \{0<1\}$$ by
\[
\lambda^+((x,h),(y,i))=(\lambda(x,y),i-h).
\]
Then $\lambda^+$ is an EL-labeling of $\wh{I_j(P)}$.
\label{elrees}
\end{lemma}

\begin{proof}
Let $(w,k)<(z,l)$ in $\wh{I_j(P)}$.  Then $w<z$ in $P$ and there is a
unique maximal chain $C=\{w=x_0<\dots<x_m=z\}$ in $[w,z]$ on
which $\lambda$ is increasing.  Thus if $\lambda^+$ is increasing on
the maximal chain $D=\{(y_0,f_0)<\dots<(y_m,f_m)\}$ in the
interval $I=[(w,k),(z,l)]$ then $y_j=x_j$ for $0 \leq j \leq m$.
Moreover, if $\lambda^+(D)=((a_1,d_1),\dots,(a_m,d_m))$ and
$\lambda^+$ is increasing on $D$ then we must have $d_i=0$ for $1 \leq
i \leq m-l+k$ and $d_i=1$ for $m-l+k<i \leq m$.  Since the
sequence $(d_1,\ldots,d_m)$, along with $f_0=k$ determines $f_i$
for $1 \leq i \leq m$, it follows that there is a unique maximal
chain $D$ in $I$ on which $\lambda^+$ is increasing.

Now let
$E=\{(v_0,e_0)< \dots <(v_m,e_m)\}$ be another maximal chain in
$I$.  Assume that $(v_i,e_i)=(y_i,f_i)$ for  $1 \leq i <t$ but
$(v_t,e_t) \neq (y_t,f_t)$.  If $v_t =y_t$ then clearly $e_t \neq f_t$.  If $v_t \neq y_t$ then we must have $\lambda(v_{t-1},v_t)>
a_t=\lambda(y_{t-1},y_t)$ in $L$.  Indeed, it is a basic property of EL-labelings that if $P$ is a poset with EL-labeling $\lambda$ then for each interval $[x,y]$, if $a$ covers $x$ in the unique increasing maximal chain of  $[x,y]$ and $b$ is an atom of $[x,y]$ other than $a$, then $\lambda(x,a) < \lambda(x,b)$ (cf. \cite[Proposition 2.5]{bj}, \cite[Lemma~5.3]{bjwa3}).
  In either case if
$t \leq m- l+k$ then we have $f_t=f_{t-1}$, and if $t>m-l+k$ then
$e_t=e_{t-1}+1$. It follows that in all cases we have
$\lambda^+((v_{t-1},e_{t-1}),(v_t,e_t)) >
\lambda^+((y_{t-1},f_{t-1}),(y_t,f_t))$ in $L \times \{0<1\}$.
Thus $\lambda^+(D) \prec \lambda^+(E)$ as required.
\end{proof}

\begin{remark} The EL-labeling $\Lambda^+$ given in Lemma~\ref{elrees} can be generalized in a straightforward way to the general case in which the chain $C_n$ in the Rees product $(P\setminus \{0\})*C_n$ is replaced by an arbitrary ranked EL-shellable poset.  An analogous results holds for CL-labelings. 
\end{remark} 

 Given an EL-labeling $\lambda^+$ as in Lemma \ref{elrees}, we
need to describe its ascent free chains.
For $k = 0,\dots, n-1$, let $S_{n,k}$ be the set of sequences $(0=d_1,\dots, d_n) \in \{0,1\}^n$ such that  the number of $d_i$ equal to 1 is $k$. Given any maximal chain $D = \{(x_0,f_0) < \dots <(x_n,f_n)\}$ of $\wh{I_j(P)}$, we have that $\{\hat 0_P=x_0 <x_1 <\dots <x_n=\hat 1_P\}$ is a maximal chain of $P$ and $(f_1-f_0, f_2-f_1, \dots,f_n-f_{n-1}) \in  S_{n,j-1}$.
Conversely,
given any maximal chain $C = \{\hat 0_P = x_0 <x_1< \dots <x_n = \hat 1_P\}$  of $P$ and any $d\in  S_{n,j-1}$, there is a unique maximal chain $D = \{(x_0,f_0) < \dots <(x_n,f_n)\}$ of $\wh{I_j(P)}$ such that $f_i-f_{i-1} = d_i$ for all $i\in [n]$.  Hence the maximal chains of $\wh{I_j(P)}$ can be identified with
pairs $(C,d)$ where $C$ is a maximal chain of $P$ and $d\in S_{n,j-1}$.
We have $$\lambda^+(C,d) = ((\lambda_1(C),d_1),\dots,(\lambda_n(C),d_n)), $$
where $\lambda_i(C)$ is the $ith$ entry $\lambda(x_{i-1},x_i)$ of $\lambda(C)$.
The following result clearly holds.
\begin{prop} \label{propascfree}   The maximal chain $(C,d)$ of $\wh{I_j(P)}$ is ascent free if and only if
 \begin{equation} \label{ascfreeeq} \forall i \in [n], \, \,\,\,\lambda_i(C) <\lambda_{i+1}(C) \implies d_i =1 \mbox{ and } d_{i+1} = 0. \end{equation}
holds.
\end{prop}

We turn now to the specific case where $P=B_n$.  The labeling
$\lambda:\Cov(B_n) \rightarrow \{1<2<\ldots <n\}$ given by
$\lambda(S,T)=x$ if $T \setminus S=\{x\}$ is an EL-labeling,
and for any maximal chain $C$ from $B_n$, each $i \in [n]$ appears
exactly once in the sequence $\lambda(C)$.  We can therefore view
the sequence $\lambda(C)$ as a permutation in $\S_n$ (in one
line notation).   Hence by the discussion preceding Proposition~\ref{propascfree}, we can identify label sequences of maximal chains of $\wh{I_j(B_n)}$
with pairs $(\sigma, d)$, where $\sigma \in \mathfrak S_n$ and $d \in S_{n,j-1}$.
Since each permutation in $\mathfrak S_n$ is the label sequence of a
unique maximal chain of $B_n$, each pair $(\sigma,d)$, where $\sigma \in \mathfrak S_n$ and $d \in S_{n,j-1}$,  is the label sequence of a unique maximal chain of $\wh{I_j(B_n)}$.
It will be convenient to view these pairs as barred permutations (that is,  type B permutations).
Indeed, we identify $(\sigma, d)$ with the barred permutation $\sigma^B$ in which
$$\sigma^B(i) = \begin{cases} \sigma(i) &\mbox{ if }d_i = 0 \\ \ol{\sigma(i)} &\mbox{ if }d_i = 1 \end{cases} $$ for each $i \in [n]$.  Let $|\sigma^B|$ be the permutation obtained from $\sigma^B$ by removing the bars.   Thus Proposition~\ref{propascfree} in the case of $P=B_n$ asserts that
$\sigma^B$ is the label sequence of an ascent free maximal chain of $\wh{I_j(B_n)}$ if and only if $\sigma^B$ is a barred permutation of length $n$  that satisfies
\begin{itemize}
\item[(A)] $\sigma^B$ has exactly $j-1$ bars
\item[(B)] $\sigma^B(1)$ is not barred
\item[(C)] $|\sigma^B|(i) <|\sigma^B|(i+1) \implies \sigma^B(i) \mbox{ is barred and } \sigma^B(i+1) \mbox { is not,}$ for each $i \in [n-1]$.
\end{itemize}

  Let $\mathcal B_{n,j-1}$ be the set of all such barred permutations.
  By Theorem~\ref{elth} we have the following result.

  \begin{thm} \label{eldim} For all $j\in[n]$, $$\dim \tilde H_{n-2}(I_j(B_n)) = |\mathcal B_{n,j-1}|.$$
  \end{thm}

  \begin{remark}  One can use Theorem~\ref{bncnsn} and Gessel's formla (\ref{ges}) to obtain an alternative proof of Theorem~\ref{eldim}, which does not involve lexicographic shellability.
  \end{remark}

Below we construct a  bijection between $\mathcal B_{n,j}$ and $\{\sigma \in \mathfrak S_n: \des(\sigma) = j\}$, resulting in an alternative proof that   $\dim \tilde H_{n-2}(I_j(B_n)) $ is the Eulerian number $a_{n,j-1}$.  But first we  use a result of Simion to derive a
 $q$-analog of Theorem~\ref{eldim}.

  \begin{thm}[Simion \cite{si}] There is an EL-labeling  $\lambda$ for $B_n(q)$ such that \begin{enumerate} \item for each maximal chain $C$ of $B_n(q)$, we have $\lambda(C) \in \mathfrak S_n$ \item for each $\sigma \in \mathfrak S_n$, there are $q^{\inv(\sigma)}$ maximal chains $C$ of $B_n(q)$ with $\lambda(C) = \sigma$.
  \end{enumerate}
  \end{thm}

    It follows from (1)  of Simion's result that  for each maximal chain $(C,d)$ of $\wh{I_j(B_n(q))}$,  the label sequence $\lambda^+(C,d)$ can be viewed as a barred permutation of length $n$. Hence by Proposition~\ref{propascfree}, a maximal chain $(C,d)$   is  ascent-free if and only if
$\lambda^+(C,d) \in \mathcal B_{n,j-1}$.
By (2) of Simion's result we have that for each $\sigma^B \in \mathcal B_{n,j-1}$, the number of maximal chains of $\wh{I_j(B_n(q))}$ whose label seqence is $\sigma^B$ is $q^{\inv(|\sigma^B|)}$.  Thus by Theorem~\ref{elth}, we have the following $q$-analog of Theorem~\ref{eldim}.

   \begin{thm}\label{dimhomel} For all $j\in[n]$, $$\dim \tilde H_{n-2}(I_j(B_n(q))) = \sum_{\sigma^B\in \mathcal B_{n,j-1}}  q^{\inv(|\sigma^B|)}.$$
  \end{thm}

For our induction proofs, we need to extend the definition of $\mathcal B_{n,j}$ to barred permutations over arbitrary finite subsets $X$ of $\PP$.    A barred permutation over $X$ is a linear arrangement of the elements of $X$ with bars over some (or none) of the elements.   Let $\mathcal B_X$ be the set of barred permutations $\sigma^B$ of $X$ that satisfy (B) and (C) of the definition of $\mathcal B_{n,j}$ given above.  Let $ \mathfrak S_X$ be the set of ordinary permutations of $X$.  If $X =\emptyset$ then $\mathfrak S_X = \mathcal B_X = \{\theta\}$.   For $\sigma^B \in \mathcal B_X$, let $|\sigma^B|$ be the permutation in $\mathfrak S_X$ obtained by removing the bars from $\sigma^B$.

 Given barred permutations $\alpha \in \mathcal B_A$ and $\beta \in \mathcal B_B$, where $A$ and $B$ are disjoint, let $\alpha \cdot \beta$ denote the barred permutation in
 $\mathcal B_{A\uplus B} $ obtained by concatenating the words $\alpha$ and $\beta$.
    We  define a map  $$\varphi: \biguplus_{\scriptsize \begin{array}{c} X \subset \PP \\ |X| < \infty\end{array} } \mathcal B_X \to \biguplus_{\scriptsize \begin{array}{c} X \subset \PP \\ |X| < \infty\end{array} } \mathfrak S_X,$$ recursively by
     $$\varphi(\sigma^B) = \begin{cases} \theta &\mbox{ if } \sigma^B= \theta
     \\ m\cdot \varphi(\alpha) &\mbox{ if  }\sigma^B  = \alpha \cdot  m \\ \varphi(\alpha)\cdot m &\mbox{ if } \sigma^B= \alpha \cdot \bar m  \mbox{ and }  \alpha \ne \theta  \\
     \varphi(\beta)\cdot m \cdot \varphi(\alpha) &\mbox{ if } \sigma^B= \alpha\cdot \bar m \cdot \beta \mbox{ and }  \alpha,\beta \ne \theta,  \end{cases}
  $$
 where   $m$ is the minimum letter of $|\sigma^B|$.

 \begin{lemma} \label{bijlem} The map $\varphi$ is a well-defined bijection which satisfies
 \begin{enumerate}
 \item $\varphi(\mathcal B_X) = \mathfrak S_X$
 \item $\des(\varphi(\sigma^B) )= \bars(\sigma^B)$, where $\bars(w)$ denotes the number of barred letters of  a barred permutation $w$
 \end{enumerate} for all finite subsets $X $ of $\PP$ and $\sigma^B \in \mathcal B_X$. \end{lemma}

\begin{proof} By (C) of the definition of $\mathcal B_X$,  if letter $m $ is unbarred in the word $\sigma^B \in \mathcal B_X$ then it is the last letter of
   $\sigma^B $.  By (B) if $m$ is barred it cannot be the first letter.  Hence the four cases of the definition of $\varphi$ cover all possibilities.   It is  also clear from the definition of $\mathcal B_X$ that if $\sigma^B \in \mathcal B_X$ and $|X| >1$ then  $\alpha \in \mathcal B_A$ and $\beta \in \mathcal B_{X\setminus( {A}\cup \{m\})}$ for some subset $A\subsetneq X$.  Hence by induction on $|X|$ we have that $\varphi$ is a well-defined map that takes elements of
$\mathcal B_X$ to $\mathfrak S_X$.

To show that $\varphi$ is a bijection satisfying (1) we  construct its inverse.  Define
$$\psi: \biguplus_{\scriptsize \begin{array}{c} X \subset \PP \\ |X| < \infty\end{array} } \mathfrak S_X \to \biguplus_{\scriptsize \begin{array}{c} X \subset \PP \\ |X| < \infty\end{array} } \mathcal B_X,$$ recursively by
  $$\psi(\sigma) = \begin{cases} \theta &\mbox{ if $ \sigma= \theta$} \\ \psi(\delta) \cdot m &\mbox{ if $\sigma= m \cdot \delta$}\\ \psi(\gamma)\cdot \bar m &\mbox{ if }\sigma = \gamma \cdot m \mbox{ and }  \gamma \ne \theta
  \\ \psi(\delta)\cdot  \bar m\cdot  \psi(\gamma) &\mbox{ if }\sigma= \gamma \cdot m \cdot  \delta \mbox{ and }  \gamma,\delta \ne \theta \end{cases},
  $$
 where $m$ is the minimum element of $\sigma$.
 One can see that conditions (B) and (C) of the definition of $\mathcal B_X$ hold for $\psi(\sigma)$ whenever they hold for $\psi(\gamma) $ and $\psi(\delta) $.  Hence by induction   $\psi$ is a well defined map.  One can easily also show by induction that $\varphi$ and $\psi$ are inverses of each other.

We also prove (2) by induction on $|X|$, with the base case $|X|=0$ being trivial.  We do the fourth case and leave the others to the reader.    Clearly
$$\bars(\sigma^B) = \bars(\alpha) + 1 + \bars(\beta).$$
 Since $m$ is the smallest element of $X$ and is not the first letter of $\varphi(\sigma^B)$, we have $$\des(\varphi(\sigma^B)) = \des(\varphi(\beta)) +1 + \des(\varphi(\alpha)).$$
 By induction we conclude that (2) holds in this case.
 \end{proof}

We now describe the permutation statistic that corresponds to $\inv(| \cdot|)$ under the map $\varphi$.
For a permutation $\sigma \in \mathfrak S_X$, an {\it admissible
inversion} of $\sigma$ is a pair $(\sigma(i),\sigma(j))$ such that
\begin{itemize}
\item $1 \leq i<j \leq n$ \item $\sigma(i) >\sigma(j)$, and \item
either \begin{itemize} \item $j<n$ and $\sigma(j)<\sigma(j+1)$ or
\item there is some $k$ such that $i<k<j$ and $\sigma(k)<\sigma(j)$.
\end{itemize} \end{itemize}

We write $\rm{ai}(\sigma)$ for the number of admissible inversions of
$\sigma$.  For example, if $\sigma = 3167542$  then the admissible inversions are
$(3,1)$ and $(3,2)$.  So $\rm{ai}(\sigma) = 2$.

\begin{lemma} \label{invtoai}  For all $\sigma^B \in \mathcal B_X $, $$\inv(|\sigma^B|) = {|X| \choose 2}- \ai(\varphi(\sigma^B) )$$ \end{lemma}

\begin{proof} Our proof proceeds by induction
on $n=|X|$, the case $n=0$ being trivial.

 If $\sigma^B  = \alpha \cdot  m $ then
\begin{eqnarray*}
\ai( \varphi(\sigma^B)) & = & \ai(m\cdot \varphi(\alpha)) \\ & = &  \ai( \varphi(\alpha)) \\
& = & {{n-1} \choose {2}}-\inv(|\alpha|) \\ & = & {{n} \choose
{2}}-(\inv(|\alpha|)+n-1) \\ & = & {n \choose 2}-\inv(|\alpha \cdot m|).
\end{eqnarray*}
Indeed, the first two equalities follow immediately from the
definitions and the third follows from our inductive hypothesis.

If $\sigma^B=\alpha \cdot \bar m$ then we derive
as in the case just above that
\begin{eqnarray*}
\ai( \varphi(\sigma^B))& = &\ai(  \varphi(\alpha) \cdot m)
\\ & = &\ai(  \varphi(\alpha) )  \\
& = & {n-1 \choose 2}-\inv(|\alpha |)  \\
& = & {n \choose 2}-\inv(|\alpha \cdot \bar m|).
\end{eqnarray*}

Finally, say $\sigma^B=\alpha \cdot  \ov{m}\cdot  \beta$ with $\alpha \in
\mathcal B_A$ and $\beta \in \mathcal B_B$, where $|A|  > 0$ and $|B|  >0$.
Set
$\inv(A,B):=|\{(a,b):a \in A,b\in B,a>b\}$   It follows quickly
from the inductive hypothesis and the definitions that
\begin{eqnarray*}
\ai( \varphi(\sigma^B))& = &\ai(  \varphi(\beta)\cdot  m\cdot \varphi(\alpha) )
\\ & = &\ai( \varphi(\beta))+ |B|+\ai( \varphi(\alpha) ) +\inv(B,A)
\\& = & {{|B|}
\choose {2}}-\inv(|\beta|)+n-1-|A|
\\ & & +{{|A|} \choose
{2}}-\inv(|\alpha|)+|A||B|-\inv(A,B).
\end{eqnarray*}

Now
\[
\inv(|\alpha\cdot \ov{m}\cdot \beta|)=\inv(|\alpha|)+|A|+\inv(|\beta|) + \inv(A,B)
\]
and a straightforward calculation shows that
\[
{{|B|} \choose {2}}+n-1+{{|A|} \choose {2}}+|A||B|={{n}
\choose {2}}.
\]
\end{proof}

By Theorem~\ref{dimhomel} and Lemmas~\ref{bijlem} and \ref{invtoai} we obtain the following result.
\begin{thm} \label{dimai}For all $j \in [n]$,  $$\dim \tilde H_{n-2}(I_j(B_n(q))) = \sum_{\scriptsize\begin{array}{c}\sigma\in \mathfrak S_n\\ \des(\sigma) = j-1\end{array}}  q^{{n\choose 2}-\ai(\sigma)}.$$
\end{thm}

Now define the permutation statistic $$\aid(\sigma) := \ai(\sigma) + \des(\sigma)$$
for all $\sigma \in \mathfrak S_n$.
By combining Theorem~\ref{dimai} with Theorem~\ref{bncnq} we arrive at,

\begin{thm}\label{equidistth} For all $n\ge 0$, $$ \sum_{\sigma \in \mathfrak S_n} q^{\aid(\sigma)} t^{\des(\sigma)} =  \sum_{\sigma \in \mathfrak S_n} q^{\maj(\sigma)} t^{\exc(\sigma)}.$$
\end{thm}

A considerable amount of work in symmetric function theory and poset topology has gone into proving Theorem~\ref{equidistth}.  We pose the question of whether there is  a nice direct bijective proof.  Here we   give a simple combinatorial proof  that $\aid$ is Mahonian.

\begin{prop} Let $F_n(q) = \sum_{\sigma \in \mathfrak S_n} q^{\aid(\sigma)}$.  Then
$F_n(q)$ satisfies the following recurrence for all $n \ge 2$,
$$F_n(q) := (1+q)F_{n-1}(q) + \sum_{j=2}^{n-1}  \left[\begin{array}{c} n-1 \\j-1\end{array}\right]_q q^j F_{j-1}(q) F_{n-j} (q) .$$  Consequently $F_n(q) = [n]_q!.$
\end{prop}

\begin{proof} The terms on the right side of the recurrence $q$-count  permutations according to the position of $1$ in the permutation.  That is,$$\sum_{\scriptsize \begin{array}{c} \sigma \in \mathfrak S_n\\ \sigma(j) = 1\end{array}} q^{\aid(\sigma)} = \begin{cases} \left[\begin{array}{c} n-1 \\j-1\end{array}\right]_q  q^j F_{j-1}(q) F_{n-j} (q) &\mbox { if } j=2,\dots,n-1\\ &\\
F_{n-1}(q) &\mbox { if } j = 1 \\
q F_{n-1}(q) &\mbox { if } j=n.\end{cases}$$

It is easy to see that $[n]_q!$ also satisfies the recurrence relation.
\end{proof}

\section{Type BC-analogs} \label{bcsec} In this section we present type BC analogs (in the context of Coxeter groups) of both the 
Bj\"orner-Welker-Jonsson derangement result and its q-analog, Corollary~\ref{bncnqcor}.

A   poset $P$ with a $\hat 0_P$ is said to be a {\em simplicial poset} if $[\hat 0_P, x]$ is a Boolean
algebra for all $x \in P$.  The prototypical example of a simplicial poset is the poset of faces  of a simplicial complex.   In fact, every simplicial poset is isomorphic to the face poset of some regular CW complex (see \cite{bj2}).
The next result follows immediately from Theorem~\ref{bncn} and
the definition of the M\"obius function.  For a ranked poset $P$ of length $n$ and $r \in \{0,1,\dots, n\}$, let $W_r(P)$ be the $r${\em th} Whitney number of the second kind of $P$, that is the number elements of rank $r$ in $P$.

\begin{cor}[of Theorem~\ref{bncn}]
Let $P$ be a ranked simplicial poset of length $n$.     Then
$$
\mu(\wh{P^- \ast C_n})=\sum_{r=0}^{n}(-1)^{r-1}W_r(P)r!.
$$
\label{boopo}
\end{cor}

We think of $B_n$ as
the poset of faces of a $(n-1)$-simplex, whose barycentric
subdivision is the Coxeter complex of type $A_{n-1}$.  Then $d_n$
is the number of derangements in the action of the associated
Coxeter group $\mathfrak S_n$ on the vertices of the simplex.  Let $PCP_n$
be the poset of  {\it simplicial} (that is, proper) faces
of the $n$-dimensional crosspolytope $CP_n$ (see for example
\cite[Section 2.3]{blswz}), whose barycentric subdivision is the Coxeter complex
of type BC. The associated Weyl group, which is isomorphic to the
wreath product $\mathfrak S_n[\zz_2]$, acts by reflections on $CP_n$ and
therefore on its vertex set.  Let $d_n^{BC}$ be the number of
derangements in this action on vertices.

\begin{thm}
For all $n$, we have
\[\dim \tilde H_{n-1}({PCP^{-}_n \ast C_n}) =
d_n^{BC}.
\]
\label{typeb}
\end{thm}

\begin{proof}
It is well known and straightforward to prove by induction on $n$
that, for $0 \leq r \leq n$, the number of $(r-1)$-dimensional faces
of $CP_n$ is $2^r{{n} \choose {r}}$.  Corollary \ref{boopo} gives
\[
 \mu(\wh{PCP^{-}_n \ast C_n})=\sum_{r=0}^{n}(-1)^{r-1}2^r{{n} \choose {r}}r!.
\]  Hence since $PCP^{-}_n$ is Cohen-Macaulay, by Theorem~\ref{bwrees} we have,
$$\dim \tilde H_{n-1}({PCP^{-}_n \ast C_n}) = \sum_{r=0}^{n}(-1)^{n-r}2^r{{n} \choose {r}}r!.$$

On the other hand, we may identify the vertices of $CP_n$ with
elements of $[n] \cup [\ov{n}]$, where $ [\ov{n}] =\{\bar 1,\dots,
\bar n\}$, so that the action of the Weyl group $W \cong
\S_n[\zz_2]$ is determined by the following facts.
\begin{itemize}
\item Each element $w \in W$ can be written uniquely as $w=(\sigma,v)$ with
$\sigma \in \S_n$ and $v \in \zz_2^n$. \item Any element of the
form $(\sigma,0)$ maps $i \in [n]$ to $\sigma(i)$ and $\ov{i} \in
[\ov{n}]$ to $\ov{\sigma(i)}$. \item Any element of the form
$(1,e_i)$, where $e_i$ is the $i^{th}$ standard basis vector in
$\zz_2^n$, exchanges $i$ and $\ov{i}$, and fixes all other
vertices.
\end{itemize}
It follows that for each $S \subseteq [n]$, the pointwise
stabilizer of $S$ in $W$ is exactly the pointwise stabilizer of
$\ov{S}:=\{\ov{i}:i \in S\}$ and is isomorphic to
$\S_{n-|S|}[\zz_2]$.  Using inclusion-exclusion as is done to
calculate $d_n$, we get
\[
d_n^{BC}=\sum_{j=0}^{n}(-1)^j{{n} \choose {j}}2^{n-j}(n-j)!.
\]
\end{proof}

Muldoon and Readdy \cite{mr} have recently obtained a dual version of Theorem~\ref{typeb} in which
the Rees product of the dual of $PCP_n$ with the chain is considered.  

Next we consider a poset that can be viewed as both a $q$-analog of $PCP_n$ and a type $BC$ analog of $B_n(q)$.  Let $\langle \cdot,\cdot\rangle$ be a nondegenerate, alternating bilinear form on the vector space $\F_q^{2n}$.  A subspace $U$ of $\F_q^{2n}$ is said to be {\em totally isotropic} if $\langle u,v\rangle = 0$ for all $u,v \in U$.   Let $PCP_n(q)$ be the poset of totally isotropic subspaces of  $\F_q^{2n}$.  The order complex of $PCP_n(q)$ is the building of type $C_n$, naturally associated to a finite group of Lie type $B_n$ or $C_n$ (see for example \cite
[Chapter V]{Br}, \cite[Appendix 6]{Ro}).  Thus we have both a $q$-analog of $PCP_n$ and a type BC analog of $B_n(q)$ (since the order complex of $B_n(q)$ is the building of type $A_{n-1}$).

Clearly $PCP_n(q)$ is a lower order ideal of $B_{2n}(q)$.

\begin{prop} \label{wpcpq} The maximal elements of  $PCP_n(q)$ all have dimension $n$.  For $r=0,\dots, n$, the number of $r$-dimensional isotropic subspaces of   $\F_q^{2n}$ is given by $$W_r(PCP_n(q))=  \left[\begin{array}{c} n \\r\end{array}\right]_q (q^n+1)(q^{n-1}+1) \cdots(q^{n-r+1}+1).$$
\end{prop}

\begin{proof}

The first claim of the proposition is a well known fact (see for example \cite[Chapter 1]{Ro}).  The second claim is also a known fact, we sketch a proof here.  The number of ordered bases for any $k$-dimensional subspace of $\F_q^{2n}$ is

\[
\prod_{j=0}^{k-1}(q^k-q^j).
\]

On the other hand, we can produce an ordered basis for a $k$-dimensional totally isotropic subspace of $\F_q^{2n}$ in $k$ steps, at each step $i$ choosing $v_i \in \langle v_1,\ldots,v_{i-1}\rangle^\perp \setminus \langle v_1,\ldots,v_{i-1} \rangle$.  The number of ways to do this is

\[
\prod_{j=0}^{k-1}(q^{2n-j}-q^j),
\]

and the proof is completed by division and manipulation.
\end{proof}

It was shown by L. Solomon (see \cite{So}) that $PCP_n(q)$ is Cohen-Macaulay.  Hence by Theorem~\ref{bwrees}, so is  ${PCP_n(q)^-*C_n}$.  We will show that $\dim \tilde H_{n-1} ({PCP_n(q)^-*C_n})$ is a polynomial in $q$ with nonnegative integral coefficients and give a combinatorial interpretation of the coefficients. 
 We first need the following q-analog of Corollary~\ref{boopo}.
We  say that a poset $P$ with $\hat 0_P$ is $q$-simplicial if each interval $[\hat 0_P,x]$ is isomorphic to $B_j(q)$ for some $j$.  

\begin{cor}[of Theorem~\ref{bncnq}]  \label{qboopo} Let $P$ be a ranked $q$-simplicial poset of length $n$.  Then $$
\mu(\wh{P^- \ast C_n})=\sum_{r=0}^{n}(-1)^{r-1}W_r(P)\sum_{\s \in \S_r} q^{\comaj(\s)+\exc(\s)}.
$$
\end{cor}

\begin{thm}\label{qbth} For all $n\ge 0$, let  $d_n(q) := \sum_{\sigma \in \mathcal D_n} q^{\comaj(\sigma)+\exc(\sigma)}$.  Then \begin{equation} \label{qbeq}  \dim \tilde H_{n-1} ({PCP_n(q)^-*C_n}) = \sum_{k =0}^n \left[\begin{array}{c} n \\k\end{array}\right]_q \,\,q^{k^2}\,\,\prod_{i=k+1}^n (1+q^i)  d_{n-k}(q). \end{equation}  Consequently,
$\dim \tilde H_{n-1} ({PCP_n(q)^-*C_n})$ is a polynomial in $q$ with nonnegative integer coefficients.
\end{thm}

\begin{proof}  We have by Proposition~\ref{wpcpq}, Corollary~\ref{qboopo}, and the fact that   ${PCP_n(q)^-*C_n}$ is Cohen-Macaulay, $$\dim \tilde H_{n-1} ({PCP_n(q)^-*C_n}) = \sum_{j=0}^n (-1)^j \left[\begin{array}{c} n \\j\end{array}\right]_q\,\, \prod_{i=j+1}^n (1+q^i)\,\, a_{n-j}(q),$$ where $a_n(q):= \sum_{\sigma \in \mathfrak S_n} q^{\comaj(\sigma)+\exc(\sigma)}$.
On the other hand by Corollary~\ref{coexcderang}, the right hand side of
(\ref{qbeq}) equals
$$\sum_{k =0}^n \left[\begin{array}{c} n \\k\end{array}\right]_q \,\,q^{k^2}\,\,\prod_{i=k+1}^n (1+q^i)  \sum_{m=0}^{n-k} (-1)^m \left[\begin{array}{c} n-k \\m\end{array}\right]_q a_{n-k-m}(q)$$

\begin{eqnarray*} &=&  \sum_{j\ge 0} a_{n-j}(q) \sum_{k \ge 0} \left[\begin{array}{c} n \\k\end{array}\right]_q
 \,\,q^{k^2}\,\,\prod_{i=k+1}^n (1+q^i) (-1)^{j-k}  \left[\begin{array}{c} n-k \\j-k\end{array}\right]_q \\ &=&  \sum_{j\ge 0} a_{n-j}(q)  \left[\begin{array}{c} n \\j\end{array}\right]_q\sum_{k \ge 0} \left[\begin{array}{c} j \\k\end{array}\right]_q
 \,\,q^{k^2}\,\,\prod_{i=k+1}^n (1+q^i) (-1)^{j-k} .\end{eqnarray*}
 Thus to prove (\ref{qbeq}) we need only show that
 $$ \prod_{i=j+1}^n (1+q^i) = \sum_{k \ge 0} \left[\begin{array}{c} j \\k\end{array}\right]_q
 \,\,q^{k^2}\,\,\prod_{i=k+1}^n (1+q^i) (-1)^{k} ,$$ holds for all $n$ and $j$.
 By Gaussian inversion this is equivalent to,
 $$q^{j^2}(-1)^j \prod_{i=j+1}^n(1+q^i) = \sum_{k \ge 0} \left[\begin{array}{c} j \\k\end{array}\right]_q (-1)^{j-k} q^{j-k \choose 2} 
\prod_{i=k+1}^n (1+q^i) ,$$
which is in turn equivalent to,
\begin{equation} \label{neweq} q^{j^2}(-1)^j  = \sum_{k \ge 0} \left[\begin{array}{c} j \\k\end{array}\right]_q (-1)^{j-k} q^{j-k \choose 2} 
\prod_{i=k+1}^j (1+q^i)  .\end{equation}

To prove (\ref{neweq}) we use the q-binomial formula,
$$\prod_{i=0}^{n-1} (x+yq^i) = \sum_{k\ge 0} \left[\begin{array}{c} n \\k\end{array}\right]_q q^{k\choose 2} x^{n-k}y^k.$$  Set $y=1$ and use Gaussian inversion to obtain
$$x^n = \sum_{k\ge 0} \left[\begin{array}{c} n \\k\end{array}\right]_q (-1)^{n-k} \prod_{i=0}^{k-1} (x+q^i)$$
Now set $x=q^n$ to obtain
\begin{eqnarray*}q^{n^2} &=& \sum_{k\ge 0} \left[\begin{array}{c} n \\k\end{array}\right]_q (-1)^{n-k} \prod_{i=0}^{k-1} (q^n+q^i)\\ 
&=&  \sum_{k\ge 0} \left[\begin{array}{c} n \\k\end{array}\right]_q (-1)^{n-k} q^{k\choose 2}\prod_{i=0}^{k-1} (q^{n-i}+1).\end{eqnarray*}
\end{proof}

Using the standard identification of  elements of $\mathfrak S_n[\Z_2]$ with barred permutations, the derangements of Theorem~\ref{typeb} are the barred permutations $\sigma$ for which $\sigma(i) \ne i$ for all $i \in [n]$.  Let $\mathcal D^{BC}_n$ be the set of such barred permutations.  
 For $\sigma \in \ \mathcal D_n^{BC}$, let $\tilde \sigma$ be the word obtained by rearranging the letters of $\sigma$ so that the fixed points of $|\sigma|$, which are all barred in $\sigma$, come first in increasing order with bars intact, followed by subword of nonfixed points  of $|\sigma|$ also with bars intact.  Now let $S$ be the set of positions in which bars appear in $\tilde\sigma$.  Define the {\em bar index}, $\bnd(\s)$ of $\sigma$ to be $ \sum_{i\in S} i$.  For example if $\s =\bar 3 \bar 2 5 \bar 4 \bar 6 1 \bar 7$ then $\tilde \s = \bar 2 \bar 4 \bar 7 \bar 3 5 \bar 6 1$ and so $\bnd (\s) = 1+2+3+4+6$.

\begin{cor} $$ \dim \tilde H_{n-1} ({PCP_n(q)^-*C_n}) = \sum_{\sigma \in \mathcal D_n^{BC}} q^{\comaj(|\sigma|) + \exc(|\sigma|) + \bnd(\sigma)}$$
\end{cor}

\begin{proof}   By Corollary~\ref{coexcderang} we have,
$$ \sum_{\sigma \in \mathcal D_n^{BC}} q^{\comaj(|\sigma|) + \exc(|\sigma|)} p^{ \bnd(\sigma)}\hspace{2in}$$
\begin{eqnarray*} &=& \sum_{k=0}^n  \sum_{\scriptsize \begin{array}{c} \s \in \mathfrak S_n \\ \fix(\s) = k\end{array}} q^{\comaj(\s)+\exc(\s)} p^{k+1 \choose 2} \prod_{i=k+1}^n (1+p^i) \\&=& \sum_{k=0}^n  \left[\begin{array}{c} n \\k\end{array}\right]_q q^{k\choose 2} d_{n-k}(q) p^{k+1 \choose 2} \prod_{i=k+1}^n (1+p^i).\end{eqnarray*}
Now set $p=q$ and apply Theorem~\ref{qbth}.
\end{proof}

\section*{Acknowledgements}
The research presented here began while both authors were visiting the Mittag-Leffler Institute as participants in a combinatorics program organized by Anders Bj\"orner and Richard Stanley.  We thank the Institute for its hospitality and support.  We are also grateful to Ira Gessel for some very useful
discussions.

\end{document}